\documentclass{article}
\usepackage{natbib}
\usepackage{color}
\usepackage[T1]{fontenc}
\usepackage[utf8]{inputenc}
\usepackage{verbatim}
\usepackage{multirow}
\usepackage{amsmath}
\usepackage{amsthm}
\usepackage{amssymb}
\usepackage{adjustbox}
\usepackage[title]{appendix}
\usepackage{mwe}
\usepackage{bbold}
\usepackage{quiver}
\usepackage{textcomp}
\usepackage{siunitx}
\usepackage{xfrac}
\usepackage{bbm}
\usepackage{amsthm}
\usepackage[colorlinks=true]{hyperref}
\usepackage{graphicx}
\usepackage{float}
\usepackage{subcaption}
\usepackage{setspace}
\usepackage{tabto}
\usepackage{geometry}
\usepackage{bigfoot}

\DeclareNewFootnote{AAffil}[arabic]
\DeclareNewFootnote{ANote}[fnsymbol]

\usepackage{etoolbox}
\makeatletter
\patchcmd\maketitle{\def\@makefnmark{\rlap{\@textsuperscript{\normalfont\@thefnmark}}}}{}{}{}
\makeatother

\makeatletter
\def\thanksAAffil#1{
  \footnotemarkAAffil\protected@xdef\@thanks{\@thanks%
        \protect\footnotetextAAffil[\the \c@footnoteAAffil]{#1}}%
}
\def\thanksANote#1{%
  \footnotemarkANote%
  \protected@xdef\@thanks{\@thanks%
        \protect\footnotetextANote[\the \c@footnoteANote]{#1}}%
}
\makeatother

\newtheorem{innercustomthm}{Theorem}
\newenvironment{customthm}[1]
  {\renewcommand\theinnercustomthm{#1}\innercustomthm}
  {\endinnercustomthm}

\theoremstyle{plain}
\newtheorem{thrm}{Theorem}[section]                                          
\newtheorem{prpstn}[thrm]{Proposition}                          
\newtheorem{lmm}[thrm]{Lemma}
\newtheorem{crllr}[thrm]{Corollary}
\theoremstyle{definition}
\newtheorem{dfntn}[thrm]{Definition}

\begin{document}
\title{Persistence-based Modes Inference}
\author{Hugo Henneuse \thanksAAffil{Laboratoire de Mathématiques d'Orsay, Université Paris-Saclay, Orsay, France}$^{\text{ ,}}$\thanksAAffil{DataShape, Inria Saclay, Palaiseau, France}\\ \href{mailto:hugo.henneuse@universite-paris-saclay.fr}{hugo.henneuse@universite-paris-saclay.fr}}
\maketitle
\begin{abstract}We address the problem of estimating multiple modes of a multivariate density using persistent homology, a central tool in Topological Data Analysis. We introduce a method based on the preliminary estimation of the $H_0$-persistence diagram to infer the number of modes, their locations, and the corresponding local maxima. For broad classes of piecewise-continuous functions with geometric control on discontinuities loci, we identify a critical separation threshold between modes, also interpretable in our framework in terms of modes prominence, below which modes inference is impossible and above which our procedure achieves minimax optimal rates.\end{abstract}

\section{Introduction}
Modes are simple measures to describe central tendencies of a density and one of the most used tool to process data. Modes inference finds applications in a wide variety of statistical tasks, as highlighted in the recent survey \cite{Chacon20}. Among them, modal approaches in clustering has gathered significant attention \citep{Fukunaga, Cheng95, Comanicu02, Li07, ChazalGuibasOudotPrimoz11, Chacon12, jiang17}.

The problem of mode(s) inference dates back to Parzen \cite{Parzen62} and has received considerable attention since. The question of consistency and convergence rates of estimators has occupied a central place. The bulk of work on this question is large but has mainly been concentrated on single mode estimation. For this problem, the popular approach is to consider an estimator of the form $\hat{x}=\arg \max_{x\in[0,1]^{d}} \hat{f}(x)$ with $\hat{f}$ an estimator of the density. Usually $\hat{f}$ is a  kernel estimator of the density, following Parzen's original work. This work already provides convergence rates but for univariate densities and under strong regularity assumptions (global regularity). Subsequent efforts have been made to weaken those assumptions and extend this work to multivariate densities \citep{Chernoff64, Samantha73, Konakov, Eddy80, Romano88, Tsybakov90, DonohoLiu91, Vieu96, Mokkadem03}. Notably, in \cite{DonohoLiu91}, the authors consider a general multivariate setting and weaken the assumptions to a local assumption around the mode. They suppose that the density essentially behaves around the mode as a power function, i.e there exists $\alpha,L,C>0$ such that, for all $y$ in a neighborhood of a mode $x$ :
\begin{equation}
\label{eq: power func}
C||x-y||^{\alpha}\leq f(x)-f(y)\leq L||x-y||^{\alpha}.
\end{equation}
Under this assumption, for $\alpha=2$, they show that Parzen's method achieves minimax rates. In these works a key question is how to choose the bandwidth for kernel estimation, which often involves knowing the regularity of the density. Interestingly, \cite{KlemelaMode} proposes an adaptive estimator using Lepski's method \cite{Lepski92} to overcome this issue.

An alternative, based on histogram estimation of the density, is proposed in  \cite{ariascastro2021estimation}. Under (\ref{eq: power func}), for all $1\geq \alpha>0$, they show that their estimator also achieves minimax rates. This approach seems rather isolated, but has great adaptivity and computational properties.\\\\
Another line of work, is to consider $\hat{x}=\arg \max_{i=1,...,n} \hat{f}(X_{i})$ where $X_{1},..., X_{n}$ are the observations sampled from the density, which reduces considerably computational costs. This was initially proposed in \cite{Devroye79} with $\hat{f}$ a kernel estimator. It was proved in \cite{Abraham03, Abrham04} that this estimator essentially shares the same asymptotic behavior as Parzen's estimator, convergence rates are not affected by the maximization over finite samples. In the same direction, in \cite{DasguptaKoptufe14}, the authors propose a procedure based on KNN-estimation of the density and show, under similar assumptions as \cite{DonohoLiu91}, that their estimator achieves minimax rates once again.

For multiples modes inference, the literature is scarcer.
Mean-shift and related procedures \citep{Fukunaga, Cheng95, Comanicu02, PerpinanB, PerpinanA, Li07} are widely used to infer multiple modes. Although these approaches exhibit great practical performance, we lack a theoretical understanding about their convergence properties. In this direction, some insights are given in \cite{ariascastro16a} proving near-consistency of the mean shift procedures.

Among the previously cited works on single-mode estimation, \cite{DasguptaKoptufe14} also proposes a \( k \)-nearest neighbors-based procedure for estimating multiple modes. For this more general problem, their method is shown to be minimax under the assumption that the densities are twice differentiable around the modes, with a negative definite Hessian at each mode (which in particular implies (\ref{eq: power func}), for $\alpha=2$), and that the modes are "$r$-salient,"  essentially meaning they are separated by valleys whose depth and width are both controlled by a parameter \( r > 0 \). Furthermore, \cite{jiang17} extends this approach to capture modal sets, allowing for cases where modes are not single points but entire regions. This broader method is shown to be minimax optimal under the assumption that the densities are globally Hölder continuous, behave as power functions around the modes, and that the modal sets are \( r \)-salient. To our knowledge, no other procedure has been demonstrated to possess such broad convergence properties to date.

Here we present an alternative relying on tools from Topological Data Analysis (TDA). TDA is a field that focuses on providing descriptors of the data, using tools from (algebraic) topology. One of these descriptors, persistent homology, permits to encode, in so-called "persistence diagrams", the evolution of the topology (in the homology sense) of the super-level sets of a density. Those diagrams compactly represent birth and death times of topological features. Under proper conditions, looking at $H_{0}$-persistence diagram (representing the evolution of connected components) permits to identify local maxima (birth times in super-level sets persistence diagram). This is illustrated in Figure \ref{fig:modes-PD}.\\
The link between estimation of persistence diagrams and modes detection has already been highlighted several times, see for example \cite{Bauer14} and \cite{Genovese16}. In a slightly different context, we must also mention \cite{ChazalGuibasOudotPrimoz11}. The authors propose an algorithm, \textit{ToMATo}, combining ideas from mean-shift with persistent homology to perform modal-clustering. The two main questions in those works, that we will address here, are how to estimate persistence diagrams and how to determine if a point in these estimated diagrams is significant or not, i.e. at which distance from the diagonal points are not due to noise with high probability. Reformulated in terms of modes, it is equivalent to decide if a mode is significant or not based on its prominence. Outside the sphere of TDA, this question has also received significant attention, with a rich modal-testing literature \citep{Silverman81,Hartigan85,Minotte97,ChengHall98, HallYork01,Fisher01, Duong08, Burman2009,AmeijeirasAlonso2016}.
Note that persistence diagram alone does not permit to localize modes, only to estimates their numbers and associated local maxima. Additional information needs to be extracted in order to infer their positions.
\begin{figure}[h]
    \centering
    \includegraphics[scale=0.6]{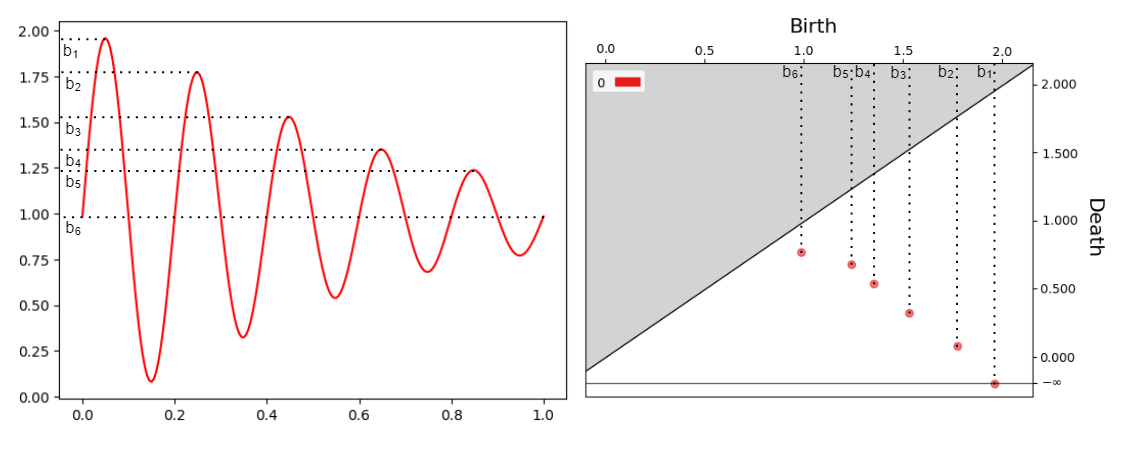}
    \caption{1D illustration of the link between local maxima and $H_{0}$ persistence diagram. The birth times $b_{1}$, ..., $b_{6}$ corresponds to the local maxima of the function.}
    \label{fig:modes-PD}
\end{figure}

The estimation of persistence diagrams of a density is an interesting question by itself. A common approach \citep[see e.g. ][]{BCL2009} in persistence diagram estimation is to use sup-norm (or other popular functional norm) stability theorem \citep{Chazal2009} to lift convergence results for density (or signal) estimation in sup norm. This approach is limited as it supposes to be able to consistently estimate the density (or the signal) in sup-norm, which will typically fail under the local assumptions designed for modes estimation cited earlier. In \cite{Bobrowski}, a different approach based on image persistence is proposed to move beyond traditional plug-in methods that rely on sup-norm stability. This framework allows for the estimation of persistence diagrams over a broad class of functions, specifically, $q-$tame bounded functions. However, the generality of this setting comes at a cost: the work does not provide convergence rates or establish formal consistency results, likely due to the inherent looseness of the underlying assumptions. Under the same motivation, for the Gaussian white noise model and the non-parametric regression, in \cite{Henneuse24a} we studied classes of piecewise Hölder functions tolerating irregularities that are difficult to handle for signal estimation. But on which it is still possible to infer persistence diagrams with convergence rates following the minimax one classically known on Hölder spaces. The proposed procedure in this work is based on histogram estimation of the (sub or super) level sets. We believe that these level sets estimators contain significantly more information than the one contained in persistence diagrams. In particular, we illustrate it in this work, showing that they contain the missing information to move from persistence diagram estimation to modes inference.
\subsection*{Contribution}
We propose a new framework for mode inference that combines geometric conditions from the TDA literature with standard analytical assumptions by introducing a new class of piecewise-continuous functions, denoted as \( S_{d}(L, \alpha, \mu, R_{\mu}, C, h_{0}) \). These classes contain \((L, \alpha)\)-piecewise-Hölder-continuous densities, with discontinuity loci having geometric complexity controlled by \( \mu \) and \( R_{\mu} \), using the \(\mu\)-reach, a geometric measure introduced in \cite{mureach}. We furthermore require that the densities within these classes satisfy \eqref{eq: power func} on \( h_{0} \)-neighborhoods of the modes. This provides control over both mode separation and prominence, playing a role analogous to the \( r \)-saliency condition introduced in \cite{DasguptaKoptufe14} and \cite{jiang17}. It is worth noting that the regularity conditions we assume are globally weaker than those imposed in \cite{jiang17}, as we do not assume global continuity. They are also locally weaker than those in \cite{DasguptaKoptufe14}, as we allow non-differentiability and even discontinuities around the modes.

 In this context, we propose a procedure based on $H_{0}$-persistence diagrams and ``rough'' super level sets estimation to infer the number of modes, their locations and associated local maxima. The procedure can be outlined as :
 \begin{itemize}
     \item Estimate the superlevel sets of the density using histograms, and apply a thickening step (depending on the parameter \( \mu \), see Section~\ref{sec: Proc and result}). This yields a collection of "rough" superlevel set estimators \( \hat{\mathcal{F}}_{\lambda} \), for \( \lambda \in \mathbb{R} \). The thickening step ensures the correct recovery of the connectivity of the level sets.
     \item Compute the associated $H_{0}-$persistence diagram $\widehat{\operatorname{dgm}(f)}$.
     \item Construct $\overline{\operatorname{dgm}(f)}$ by removing the points in $\widehat{\operatorname{dgm}(f)}$ at distance at most $\delta$ from the diagonal (i.e. with lifetime shorter than $\delta$), where $\delta$ is chosen depending on some parameters of the model we consider. 
     \item The number of points $\hat{k}$ in $\overline{\operatorname{dgm}(f)}$ is our estimator of the number of modes. 
     \item For each $(\hat{b}_{i},\hat{d}_{i})\in\overline{\operatorname{dgm}(f)}$, we can associate a connected component $\hat{C}_{i}$ of $\hat{\mathcal{F}}_{\hat{b}_{i}}$ (see Section \ref{Section : Sup modes 1}). Take any $\hat{x}_{i}\in \hat{C}_{i}$. The collection $\{\hat{x}_{i},i\in \{1,...,\hat{k}\}\}$ is our estimator of the modes and $\{\hat{b}_{i},i\in \{1,...,\hat{k}\}\}$ our estimator of the associated local maxima.
 \end{itemize}
We study the convergence properties of the proposed procedure over the classes \( S_{d}(L,\alpha,\mu, R_{\mu},C,h_{0}) \). More precisely, we identify thresholds on mode separation and prominence, both of which can be formalized in our framework as a threshold condition on the parameter \( h_0 \). Below this threshold, mode detection is impossible (Proposition~\ref{prp: tight separation}); above it, our procedure recovers, with high probability, the exact number of modes and estimates their locations and associated maxima at minimax rates (Theorems~\ref{coroModes1} and~\ref{lowerboundsModes}).

These results are of significant interest for several reasons. First, they establish minimax statistical guarantees for our procedure within a broad framework. Notably, this framework accommodates densities with discontinuities, possibly near or at the modes, cases that typically challenge standard approaches such as mean-shift, as we demonstrate through numerical experiments. Secondly and more broadly, while the assumption that modes must be sufficiently prominent and sufficiently separated is central in the literature concerning modes estimation, for instance \cite{DasguptaKoptufe14, jiang17}, which formalized this as a \( r \)-saliency condition, the quantification of ``how much is sufficient'' has, to the best of our knowledge, not yet been addressed. Our results, in particular the identification of the threshold on \( h_0 \), provide precise insights into the minimal separation and prominence needed for consistent modes estimation.

Furthermore, in the process of establishing our results on mode inference, we also derive convergence guarantees for our estimators of the \( H_0 \) persistence diagram (Propositions~\ref{MainProp}, \ref{Lower Bound Diag}, and Corollary~\ref{estimation borne sup}), thereby extending the results obtained in our earlier work~\cite{Henneuse24a}. These findings are of independent interest, as they demonstrate that consistent estimation of \( H_0 \) persistence diagrams is possible under significantly weaker geometric conditions than those previously required. In particular, the proposed method, based on thickened histograms, is able to recover consistently persistence diagrams in cases where the plug-in histogram estimator from our earlier work fails.\\\\
The paper is organized as follows. Section \ref{Appendix: Background} reviews the geometric and topological background relevant to our work. Section \ref{Framework} introduces the formal framework under study. Section \ref{sec: Proc and result} presents our procedure along with the main theoretical results. Section \ref{sec: proofs} contains the proofs of these results. Section \ref{Section: Numerical illustration} provides numerical illustrations. Additional results and the proofs of technical lemmas are deferred to the Appendix.
 \section{Background}
\label{Appendix: Background}
This section provides the necessary background to follow this paper.
\subsection{Distance function, generalized gradient and $\mu$-reach}
In this section, we recall some concepts from geometric measure theory involved in this work. For a set $K\subset[0,1]^{d}$, we denote $\overline{K}$ its adherence and $\partial K$ its boundary. Let $K\subset[0,1]^{d}$ a compact set, the \textbf{distance function} $d_{K}$ is given by,
$$d_{K}:x\mapsto \min\limits_{y\in K}||x-y||_{2}.$$
Generally the distance function is not differentiable everywhere, but we can define a generalized gradient function that matches the gradient at points where the distance function is differentiable. Consider the set of closest points to $x$ in $K$,
$$\Gamma_K(x)=\{y \in K \mid ||x-y||_{2}=d_{K}(x)\}$$
 and for $x\in[0,1]^{d}\setminus K$, let $\Theta_{K}(x)$ the center of the unique smallest ball enclosing $\Gamma_K(x)$, the \textbf{generalized gradient function} $\nabla_{d_K}(x)$ is then defined as,
$$\nabla_{d_K}(x)=\frac{x-\Theta_{K}(x)}{d_{k}(x)}.$$
 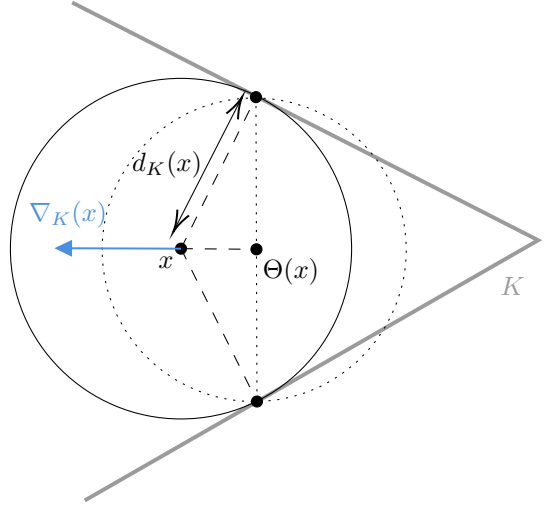
\begin{figure}[H]
    \centering
    \begin{tikzpicture}[x=0.75pt,y=0.75pt,yscale=-1,xscale=1]

\draw [color={rgb, 255:red, 155; green, 155; blue, 155 }  ,draw opacity=1 ][line width=1.5]    (137.98,265.87) -- (366.67,134.5) -- (131.67,14.53) ;
\draw   (100.5,138.76) .. controls (100.5,91.28) and (139.02,52.8) .. (186.54,52.8) .. controls (234.06,52.8) and (272.59,91.28) .. (272.59,138.76) .. controls (272.59,186.23) and (234.06,224.71) .. (186.54,224.71) .. controls (139.02,224.71) and (100.5,186.23) .. (100.5,138.76) -- cycle ;
\draw  [fill={rgb, 255:red, 0; green, 0; blue, 0 }  ,fill opacity=1 ] (183.76,138.76) .. controls (183.76,137.22) and (185,135.97) .. (186.54,135.97) .. controls (188.08,135.97) and (189.33,137.22) .. (189.33,138.76) .. controls (189.33,140.29) and (188.08,141.54) .. (186.54,141.54) .. controls (185,141.54) and (183.76,140.29) .. (183.76,138.76) -- cycle ;
\draw  [fill={rgb, 255:red, 0; green, 0; blue, 0 }  ,fill opacity=1 ] (221.55,62.29) .. controls (221.55,60.75) and (222.8,59.5) .. (224.34,59.5) .. controls (225.88,59.5) and (227.13,60.75) .. (227.13,62.29) .. controls (227.13,63.82) and (225.88,65.07) .. (224.34,65.07) .. controls (222.8,65.07) and (221.55,63.82) .. (221.55,62.29) -- cycle ;
\draw  [fill={rgb, 255:red, 0; green, 0; blue, 0 }  ,fill opacity=1 ] (222.07,215.9) .. controls (222.07,214.36) and (223.32,213.12) .. (224.86,213.12) .. controls (226.4,213.12) and (227.64,214.36) .. (227.64,215.9) .. controls (227.64,217.44) and (226.4,218.68) .. (224.86,218.68) .. controls (223.32,218.68) and (222.07,217.44) .. (222.07,215.9) -- cycle ;
\draw  [dash pattern={on 4.5pt off 4.5pt}]  (186.54,138.76) -- (224.86,215.9) ;
\draw  [dash pattern={on 4.5pt off 4.5pt}]  (186.54,138.76) -- (224.34,62.29) ;
\draw  [dash pattern={on 0.84pt off 2.51pt}] (146.79,139.31) .. controls (146.79,97.01) and (181.12,62.72) .. (223.46,62.72) .. controls (265.81,62.72) and (300.13,97.01) .. (300.13,139.31) .. controls (300.13,181.61) and (265.81,215.9) .. (223.46,215.9) .. controls (181.12,215.9) and (146.79,181.61) .. (146.79,139.31) -- cycle ;
\draw  [dash pattern={on 0.84pt off 2.51pt}]  (224.34,62.29) -- (224.86,215.9) ;
\draw  [fill={rgb, 255:red, 0; green, 0; blue, 0 }  ,fill opacity=1 ] (221.81,139.09) .. controls (221.81,137.56) and (223.06,136.31) .. (224.6,136.31) .. controls (226.14,136.31) and (227.39,137.56) .. (227.39,139.09) .. controls (227.39,140.63) and (226.14,141.88) .. (224.6,141.88) .. controls (223.06,141.88) and (221.81,140.63) .. (221.81,139.09) -- cycle ;
\draw [color={rgb, 255:red, 74; green, 144; blue, 226 }  ,draw opacity=1 ][line width=0.75]    (186.54,138.76) -- (125.67,138.54) ;
\draw [shift={(122.67,138.53)}, rotate = 0.2] [fill={rgb, 255:red, 74; green, 144; blue, 226 }  ,fill opacity=1 ][line width=0.08]  [draw opacity=0] (8.93,-4.29) -- (0,0) -- (8.93,4.29) -- cycle    ;
\draw  [dash pattern={on 4.5pt off 4.5pt}]  (186.54,138.76) -- (224.6,139.09) ;
\draw    (216.42,64.05) -- (183.58,128.48) ;
\draw [shift={(182.67,130.27)}, rotate = 297.01] [color={rgb, 255:red, 0; green, 0; blue, 0 }  ][line width=0.75]    (10.93,-3.29) .. controls (6.95,-1.4) and (3.31,-0.3) .. (0,0) .. controls (3.31,0.3) and (6.95,1.4) .. (10.93,3.29)   ;
\draw [shift={(217.33,62.27)}, rotate = 117.01] [color={rgb, 255:red, 0; green, 0; blue, 0 }  ][line width=0.75]    (10.93,-3.29) .. controls (6.95,-1.4) and (3.31,-0.3) .. (0,0) .. controls (3.31,0.3) and (6.95,1.4) .. (10.93,3.29)   ;

\draw (346.41,151.58) node [anchor=north west][inner sep=0.75pt]  [font=\small,color={rgb, 255:red, 155; green, 155; blue, 155 }  ,opacity=1 ]  {$K$};
\draw (173.88,141.49) node [anchor=north west][inner sep=0.75pt]  [font=\small]  {$x$};
\draw (226.6,142.49) node [anchor=north west][inner sep=0.75pt]  [font=\small]  {$\Theta ( x)$};
\draw (108.67,113.47) node [anchor=north west][inner sep=0.75pt]  [font=\small,color={rgb, 255:red, 74; green, 144; blue, 226 }  ,opacity=1 ]  {$\nabla _{d_K}( x)$};
\draw (160.67,88.4) node [anchor=north west][inner sep=0.75pt]  [font=\small]  {$d_{K}( x)$};
\end{tikzpicture}
    \caption{2D example with 2 closest points}
\end{figure}
We can now introduce the notion of $\mu-$reach \citep{mureach} that will be used to measure and control the geometric complexity of discontinuities sets (see Assumption \textbf{A3}, Section \ref{Framework}). Let $K \subset [0,1]^{d}$ a compact set and $0\leq \mu\leq 1$, its \textbf{$\mu$-reach} is defined by,
\begin{equation}
\label{def mu-reach}
\operatorname{reach}_\mu(K)=\inf \left\{r \mid \inf _{d_K^{-1}(r)}\left\|\nabla_{d_K}\right\|_{2}<\mu\right\}.
\end{equation}
 The $\mu-$reach is positive for a large class of sets. For example, all piecewise linear compact sets have positive $\mu-$reach (for some $\mu>0$). For $\mu=1$, this simply corresponds to the reach, a standard curvature measure, introduced in \cite{Fed59} and used in our earlier work \citep{Henneuse24a} to control the geometry of the discontinuities. In our context, using the $\mu-$reach instead of the reach to measure the complexity of the discontinuities allows considering significantly wider classes of irregular densities, tolerating corners and multiple points (i.e. self intersections) in their discontinuities set.
 
Another way to understand the $\mu-$reach is to see it as the distance from the $\mu-$medial axis.  The \textbf{$\mu-$medial axis} of a set $K\subset[0,1]^{d}$ is defined by,
    $$\operatorname{Med}_{\mu}(K)=\left\{x\in[0,1]^{d}\quad|\quad\|\nabla_{d_K}(x)\|_{2}<\mu\right\}.$$
The $\mu-$reach of $K$ is then equal to $d_{2}(K,\overline{\operatorname{Med}_{\mu}(K)})$.
\begin{figure}[H]
\centering
\begin{subfigure}{.5\textwidth}
  \centering
\begin{tikzpicture}[x=0.75pt,y=0.75pt,yscale=-1,xscale=1]

\draw   (248.6,57.2) -- (442,57.2) -- (442,250.6) -- (248.6,250.6) -- cycle ;
\draw  [dash pattern={on 4.5pt off 4.5pt}] (316,88.6) .. controls (331,70.6) and (388,54) .. (368,74) .. controls (348,94) and (346,103.6) .. (366,133.6) .. controls (386,163.6) and (309,202.6) .. (308,168.6) .. controls (307,134.6) and (258,120.6) .. (257,108.6) .. controls (256,96.6) and (255.5,92.15) .. (266.75,83.38) .. controls (278,74.6) and (301,106.6) .. (316,88.6) -- cycle ;
\draw  [dash pattern={on 4.5pt off 4.5pt}] (382,142.6) .. controls (402,132.6) and (445,131.6) .. (433,148.6) .. controls (421,165.6) and (439,209.6) .. (427,222.6) .. controls (415,235.6) and (408,206.2) .. (390,235.6) .. controls (372,265) and (362,152.6) .. (382,142.6) -- cycle ;
\draw  [dash pattern={on 4.5pt off 4.5pt}] (268,176) .. controls (278,171) and (312,170.6) .. (292,190.6) .. controls (272,210.6) and (332,180.6) .. (352,210.6) .. controls (372,240.6) and (301,256.6) .. (281,226.6) .. controls (261,196.6) and (253,208.5) .. (253,198.5) .. controls (253,188.5) and (258,181) .. (268,176) -- cycle ;
\draw  [dash pattern={on 4.5pt off 4.5pt}] (379.4,189.8) .. controls (379.4,178.31) and (388.71,169) .. (400.2,169) .. controls (411.69,169) and (421,178.31) .. (421,189.8) .. controls (421,201.29) and (411.69,210.6) .. (400.2,210.6) .. controls (388.71,210.6) and (379.4,201.29) .. (379.4,189.8) -- cycle ;
\draw  [dash pattern={on 4.5pt off 4.5pt}] (348,140.6) .. controls (346,126.6) and (349,110.6) .. (322,105.6) .. controls (295,100.6) and (270,100.6) .. (283,107.6) .. controls (296,114.6) and (350,154.6) .. (348,140.6) -- cycle ;

\draw (314,109.4) node [anchor=north west][inner sep=0.75pt]    {$M_{1}$};
\draw (317,152.4) node [anchor=north west][inner sep=0.75pt]    {$M_{2}$};
\draw (295,210.4) node [anchor=north west][inner sep=0.75pt]    {$M_{3}$};
\draw (389,179.4) node [anchor=north west][inner sep=0.75pt]    {$M_{4}$};
\draw (388,91.4) node [anchor=north west][inner sep=0.75pt]    {$M_{6}$};
\draw (396,142.4) node [anchor=north west][inner sep=0.75pt]    {$M_{5}$};

\end{tikzpicture}

\caption{Positive ($1-$)reach}

\end{subfigure}%
\begin{subfigure}{.5\textwidth}
\centering
\begin{tikzpicture}[x=0.75pt,y=0.75pt,yscale=-1,xscale=1]

\draw  [draw opacity=0][fill={rgb, 255:red, 208; green, 2; blue, 27 }  ,fill opacity=1 ] (295.36,225.6) .. controls (295.36,223.81) and (296.81,222.36) .. (298.6,222.36) .. controls (300.39,222.36) and (301.84,223.81) .. (301.84,225.6) .. controls (301.84,227.39) and (300.39,228.84) .. (298.6,228.84) .. controls (296.81,228.84) and (295.36,227.39) .. (295.36,225.6) -- cycle ;
\draw  [draw opacity=0][fill={rgb, 255:red, 208; green, 2; blue, 27 }  ,fill opacity=1 ] (245.36,200.6) .. controls (245.36,198.81) and (246.81,197.36) .. (248.6,197.36) .. controls (250.39,197.36) and (251.84,198.81) .. (251.84,200.6) .. controls (251.84,202.39) and (250.39,203.84) .. (248.6,203.84) .. controls (246.81,203.84) and (245.36,202.39) .. (245.36,200.6) -- cycle ;
\draw  [draw opacity=0][fill={rgb, 255:red, 208; green, 2; blue, 27 }  ,fill opacity=1 ] (295.36,250.6) .. controls (295.36,248.81) and (296.81,247.36) .. (298.6,247.36) .. controls (300.39,247.36) and (301.84,248.81) .. (301.84,250.6) .. controls (301.84,252.39) and (300.39,253.84) .. (298.6,253.84) .. controls (296.81,253.84) and (295.36,252.39) .. (295.36,250.6) -- cycle ;
\draw  [draw opacity=0][fill={rgb, 255:red, 208; green, 2; blue, 27 }  ,fill opacity=1 ] (351.76,160.6) .. controls (351.76,158.81) and (353.21,157.36) .. (355,157.36) .. controls (356.79,157.36) and (358.24,158.81) .. (358.24,160.6) .. controls (358.24,162.39) and (356.79,163.84) .. (355,163.84) .. controls (353.21,163.84) and (351.76,162.39) .. (351.76,160.6) -- cycle ;
\draw  [draw opacity=0][fill={rgb, 255:red, 208; green, 2; blue, 27 }  ,fill opacity=1 ] (295.36,200.6) .. controls (295.36,198.81) and (296.81,197.36) .. (298.6,197.36) .. controls (300.39,197.36) and (301.84,198.81) .. (301.84,200.6) .. controls (301.84,202.39) and (300.39,203.84) .. (298.6,203.84) .. controls (296.81,203.84) and (295.36,202.39) .. (295.36,200.6) -- cycle ;
\draw  [draw opacity=0][fill={rgb, 255:red, 208; green, 2; blue, 27 }  ,fill opacity=1 ] (288.47,119.47) .. controls (288.47,117.68) and (289.92,116.23) .. (291.71,116.23) .. controls (293.5,116.23) and (294.96,117.68) .. (294.96,119.47) .. controls (294.96,121.26) and (293.5,122.71) .. (291.71,122.71) .. controls (289.92,122.71) and (288.47,121.26) .. (288.47,119.47) -- cycle ;
\draw  [draw opacity=0][fill={rgb, 255:red, 208; green, 2; blue, 27 }  ,fill opacity=1 ] (336.76,57.6) .. controls (336.76,55.81) and (338.21,54.36) .. (340,54.36) .. controls (341.79,54.36) and (343.24,55.81) .. (343.24,57.6) .. controls (343.24,59.39) and (341.79,60.84) .. (340,60.84) .. controls (338.21,60.84) and (336.76,59.39) .. (336.76,57.6) -- cycle ;
\draw  [draw opacity=0][fill={rgb, 255:red, 208; green, 2; blue, 27 }  ,fill opacity=1 ] (245.33,111.59) .. controls (245.33,109.8) and (246.78,108.35) .. (248.57,108.35) .. controls (250.36,108.35) and (251.81,109.8) .. (251.81,111.59) .. controls (251.81,113.38) and (250.36,114.83) .. (248.57,114.83) .. controls (246.78,114.83) and (245.33,113.38) .. (245.33,111.59) -- cycle ;
\draw  [draw opacity=0][fill={rgb, 255:red, 208; green, 2; blue, 27 }  ,fill opacity=1 ] (423.62,127.5) .. controls (423.62,125.71) and (425.07,124.26) .. (426.86,124.26) .. controls (428.65,124.26) and (430.1,125.71) .. (430.1,127.5) .. controls (430.1,129.29) and (428.65,130.74) .. (426.86,130.74) .. controls (425.07,130.74) and (423.62,129.29) .. (423.62,127.5) -- cycle ;
\draw   (248.6,57.2) -- (442,57.2) -- (442,250.6) -- (248.6,250.6) -- cycle ;
\draw  [dash pattern={on 4.5pt off 4.5pt}]  (250,112) .. controls (314,143.6) and (300,87.6) .. (340,57.6) ;
\draw  [dash pattern={on 4.5pt off 4.5pt}]  (248.6,200.6) -- (298.6,200.6) ;
\draw  [dash pattern={on 4.5pt off 4.5pt}]  (298.6,200.6) -- (298.6,250.6) ;
\draw  [dash pattern={on 4.5pt off 4.5pt}] (359,197.6) .. controls (369,202.6) and (355.5,247.2) .. (407.5,157.6) .. controls (459.5,68) and (401,179) .. (421,209) .. controls (441,239) and (367,259.6) .. (347,229.6) .. controls (327,199.6) and (349,192.6) .. (359,197.6) -- cycle ;
\draw  [dash pattern={on 4.5pt off 4.5pt}] (366.5,74.6) .. controls (386.5,64.6) and (423,93.6) .. (403,113.6) .. controls (383,133.6) and (375,190.6) .. (355,160.6) .. controls (335,130.6) and (346.5,84.6) .. (366.5,74.6) -- cycle ;
\draw  [dash pattern={on 4.5pt off 4.5pt}]  (291.71,119.47) .. controls (346.26,139.6) and (315,190.6) .. (355,160.6) ;
\draw  [dash pattern={on 4.5pt off 4.5pt}]  (355,160.6) .. controls (338.5,190.6) and (326.5,203.6) .. (298.6,225.6) ;

\draw (268,72.4) node [anchor=north west][inner sep=0.75pt]    {$M_{1}$};
\draw (265,215.4) node [anchor=north west][inner sep=0.75pt]    {$M_{2}$};
\draw (267,150.4) node [anchor=north west][inner sep=0.75pt]    {$M_{3}$};
\draw (359,101.4) node [anchor=north west][inner sep=0.75pt]    {$M_{4}$};
\draw (371,212.4) node [anchor=north west][inner sep=0.75pt]    {$M_{5}$};
\draw (414,76.4) node [anchor=north west][inner sep=0.75pt]    {$M_{6}$};

\end{tikzpicture}

  \caption{Positive $\mu-$reach ($\mu$ small)}
  
\end{subfigure}
\caption{(a) displays a partition $M_{1}$,..., $M_{6}$ such that $\operatorname{reach}_{1}\left(\partial M_{1}\cup...\cup \partial M_{6}\right)>0$. (b) displays a partition $M_{1}$,..., $M_{6}$ such that $\operatorname{reach}_{1}\left(\partial M_{1}\cup...\cup \partial M_{6}\right)=0$ (in red are highlighted problematic points) but for sufficiently small $\mu>0$, $\operatorname{reach}_{\mu}\left(\partial M_{1}\cup...\cup \partial M_{6}\right)>0$. }

\end{figure}
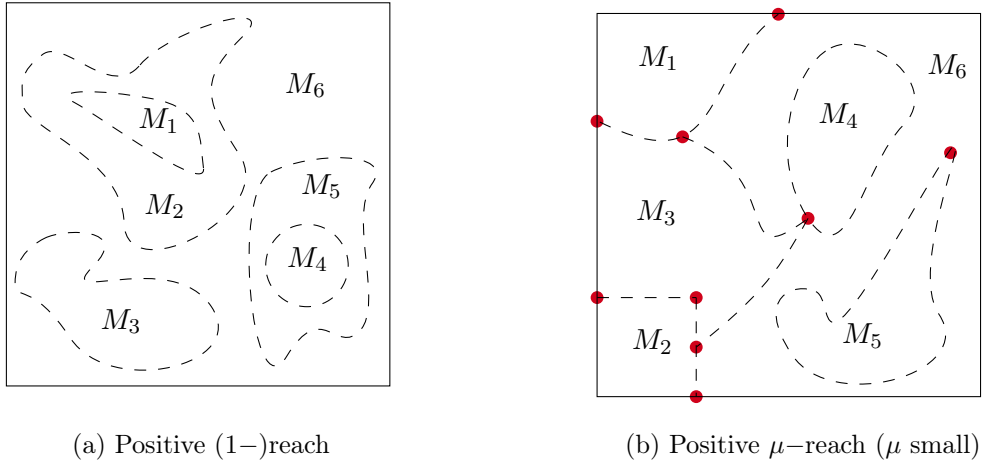
\begin{figure}[H]
    \centering
    \begin{tikzpicture}[x=0.75pt,y=0.75pt,yscale=-1,xscale=1]

\draw    (132.77,246.69) -- (230.55,158.28) ;
\draw    (230.55,158.28) -- (317.89,246.69) ;
\draw    (129.92,70.58) .. controls (186.88,136.8) and (160.3,163.65) .. (230.55,158.28) .. controls (300.81,152.91) and (289.98,82.81) .. (323.59,13.31) ;
\draw [color={rgb, 255:red, 155; green, 155; blue, 155 }  ,draw opacity=1 ][line width=1.5]    (129.92,70.58) .. controls (238.15,124.27) and (264.03,76.85) .. (323.59,13.31) ;
\draw [color={rgb, 255:red, 155; green, 155; blue, 155 }  ,draw opacity=1 ][line width=1.5]    (323.59,13.31) -- (317.89,246.69) ;
\draw [color={rgb, 255:red, 155; green, 155; blue, 155 }  ,draw opacity=1 ][line width=1.5]    (132.77,246.69) -- (317.89,246.69) ;
\draw [color={rgb, 255:red, 155; green, 155; blue, 155 }  ,draw opacity=1 ][line width=1.5]    (129.92,70.58) -- (132.77,246.69) ;
\draw [color={rgb, 255:red, 74; green, 144; blue, 226 }  ,draw opacity=1 ][line width=3]    (165.05,123.2) .. controls (178.09,138.65) and (174.8,165.73) .. (230.55,158.28) .. controls (286.31,150.82) and (286.8,120.9) .. (298.84,83.41) ;
\draw  [draw opacity=0][fill={rgb, 255:red, 155; green, 155; blue, 155 }  ,fill opacity=1 ] (128.91,70.58) .. controls (128.91,70.05) and (129.36,69.61) .. (129.92,69.61) .. controls (130.48,69.61) and (130.94,70.05) .. (130.94,70.58) .. controls (130.94,71.12) and (130.48,71.55) .. (129.92,71.55) .. controls (129.36,71.55) and (128.91,71.12) .. (128.91,70.58) -- cycle ;
\draw  [color={rgb, 255:red, 208; green, 2; blue, 27 }  ,draw opacity=1 ][fill={rgb, 255:red, 208; green, 2; blue, 27 }  ,fill opacity=1 ] (228.63,158.28) .. controls (228.63,157.28) and (229.49,156.46) .. (230.55,156.46) .. controls (231.62,156.46) and (232.48,157.28) .. (232.48,158.28) .. controls (232.48,159.28) and (231.62,160.09) .. (230.55,160.09) .. controls (229.49,160.09) and (228.63,159.28) .. (228.63,158.28) -- cycle ;
\draw  [draw opacity=0][fill={rgb, 255:red, 155; green, 155; blue, 155 }  ,fill opacity=1 ] (131.76,246.69) .. controls (131.76,246.15) and (132.21,245.72) .. (132.77,245.72) .. controls (133.33,245.72) and (133.78,246.15) .. (133.78,246.69) .. controls (133.78,247.22) and (133.33,247.66) .. (132.77,247.66) .. controls (132.21,247.66) and (131.76,247.22) .. (131.76,246.69) -- cycle ;
\draw  [draw opacity=0][fill={rgb, 255:red, 155; green, 155; blue, 155 }  ,fill opacity=1 ] (316.88,246.69) .. controls (316.88,246.15) and (317.33,245.72) .. (317.89,245.72) .. controls (318.45,245.72) and (318.91,246.15) .. (318.91,246.69) .. controls (318.91,247.22) and (318.45,247.66) .. (317.89,247.66) .. controls (317.33,247.66) and (316.88,247.22) .. (316.88,246.69) -- cycle ;
\draw  [draw opacity=0][fill={rgb, 255:red, 155; green, 155; blue, 155 }  ,fill opacity=1 ] (322.58,13.31) .. controls (322.58,12.78) and (323.03,12.34) .. (323.59,12.34) .. controls (324.15,12.34) and (324.6,12.78) .. (324.6,13.31) .. controls (324.6,13.85) and (324.15,14.28) .. (323.59,14.28) .. controls (323.03,14.28) and (322.58,13.85) .. (322.58,13.31) -- cycle ;

\draw (329.4,215.2) node [anchor=north west][inner sep=0.75pt]  [font=\small]  {$\textcolor[rgb]{0.61,0.61,0.61}{K}$};
\draw (209.73,201.47) node [anchor=north west][inner sep=0.75pt]  [font=\small]  {$\operatorname{Med}_{1}( K)$};
\draw (253.73,159.47) node [anchor=north west][inner sep=0.75pt]  [font=\small,color={rgb, 255:red, 74; green, 144; blue, 226 }  ,opacity=1 ]  {$\textcolor[rgb]{0.29,0.56,0.89}{\operatorname{Med}_{\mu }( K)}$};
\draw (188.73,129.47) node [anchor=north west][inner sep=0.75pt]  [font=\small,color={rgb, 255:red, 74; green, 144; blue, 226 }  ,opacity=1 ]  {$\textcolor[rgb]{0.82,0.01,0.11}{\operatorname{Med}_{0}( K)}$};

\end{tikzpicture}
\caption{Illustration of $\mu-$medial axis for a set $K$. In black is represented the $1-$medial axis, in red the $0-$medial axis and in blue the $\mu-$medial axis for a small $0<\mu<1/2$.}
\end{figure}
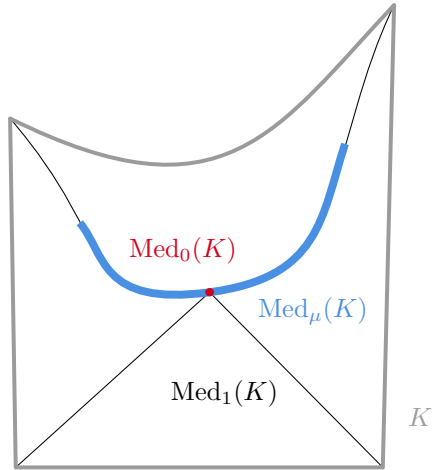
\subsection{Filtration, persistence module and $H_{0}$-persistence diagram}
In this section, we briefly present some notions related to persistent homology. Persistent homology encodes the evolution of topological features (in the homology sense) along a family of nested spaces, called \textbf{filtration}. Moving along indices, topological features (connected components, cycles, cavities, ...) can appear or die (existing connected components merge, cycles or cavities are filled, ...). In this paper, we focus on $H_{0}-$persistent homology, that describes the evolution of connectivity. For a broader overview and visual illustrations of persistent homology, we recommend \cite{chazal2021introduction}. For detailed and rigorous constructions, see \cite{chazal2013}. Additionally, since the construction discussed here involves (singular) homology, the reader can refer to \cite{Hatcher}.\\\\
The typical filtration that we will consider in this paper is, for a function $f:\mathbb{R}^{d}\rightarrow\mathbb{R}$, the family of superlevel sets $\left(\mathcal{F}_{\lambda}\right)_{\lambda\in\mathbb{R}}=\left(f^{-1}([\lambda,+\infty[)\right)_{\lambda\in\mathbb{R}}$. The associated family of homology groups of degree $0$, $\mathbb{V}_{f,0}=\left(H_{0}\left(\mathcal{F}_{\lambda}\right)\right)_{\lambda\in\mathbb{R}}$, equipped with $v_{\lambda}^{\lambda^{'}}$ the linear map induced by the inclusion $\mathcal{F}_{\lambda}\subset\mathcal{F}_{\lambda^{\prime}}$, for all $\lambda>\lambda^{'}$ forms a \textbf{persistence module}. To be more precise, in this paper, $H_{0}(.)$ is the singular homology functor in degree $0$ with coefficients in a field (typically $\mathbb{Z}/2\mathbb{Z}$). Hence, $H_{0}\left(\mathcal{F}_{\lambda}\right)$ is a vector space.\\\\
if, for all $\lambda>\lambda^{\prime} \in \mathbb{R}$, $\operatorname{rank}(v_\lambda^{\lambda^{\prime}})$ is finite, the module is said to be \textbf{$q-$tame}. It is then possible to show that the algebraic structure of the persistence module encodes exactly the evolution of the topological features along the filtration. For details, we encourage the reader to look at \cite{chazal2013}. Furthermore, the algebraic structure of such a persistence module can be summarized by a collection $\{(b_{i},d_{i}), i\in I\}\subset\overline{\mathbb{R}}^{2}$, which defines the \textbf{persistence diagram}. Following previous remarks, for $H_{0}-$persistent homology, $b_{i}$ corresponds to the birth time of a connected component, $d_{i}$ to its death time and $d_{i}-b_{i}$ to its lifetime.
\begin{figure}[H]
\centering
 \includegraphics[scale=0.55]{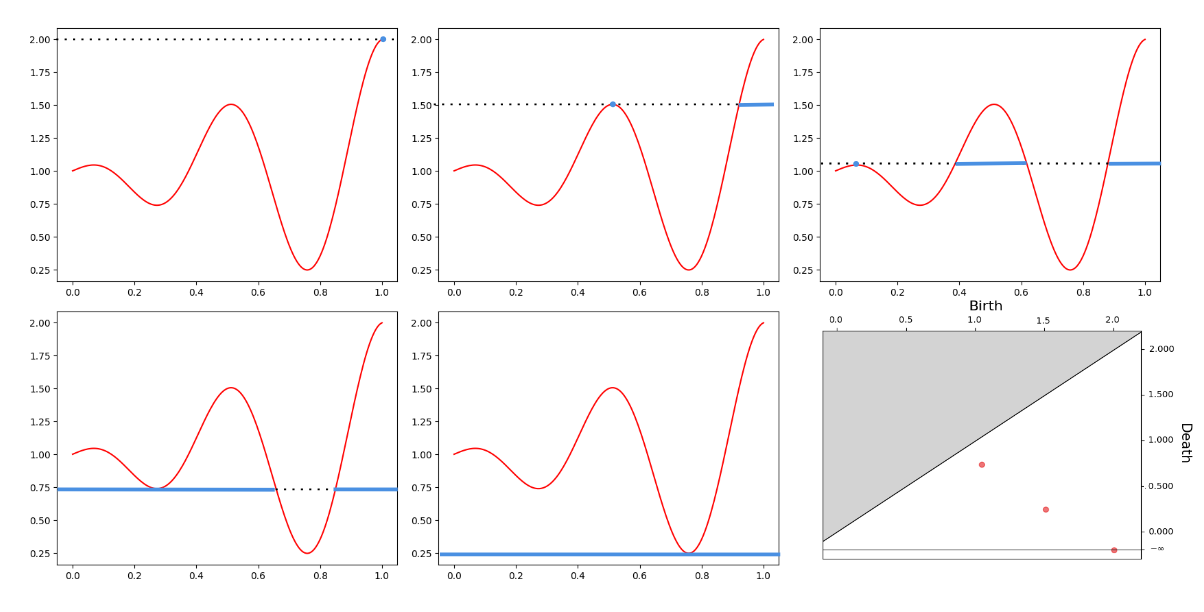}
\caption{Super level sets filtration of $f(x)=x\cos(4\pi x)$ over $[0,1]$ and the associated $H_{0}$-persistence diagram.}
\end{figure}
 To compare persistence diagrams, a popular distance, especially in statistical contexts, is the \textbf{bottleneck distance}, defined for two persistence diagrams $D_{1}$ and $D_{2}$ by,
$$d_{b}\left(D_{1},D_{2}\right)=\underset{\gamma\in\Gamma}{\inf}\underset{p\in D_{1}}{\sup}||p-\gamma(p)||_{\infty}$$
with $\Gamma$ the set of all bijections between $D_{1}$ and $D_{2}$ both enriched with the diagonal $\Delta=\{(x,x)\text{ s.t. }x\in\mathbb{R}^{2}\}$. The addition of $\Delta$ allows comparing diagrams with different cardinalities. \\\\
Now, we present the algebraic stability theorem for the bottleneck distance. This theorem was the key for proving upper bounds in \cite{Henneuse24a}, it will be also employed in this paper. This theorem relies on interleaving between modules. Let $f:[0,1]^{d}\rightarrow \mathbb{R}$ and $g:[0,1]^{d}\rightarrow \mathbb{R}$, the associated $H_{0}-$persistence modules for the superlevel sets filtrations, denoted respectively $\mathbb{V}$ and $\mathbb{W}$, are said to be \textbf{$\varepsilon$-interleaved} if there exist two families of linear maps $\phi=\left(\phi_{\lambda}\right)_{\lambda\mathbb{R}}$ and $\psi=\left(\psi_{\lambda}\right)_{\lambda\mathbb{R}}$ (which we will refer to as morphisms between persistence modules) where $\phi_{\lambda}:\mathbb{V}_{\lambda}\rightarrow\mathbb{W}_{\lambda-\varepsilon}$, $\psi_{\lambda}:\mathbb{V}_{\lambda}\rightarrow\mathbb{W}_{\lambda-\varepsilon}$, and for all $\lambda>\lambda^{'}$ the following diagrams commute,
\begin{equation}
\begin{tikzcd}
	{\mathbb{V}_{\lambda }} && {\mathbb{V}_{\lambda ^{'}}} & {\mathbb{W}_{\lambda }} && {\mathbb{W}_{\lambda ^{'}}} \\
	{\mathbb{W}_{\lambda -\varepsilon}} && {\mathbb{W}_{\lambda '-\varepsilon}} & {\mathbb{V}_{\lambda -\varepsilon}} && {\mathbb{V}_{\lambda '-\varepsilon}} \\
	{\mathbb{V}_{\lambda }} && {\mathbb{V}_{\lambda-2\varepsilon}} & {\mathbb{W}_{\lambda}} && {\mathbb{W}_{\lambda-2\varepsilon}} \\
	& {\mathbb{W}_{\lambda -\varepsilon}} &&& {\mathbb{V}_{\lambda -\varepsilon}}
	\arrow["{\phi _{\lambda }}"', from=1-1, to=2-1]
	\arrow["{\phi _{\lambda ^{'}}}", from=1-3, to=2-3]
	\arrow["{w_{\lambda -\varepsilon}^{\lambda ^{'} -\varepsilon}}"', from=2-1, to=2-3]
	\arrow["{w_{\lambda }^{\lambda ^{'}}}", from=1-4, to=1-6]
	\arrow["{\psi_{\lambda }}"', from=1-4, to=2-4]
	\arrow["{v_{\lambda -\varepsilon}^{\lambda ^{'} -\varepsilon}}"', from=2-4, to=2-6]
	\arrow["{\psi_{\lambda ^{'}}}", from=1-6, to=2-6]
	\arrow["{v_{\lambda }^{\lambda ^{'}-2\varepsilon}}", from=3-1, to=3-3]
	\arrow["{\phi _{\lambda }}"', from=3-1, to=4-2]
	\arrow["{\psi_{\lambda-\varepsilon}}"', from=4-2, to=3-3]
	\arrow["{\psi_{\lambda }}"', from=3-4, to=4-5]
	\arrow["{v_{\lambda }^{\lambda ^{'}}}", from=1-1, to=1-3]
	\arrow["{w_{\lambda }^{\lambda ^{'}-2\varepsilon}}", from=3-4, to=3-6]
	\arrow["{\phi _{\lambda-\varepsilon}}"', from=4-5, to=3-6]
\end{tikzcd}
\label{interleaving def}
\end{equation}
\begin{customthm}{ \citep[]["algebraic stability"]{Chazal2009}}
Let $\mathbb{V}$ and $\mathbb{W}$ two $q-$tame persistence modules. If  $\mathbb{V}$ and $\mathbb{W}$  are $\varepsilon-$interleaved then,
$$d_{b}\left(\operatorname{dgm}(\mathbb{V}),\operatorname{dgm}(\mathbb{W})\right)\leq \varepsilon.$$
\end{customthm}
  A direct corollary of this result, proved  earlier in special cases \citep{Baranikov94,CSEH2005}, is the sup-norm stability of persistence diagrams. We insist on the fact that this is a strictly weaker result than algebraic stability.
\begin{customthm}{["sup norm stability"]}
    Let $f$ and $g$ two real-valued function and $s\in\mathbb{N}$. If their persistence modules for the $s-th$ order homology, denoted $\mathbb{V}$ and $\mathbb{W}$, are $q-$tame, then,
$$d_{b}\left(\operatorname{dgm}\left(\mathbb{V}\right),\operatorname{dgm}\left(\mathbb{W}\right)\right)\leq ||f-g||_{\infty}.$$
\end{customthm}
 As highlighted in the introduction, the sup-norm stability is often used to upper bound the errors in bottleneck distance of "plug-in" estimators of persistence diagrams. It enables the direct translation of convergence rates in sup-norm  to convergence rates in bottleneck distance, which for regular classes of signals provides minimax upper bounds. However, when the convergence in sup-norm of the preliminary estimator is not ensured, this approach falls short. 
\begin{figure}[H]
\centering
\begin{subfigure}{.5\textwidth}
  \centering
  \includegraphics[scale=0.45]{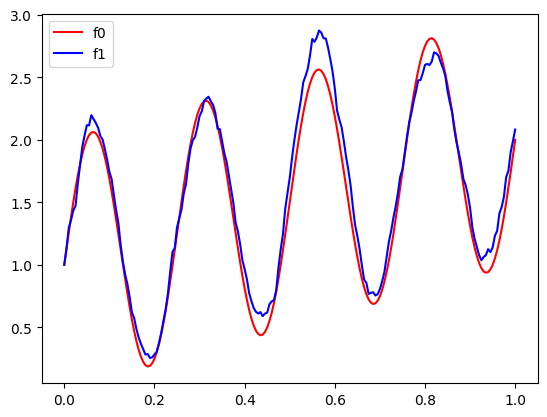}
  \caption{graphs of $f_0$ and $f_1$}
  
\end{subfigure}%
\begin{subfigure}{.5\textwidth}
  \centering
  \includegraphics[scale=0.5]{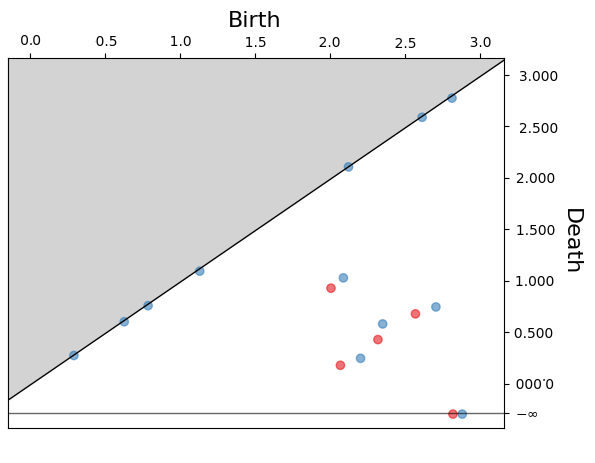}
  \caption{persistence diagrams of $f_0$ and $f_1$}
  
\end{subfigure}
\caption{1D Illustration of stability theorems.}

\end{figure}
\section{Framework}
\label{Framework}
Now, we discuss the precise settings considered in this paper. Let $f:[0,1]^{d}\rightarrow \mathbb{R}$ a probability density, we suppose that we observe $X=\{X_{1},...,X_{n}\}$ points sampled from $\mathbb{P}_{f}$.
 
Regularity assumptions on the densities are inspired by the one consider in \cite{Henneuse24a} (adapted for super level filtration). A notable difference is that we relax the positive reach assumption made there into a positive $\mu-$reach assumption. This allows considerably wilder discontinuities sets, tolerating multiple points and corners. We consider the following assumptions over $f$ :
\begin{itemize}
\item \textbf{A1.} f is a piecewise $(L,\alpha)-$Hölder-continuous probability density, i.e. there exists $M_{1},...,M_{l}$ disjoint open sets of $[0,1]^{d}$ such that,
$$\bigcup\limits_{i=1}^{l}\overline{M_{i}}=[0,1]^{d}$$ and $f|_{M_{i}}$ is in $H(L,\alpha)$, for all $ i\in\{1,...,l\}$, with,
$$H(L,\alpha)=\left\{f:[0,1]^{d}\rightarrow \mathbb{R}\text{ s.t. }|f(x)-f(y)|\leq L\|x-y\|_{2}^{\alpha},\forall x,y\in[0,1]^{d}\right\}.$$
Assumption \textbf{A1} control, on each $M_{i}$, how $f$ is ``spiky''. The smaller the value of $\alpha$, the less mass there is around modes, making it more challenging to infer local maxima accurately.
\item \textbf{A2}. $f$ satisfies, for all $ x_{0}\in [0,1]^{d}$,
$$\underset{x\in \bigcup\limits_{i=1}^{l}M_{i}\rightarrow x_{0}}{\limsup}f(x)=f(x_{0}).$$
In this context, two signals, differing only on a null set, are statistically indistinguishable. Persistent homology is sensitive to point-wise irregularity, two signals differing only on a null set can have persistence diagrams that are arbitrarily far. Assumption \textbf{A2} prevents such scenario.
\item \textbf{A3.} Let $\mu\in]0,1]$ and $R_{\mu}>0$, for all $I\subset \{1,...,l\}$
$$\operatorname{reach}_{\mu}\left(\bigcup\limits_{i\in I}\partial M_{i}\right)\geq R_{\mu}.$$
Here $\operatorname{reach}_{\mu}$ denotes the $\mu-$reach (defined by (\ref{def mu-reach}) in Section \ref{Appendix: Background}). This can be thought of as a geometric characterization of the regularity of the discontinuity set. In a similar spirit, in \cite{Henneuse24a}, we used the reach \citep{Fed59} to control the geometry of the discontinuities. As highlighted in Section~\ref{Appendix: Background}, the reach coincides with the $\mu$-reach considered here in the special case $\mu = 1$. As previously explained, this particular case imposes several important constraints on the discontinuity set: for instance, it cannot contain self-intersections (i.e., multiple points) or corners. By considering $\mu < 1$, we are able to encompass a much broader class of discontinuity sets, including those that exhibit self-intersections and corners.
In deed, the combination of Assumptions \textbf{A3}, \textbf{A2} and \textbf{A1} ensures that at all point $x\in[0,1]^{d}$ there is a small half cone of apex $x$ and angles $2\cos^{-1}(\mu)$ on which $f$ is $(L,\alpha)-$Hölder. Following the remark made on \textbf{A1}, it gives control over the mass of the density around modes. 
\end{itemize}
The class of functions satisfying \textbf{A1}, \textbf{A2} and \textbf{A3} is denoted $S_{d}(L,\alpha,\mu,R_{\mu})$. As proved  in Appendix \ref{Appendix:q-tame}, densities in $S_{d}(L,\alpha,\mu,R_{\mu})$ have well-defined persistence diagram. We prove in this work (Proposition \ref{MainProp}) that these assumptions are sufficient to infer the $H_{0}-$persistence diagram coming from the super level sets of $f$ consistently. For modes estimation, we require an additional assumption :
\begin{itemize}
    \item \textbf{A4.} Let $h_{0}>0$, $0<C<L$ and $0<\alpha\leq 1$, for any $x$ local maxima of $f$ and $y\in B_{2}(x,h_{0})$, 
    $$C\|x-y\|^{\alpha}_{2}\leq f(x)-f(y).$$
     This assumption is common in the context of modes inference \citep{DonohoLiu91, DasguptaKoptufe14, jiang17, ariascastro2021estimation}. It has several implications. First, it ensures that local maxima are strict (modal sets are here singletons). Secondly, it ensures that the modes are well separated (at distance at least $h_{0}$) and that $f$ is not too flat around modes. The larger the value of $\alpha$, the more $f$ flattens around modes, making them more challenging to locate accurately. In particular, denote $\delta=Ch_{0}^{\alpha}$ and $\operatorname{dgm}(f)$ the $H_{0}-$persistence diagram of $f$, we have,
    $$\inf\limits_{(b,d)\in \operatorname{dgm}\left(f\right)}b-d\geq \delta.$$
    Thus, \textbf{A4} ensures that $\operatorname{dgm}(f)$ contains no point at a distance less than $\delta$ from the diagonal, i.e. the prominence of the modes is lower bounded by $\delta$.Consequently, it avoids having arbitrarily small oscillations and undetectable modes. To this extent, it is comparable to the ``$r-$saliency'' assumption from \cite{DasguptaKoptufe14} and \cite{jiang17}, in the sense that it ensures that modes are separated by ``valleys'' of controlled depth and width. Then, it also controls the number of modes that $f$ can admit, as we are working on the compact set $[0,1]^{d}$ the maximal number of modes is of order $\asymp (1/h_{0})^{d}$.
\end{itemize}
The class of functions satisfying Assumptions \textbf{A1}, \textbf{A2}, \textbf{A3} and \textbf{A4} is denoted $S_{d}(L,\alpha,\mu, R_{\mu},C,h_{0})$.
\section{Procedures descriptions and results}
\label{sec: Proc and result}
In this section, we present our estimation strategy in detail and state several associated convergence results. Section~\ref{PersEst sec} focuses on the preliminary step of estimating $H_0$ persistence diagrams, while Section~\ref{Section : modes estim} is devoted to the inference of multiple modes.
\subsection{$H_{0}-$persistence diagram estimation}
\label{PersEst sec}
This section presents our procedure for persistence diagram estimation, which serves as the foundation of our mode inference method. We prove its consistency and quantify its convergence rates over the classes $S_{d}(L,\alpha,\mu,R_{\mu})$, which will be instrumental in later sections to establish our main results on mode inference. Unlike the setting of \cite{Henneuse24a}, where the positive reach assumption allows for direct plug-in estimation via histograms, the weaker assumptions considered here render this approach insufficient, as illustrated in Figure \ref{problem mu reach H0}. The presence of self-intersections, sharp angles, or intersections with the boundary of $[0,1]^d$ in the piece boundaries may prevent histogram approximations from correctly recovering the connectivity of the superlevel set filtration, even in the absence of noise. To address these additional challenges, we introduce an extra thickening step, similarly to the methods proposed in \cite{Shin2017} and Chapter 5 of \cite{KimThesis}.

\begin{figure}[h]
    \centering
    \includegraphics[width=5 cm]{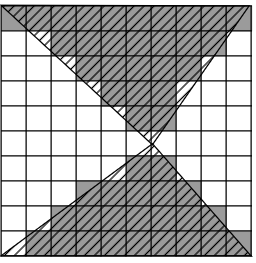}
    \caption{$\lambda-$superlevel cubical approximation for $f$ the function defined as $0$ outside the hatched area and $K$ inside (for arbitrarily large $K$). The histogram approximation fails to identify the connectivity of the two triangles for at least all $3K/4<\lambda<K$.}
    \label{problem mu reach H0}
\end{figure}

Let denote,
$$A^{b}=\left\{x\in \mathbb{R}^{d} \text{ s.t. } d_{\infty}\left(x,A\right)\leq b\right\}$$ 
with 
$$d_{\infty}\left(x,A\right)=\inf\limits_{y\in A}||x-y||_{\infty}.$$

Let $h>0$ such that $1/h$ is an integer, consider $G(h)$ the regular orthogonal grid over $[0,1]^{d}$ of step $h$ and $\mathfrak{C}_{h}$ the collection of all the closed hypercubes of side $h$ composing $G(h)$. We define,  for all $\lambda\in\mathbb{R}$, the $\lambda-$superlevel estimator,
$$\widehat{\mathcal{F}}_{\lambda}=\left(\bigcup\limits_{H\in \mathfrak{C}_{h,\lambda}}H\right)^{\lceil\sqrt{d}/\mu\rceil h} \text{ with }\mathfrak{C}_{h,\lambda}=\left\{H\in \mathfrak{C}_{h}\text{ such that }\frac{|X\cap H|}{nh^{d}}\geq \lambda\right\}$$
and $\lceil.\rceil$ the ceiling function. Let $\lambda>\lambda^{'}$, we denote,
$$\hat{v}_{\lambda}^{\lambda^{'}}:H_{0}\left(\widehat{\mathcal{F}}_{\lambda}\right)\rightarrow H_{0}\left(\widehat{\mathcal{F}}_{\lambda^{'}}\right)$$
the map induced by the inclusion $\widehat{\mathcal{F}}_{\lambda}\subset \widehat{\mathcal{F}}_{\lambda^{'}}$. Now, we introduce $\widehat{\mathbb{V}}_{f,0}$ the persistence module associated to $(H_{0}(\widehat{\mathcal{F}}_{\lambda}))_{\lambda\in \mathbb{R}}$ equipped with the collection of maps $(\widehat{v}_{\lambda,h}^{\lambda^{'}})_{\lambda>\lambda^{'}}$  and $\widehat{\operatorname{dgm}(f)}$ the associated $H_{0}-$persistence diagram. These diagrams are well-defined, as we prove in Appendix \ref{Appendix:q-tame} that $\widehat{\mathbb{V}}_{f,0}$ is $q-$tame.\\\\
\textbf{Calibration.} A natural question is how to choose the parameter $h$. Following the proof of Lemma \ref{lmm1b}, we choose $h$ such that :
\begin{equation}
\label{eq: calibration}
h^{\alpha}>\sqrt{\frac{\log\left(1/h^{d}\right)}{nh^{d}}}
\end{equation}
In particular, we can choose,
$$h\asymp\left(\frac{\log(n)}{n}\right)^{\frac{1}{d+2\alpha}}.$$
\textbf{Computation. }By construction, for all $\lambda\in\mathbb{R}$, $\widehat{\mathcal{F}}_{\lambda}$ is simply a union of cubes from the regular grid $G(h)$, thus, it can be thought as a (geometric realization of) a cubical complex or even a simplicial complex. Hence, it allows practical computation of its persistence diagram.\\\\
\textbf{Consistency.} We now present a key result that will play a central role in our proofs concerning mode inference: Proposition \ref{MainProp}, which quantifies, in probabilistic terms, the convergence rates achieved over the class $S_d(L,\alpha,\mu,R_{\mu})$.
\begin{prpstn}
\label{MainProp}
Let $h\asymp\left(\sfrac{\log(n)}{n}\right)^{\frac{1}{d+2\alpha}}$. There exists $\tilde{c}_{0}$ and $\tilde{c}_{1}$ such that, for all $A\geq 0$,
$$\sup\limits_{f\in S_{d}(L,\alpha,\mu,R_{\mu})}\mathbb{P}\left(d_{\infty}\left(\widehat{\operatorname{dgm}(f)},\operatorname{dgm}(f)\right)\geq A\left(\frac{\log(n)}{n}\right)^{\frac{\alpha}{d+2\alpha}}\right)\leq \tilde{c}_{0}\exp\left(-\Tilde{c}_{1}A^{2}\right).$$ 
\end{prpstn}
We believe that this result is also of independent interest, as it significantly extends the results of \cite{Henneuse24a} to a much broader class of signals. To complement this, we provide in Appendix \ref{sec: aux result} a corollary establishing consistency guarantees in expectation (Corollary \ref{estimation borne sup}), along with matching minimax lower bounds (Proposition \ref{Lower Bound Diag}), proving that the obtained rates are optimal.

These rates match those of Theorem~1 in \cite{Henneuse24a}, derived in the slightly different context of nonparametric regression. They also coincide with those from Corollary~4.4 in \cite{BCL2009}, which were established for nonparametric regression over Hölder spaces. This demonstrates that even when the signal is only piecewise Hölder with discontinuities loci having a positive $\mu-$reach, the $H_0$ persistence diagram can still be estimated at the same rates as in the fully Hölder-continuous setting.

The proof of Proposition~\ref{MainProp}, presented in Section~\ref{Appendix: Sup PD}, follows the same general strategy as the proof of Theorem~1 in \cite{Henneuse24a}, with several key refinements. In particular, under the weaker $\mu$-reach assumption, we lose some of the geometric properties provided by the reach assumption, which were extensively leveraged in our previous work.

\subsection{Multiple modes estimation}
\label{Section : modes estim}
In this section, we aim to derive, from the previous procedure, estimators for the number of modes of $f$, their locations, and the value of $f$ at the modes, i.e. infer the number of local maxima, their values and their locations. These local maxima simply correspond to the birth times in the $H_{0}-$persistence diagram of $f$. We show that removing points ``too close'' to the diagonal in $\widehat{\operatorname{dgm}(f)}$ gives a procedure that allows recovering consistently the number of modes and the associated values of $f$. Then, using this "regularized" persistence diagram, the estimation of the modes can be derived from the filtration $\hat{\mathcal{F}}$.
\label{Section : Sup modes 1}
Let $\{x_{1},...,x_{k}\}$ the set of modes of $f$ and $\{m_{1},...,m_{k}\}$ the associated maxima. In this setting, an estimator of the number of modes is given by,
$$\hat{k}=\left|\left\{(b,d)\in\widehat{\operatorname{dgm}\left(f\right)}, b-d>\delta/2\right\}\right|$$
and an estimator of the complete list of local maxima of $f$, by,
$$\hat{m}=\left\{\hat{m}_{1},\hat{m}_{2},...,\hat{m}_{\hat{k}}\right\}=\left\{b|(b,d)\in\widehat{\operatorname{dgm}\left(f\right)}, b-d>\delta/2\right\}.$$
Let $\mathcal{C}(\hat{b}_{i})$ the collection of connected components appearing in the filtration $\hat{\mathcal{F}}$ at time $\hat{b}_{i}$, we say that those connected components have associated birth time $\hat{b}_{i}$. We fix an order on its elements arbitrarily : $\tilde{C}_{i,1}$,...,$\tilde{C}_{i,p_{i}}$. Let $\tilde{C}_{i,p}\in \mathcal{C}(\hat{b}_{i})$, we define the associated death time of this connected component to be the smallest $d$ such that $\tilde{C}_{i,p}$ is not connected to any elements of $\mathcal{C}(\hat{b}_{j})$ for all $\hat{b}_{j}>\hat{b}_{i}$ or any $\tilde{C}_{i,q}$ with $q<p$ in $\hat{\mathcal{F}}_{d}$. With this definition, we ensure that there exists a unique collection $\{\hat{C}_{i}\in \mathcal{C}(\hat{m}_{i})\}_{1\leq i\leq \hat{k}}$ such that, for all $1\leq i\leq \hat{k}$, the death time $\hat{d}_{i}$ associated to $\hat{C}_{i}$ satisfies $\hat{m}_{i}-\hat{d}_{i}>\delta/2$. Furthermore $\hat{C}_{i}\cap \hat{C}_{j}\ne\emptyset$, for all $i\ne j$. We denote the obtained collection of disjoint subsets :
$$\hat{C}=\left\{\hat{C}_{1},...,\hat{C}_{\hat{k}}\right\}.$$
Then, we can estimate the modes by taking, for all $i\in\{1,...,\hat{k}\}$, any $\hat{x}_{i}$ in $\hat{C}_{i}$ and define the collection of estimated modes by :
$$\hat{x}=\left\{\hat{x}_{1},...,\hat{x}_{\hat{k}}\right\}.$$
Although $\hat{C}$ and $\hat{x}$ depend on the choice of ordering on the $\mathcal{C}(\hat{b}_{i})$, we show that any such constructed collection provides a consistent estimator of the modes of $f$. We show as well that $\hat{m}$ is a consistent estimator of the associated local maxima. The convergence rates achieved by these estimators are given by the following theorem.
\begin{thrm}
\label{coroModes1}
Let $f\in S_{d}(L,\alpha,\mu,R_{\mu},C,h_{0})$, $h\asymp\left(\sfrac{\log(n)}{n}\right)^{\frac{1}{d+2\alpha}}$ and suppose $h_0\geq 2(\sqrt{d}+\lceil\sqrt{d}/\mu\rceil)h$. There exist $\check{c}_{0}$, $\check{c}_{1}$, $\check{c}_{2}$ and $\check{c}_{3}$ depending only on $L$ ,$\alpha$, $\mu$, $R_{\mu}$, and $C$ such that, for all $A\geq 0$ with probability at least $1-\check{c}_{0}\exp(-\check{c}_{1}A^{2})-\check{c}_{2}\exp(-\check{c}_{3}(h_{0}/h)^{2\alpha})$,
$$k=\hat{k}$$
and for all $i\in\{1,...,k\}$, there exists distinct $(\hat{x}_{i},\hat{m}_{i})\in \hat{x}\times \hat{m}$ such that,
$$\|\hat{x}_{i}-x_{i}\|_{\infty}\leq A h$$
and 
$$|\hat{m}_{i}-m_{i}|\leq  A h^{\alpha}.$$
\end{thrm}
Several conclusions can be drawn from this theorem. First, for fixed parameters $L$, $\alpha$, $\mu$, $R_{\mu}$, $C$ and $h_0$, Theorem \ref{coroModes1} implies that, as $\lim_{n \to \infty} \frac{h_0}{h} = +\infty$, for sufficiently large $n$, we can infer with high probability the exact number of modes, locate them at a rate of $O\left(\left(\sfrac{\log(n)}{n}\right)^{\frac{1}{2\alpha+d}}\right)$, and estimate the associated local maxima at a rate of $O\left(\left(\sfrac{\log(n)}{n}\right)^{\frac{\alpha}{2\alpha+d}}\right)$.
Secondly, this result also addresses a perhaps more interesting question: how far apart or how prominent must the modes be for our procedure to identify them (and their associated maxima) with high probability ? Equivalently, how large should $h_0$ be, or how large should $\delta$ be, in comparison to $n$ ? In this regard, Theorem \ref{coroModes1} indicates that the critical value of $h_0$ for our procedure to succeed is of the order $h \asymp \left(\sfrac{\log(n)}{n}\right)^{\frac{1}{2\alpha+d}}$. This critical $h_0$ can be interpreted as a minimal separation distance between modes that allows for their identification by our procedure. In particular, if $h_{0}\geq B h$, with $B>2(\sqrt{d}+\lceil\sqrt{d}/\mu\rceil)$, we have, with probability at least 
$1-\check{c}_{0}\exp(-\check{c}_{1}A^{2})-\check{c}_{2}\exp(-\check{c}_{3}B^{2\alpha})$, 
$$\hat{k}=k\text{ and, for all $1\leq i\leq k$, }||\hat{x}_i-x_i||_{\infty}\leq Ah\text{ and }|\hat{m}_i-m_i|\leq Ah^{\alpha}.$$
This also gives that, if $h=o(h_{0})$ (and other parameters are fixed) :
$$\lim_{n\rightarrow+\infty}\mathbb{P}\left(\left\{\hat{k}=k\right\}\cap \left(\bigcap_{i=1}^{k}\left(\{||\hat{x}_i-x_i||_{\infty}\leq Ah\}\cap \{|\hat{m}_i-m_i|\leq Ah\}\right)\right)\right)\geq 1-\check{c}_{0}\exp(-\check{c}_{1}A^{2}).$$
 Note that, as previoulsy emphasized, due to assumption $\textbf{A4}$ (recalling that $\delta = C h_0^{\alpha}$), this condition can be equivalently reformulated in terms of a critical prominence $\delta$. In this case, it implies a critical prominence of order $ h^{\alpha} \asymp \left(\sfrac{\log(n)}{n}\right)^{\frac{\alpha}{2\alpha+d}}$. It can also be interpreted in terms of the maximal number of modes, which is then of order $\asymp \left(\sfrac{n}{\log(n)}\right)^{\frac{d}{2\alpha + d}}$.

In addition to the above result, one may also wonder about the more realistic setting where $\delta$ is unknown. To address this, we propose in Appendix \ref{apdx:complement modes} a penalization procedure for selecting $\delta$, and adapt our method accordingly. Under slightly stronger threshold assumption, requiring that $h_{0}>>\left(\sfrac{n}{\log(n)}\right)^{\frac{1}{2(2\alpha + d)}}$ or equivalently $\delta>>\left(\sfrac{n}{\log(n)}\right)^{\frac{\alpha}{2(2\alpha + d)}}$, we show that, with this selected value, the number of modes can still be recovered with high probability, and both the locations of the modes and their associated local maxima can be estimated at the same rates as in Theorem \ref{coroModes1}.\\\\
In light of Theorem \ref{coroModes1}, several questions arise regarding its optimality: first, the optimality in terms of the convergence rates for modes locations and their associated local maxima; second, the optimality concerning the tightness of the critical $h_0$.

Regarding the convergence rates, it is noteworthy that, up to a logarithmic factor, the rates we have obtained align with the minimax rates for single-mode estimation established in \cite{ariascastro2021estimation}. By adapting their proof, we show in Theorem \ref{lowerboundsModes} that similar lower bounds also hold in our setting. Moreover, we refine this analysis further by recovering an additional \(\log(1/h_0)\) factor.
\begin{thrm}
\label{lowerboundsModes}
For sufficiently small $h_0$, there are two strictly positive constants $A$ and $B$ and a family of densities of $S_{d}(L,\alpha,\mu,R_{\mu},C,h_{0})$ with (pairwise) modes separated by $A(\log(1/h_0)/n)^{\frac{1}{d+2\alpha}}$ and local maxima separated by $B(\log(1/h_0)/n)^{\frac{\alpha}{d+2\alpha}}$ that cannot be distinguished with more than probability $1/5$ for a sample size of $n$.
\end{thrm}
In the case where the parameters $L$, $\alpha$, $\mu$, $R_{\mu}$, $C$ and $h_0$ are fixed and \( n \) tends to \( +\infty \), our result does not provide any additional insight beyond Theorem 2 of \cite{ariascastro2021estimation}. This confirms that, in this setting, our proposed procedure is minimax (up to logarithmic factors). However, when \( h_0 \) depends on \( n \), our result establishes a slightly stronger lower bound. Specifically, if \( h_0 \asymp (1/n)^{\gamma} \) for some \( \gamma > 0 \), it implies that the modes cannot be estimated at a rate faster than \( O((\log n / n)^{\frac{1}{d+2\alpha}}) \), and that the corresponding maxima cannot be estimated at a rate faster than \( O((\log n / n)^{\frac{\alpha}{d+2\alpha}}) \). In particular, if \( \gamma > \frac{1}{d+2\alpha} \), Theorem \ref{lowerboundsModes} establishes the sharp minimaxity of our procedure, fully accounting for the logarithmic factors appearing in Theorem \ref{coroModes1}. \\\\
Now, regarding the sharpness of the critical threshold \( h_0 \), a natural question arises: can we still recover some information about the modes below this threshold, albeit potentially at slower rates ? Proposition~\ref{prp: tight separation} shows that this is essentially impossible. More precisely, even if we only aim to recover the number of modes, disregarding their locations, Proposition~\ref{prp: tight separation} establishes that, below the threshold (up to logarithmic factors), this task is infeasible. In particular, we construct, in this case, a bimodal and an unimodal distribution that are statistically indistinguishable with high probability.

Moreover, even if we weaken our goal and attempt only to estimate the modes as a set of points (without identifying them individually), consistency remains unattainable. Specifically, considering the modes as compact subsets of \( \mathbb{R}^d \) and measuring their discrepancy via the Hausdorff distance
\[
d_H(A,B) = \max\left( \sup_{x \in B} \inf_{y \in A} \|x - y\|_2, \sup_{x \in A} \inf_{y \in B} \|x - y\|_2 \right),
\]
Proposition~\ref{prp: tight separation} still implies that, when \( h_0 \) falls below the critical threshold, consistent recovery is impossible even with respect to this weaker metric. In particular, the two densities we construct have sets of modes that are well-separated in Hausdorff distance, yet they remain statistically indistinguishable with high probability.
\begin{prpstn}
\label{prp: tight separation}
There exist strictly positive constants $A$ and $B$, a bimodal density $f_0\in S_{d}(L,\alpha,\mu,R_{\mu},C,h_{0})$ and an unimodal density $f_1\in S_{d}(L,\alpha,\mu,R_{\mu},C,h_{0})$, with $h_0= n^{\frac{-1}{2\alpha+d}}/A$ such that the sets of modes of $f_0$ and $f_1$ are at Hausdorff distance $B$ and $f_0$ cannot be distinguished from $f_1$ with more than probability 1/5 for a sample size n.
\end{prpstn}
Hence, this proposition tells us that, up to a logarithmic factor, the critical $h_0$ obtained in Theorem \ref{coroModes1} is optimal, as if $h_0$ is smaller, even estimating the number of modes (without considering their locations or the associated maxima) or estimating the sets of modes in Hausdorff distance becomes hopeless.
\section{Proofs of main results}
\label{sec: proofs}
This section is devoted to the proofs of Proposition \ref{MainProp}, Theorem \ref{coroModes1}, Theorem \ref{lowerboundsModes} and Proposition \ref{prp: tight separation}.
\subsection{Proof of Proposition \ref{MainProp}}
\label{Appendix: Sup PD}
This section is dedicated to the proof of Proposition~\ref{MainProp}. The strategy is to construct an interleaving between the persistence modules $\mathbb{V}_{f,0}$ and $\widehat{\mathbb{V}}_{f,0}$, allowing us to invoke the algebraic stability theorem~\cite{Chazal2009}. To this end, we construct two morphisms, $\overline{\psi}:\mathbb{V}_{f,0} \rightarrow \widehat{\mathbb{V}}_{f,0}$ and $\overline{\phi}:\widehat{\mathbb{V}}_{f,0} \rightarrow \mathbb{V}_{f,0}$, satisfying the interleaving condition~\eqref{interleaving def}. The overall construction follows the blueprint used in the proof of Theorem~1 in~\cite{Henneuse24a}, but it requires substantial modifications. In particular, the previous approach relied heavily on the positive reach assumption, which guaranteed the existence of a continuous projection map near the boundaries of the regular pieces. In the more general setting considered here, this assumption no longer holds, necessitating several non-trivial adaptations. To address this, we rely on key geometric properties of sets with positive $\mu$-reach, as established in Lemma~3.1 of~\cite{SSDO07} and Theorem~12 of~\cite{kim2020homotopy}. These results allow us to circumvent the lack of a projection map near the piece boundaries. The remainder of this section details the construction. Note that several lemmas we introduce here will also be reused later to prove Theorem \ref{coroModes1}.
\subsubsection{ Ingredient 1 : inclusions between level sets}
If there existed an $\varepsilon>0$ such that for all $\lambda\in \mathbb{R}$, $\mathcal{F}_{\lambda}\subset \widehat{\mathcal{F}}_{\lambda-\varepsilon}\subset \mathcal{F}_{\lambda-2\varepsilon}$, an $\varepsilon-$interleaving would be directly given taking the morphisms induced by inclusions. Here we do not have such nice inclusions, but we show, in Lemma \ref{interleavish 3}, a slightly weaker double inclusion.\\\\
First, note that the $f\in S_{d}\left(L,\alpha,\mu,R_{\mu}\right)$ are uniformly bounded, as stated in the following lemma, whose proof can be found in Appendix \ref{Appendix: Lemmbound}. 
\begin{lmm}
    \label{lemmBound}
    Let $f\in S_{d}\left(L,\alpha,\mu,R_{\mu}\right)$, there exists a constant $\kappa=\kappa(d,L,\alpha,R_{\mu})$ depending only on $d$, $L$, $\alpha$ and $R_{\mu}$, such that $\|f\|_{\infty}\leq \kappa$.
\end{lmm}
 And denote, $N_{h}$ the random variable defined by,
$$N_{h}=\frac{\max\limits_{H\in \mathfrak{C}_{h}}\left|\frac{|X\cap H|}{nh^{d}}-\mathbb{E}\left[\frac{|X\cap H|}{nh^{d}}\right]\right|}{\sqrt{3\kappa\frac{\log\left(1/h^{d}\right)}{nh^{d}}}}$$
\begin{lmm}
\label{interleavish 3}
Let $f\in S_{d}\left(L,\alpha,\mu,R_{\mu}\right)$ and $h\asymp\left(\sfrac{\log(n)}{n}\right)^{\frac{1}{d+2\alpha}}$. For $n$ sufficiently large such that $h<\mu R_{\mu}/\sqrt{d}$, for all $\lambda\in\mathbb{R}$, we have,
$$\mathcal{F}_{\lambda+\left(2\sqrt{3\kappa}N_{h}+L\left(\lceil\sqrt{d}/\mu\rceil+\sqrt{d}\right)^{\alpha}\right)h^{\alpha}}\subset \widehat{\mathcal{F}}_{\lambda}\subset \mathcal{F}_{\lambda-2\sqrt{3\kappa}N_{h}h^{\alpha}}^{(\sqrt{d}+\lceil\sqrt{d}/\mu\rceil)h}.$$
\end{lmm}
This proposition follows from Lemma 3.1 from \cite{SSDO07} and the following lemma, whose proof can be found in Appendix \ref{Appendix : lmm1b}.
\begin{lmm}
\label{lmm1b}
Let $f\in S_{d}(L,\alpha,\mu,R_{\mu})$ and $h>0$ satisfying (\ref{eq: calibration}). Let 
$$H\subset \mathcal{F}_{\lambda-2\sqrt{3\kappa}N_{h}h^{\alpha}}^{c}\cap \mathfrak{C}_{h}$$
and $$H^{'}\subset\mathcal{F}_{\lambda+2\sqrt{3\kappa}N_{h}h^{\alpha}}\cap \mathfrak{C}_{h}.$$
We then have that,
$$\frac{|X\cap H|}{nh^{d}}< \lambda\text{ and }\frac{|X\cap H^{'}|}{nh^{d}}> \lambda.$$
\end{lmm}
For a set $K\subset[0,1]^{d}$ and $r\geq 0$ we denote $B_{2}(K,r)=\{x\in[0,1]^{d}\text{ s.t. }d_{K}(x)<r\}$ and $\overline{B}_{2}(K,r)=\{x\in[0,1]^{d}\text{ s.t. }d_{K}(x)\leq r\}$.
\begin{lmm}{\cite[Lemma 3.1.]{SSDO07}}
\label{lmm SSDO7}
Let $K\subset [0,1]^{d}$ a compact set and let $\mu>0, r>0$ be such that $r< \operatorname{reach}_\mu(K)$. For any $x \in \overline{B}_{2}(K,r) \backslash K$, we have,
$$
d_{2}\left(x, \partial \overline{B}_{2}(K,r)\right) \leq \frac{r-d_K(x)}{\mu} \leq \frac{r}{\mu}.
$$
\end{lmm}
\begin{proof}[Proof of Lemma \ref{interleavish 3}]
We begin by proving the lower inclusion, let,
$$x\in \mathcal{F}_{\lambda+\left(2\sqrt{3\kappa}N_{h}+L\left(\lceil\sqrt{d}/\mu\rceil+\sqrt{d}\right)^{\alpha}\right)h^{\alpha}}.$$
Without loss of generality, let us suppose $x\in M_{i}$, or $x\in \partial M_{i}$ and $\limsup_{z\in M_{i}\rightarrow x}f(z)\geq f(x)$. If,
$$B_{2}\left(x,\sqrt{d}h\right)\subset\left(\bigcup\limits_{i=1}^{l}\partial M_{i}\right)^{c}$$
then, $H_{x,h}$, the hypercube of $\mathfrak{C}_{h,\lambda}$ containing $x$ is included in $\overline{M}_{i}$. Assumption \textbf{A1} implies that for all $y\in B_2(x,\sqrt{d}h)\cap M_i$ we have :
$$|f(x)-f(y)|\leq L|||x-y||^{\alpha}\leq L(\sqrt{d}h)^{\alpha},$$
thus, $y\in \mathcal{F}_{\lambda+2\sqrt{3\kappa}N_{h}h^{\alpha}}$. Now, for all $y\in B_2(x,\sqrt{d}h)\cap \overline{M}_i$, there exists a sequence $(y_{n})_{n\in\mathbb{N}}$ of points of $B_2(x,\sqrt{d}h)\cap M_i$ such that $\lim_{n\rightarrow\infty}y_{n}=y$. As, for all $n\in \mathbb{N}$, we have shown that $y_{n}\in \mathcal{F}_{\lambda+2\sqrt{3\kappa}N_{h}h^{\alpha}}$, by Assumption \textbf{A2}, it follows that :
$$f(y)\geq \limsup_{n\rightarrow+\infty} f(y_n)\geq \lambda+2\sqrt{3\kappa}N_{h}h^{\alpha}.$$
Hence $y\in \mathcal{F}_{\lambda+2\sqrt{3\kappa}N_{h}h^{\alpha}}$. Thus, as $H_{x,h}\subset B_2(x,\sqrt{d}h)\cap \overline{M}_i$, the combination of assumption \textbf{A1} and \textbf{A2} ensures $H_{x,h}\subset \mathcal{F}_{\lambda+2\sqrt{3\kappa}N_{h}h^{\alpha}}$. Hence, it follows from Lemma \ref{lmm1b} that $H_{x,h}\in \mathfrak{C}_{h,\lambda}$ and consequently $x\in \widehat{\mathcal{F}}_{\lambda}$. Else,  as $\sqrt{d}h/\mu<R_{\mu}$, under Assumption \textbf{A3}, by Lemma \ref{lmm SSDO7}, there exists,
$$y\in \left(B_{2}\left(\bigcup\limits_{i=1}^{l}\partial M_{i},{\sqrt{d}h}\right)\right)^{c}\cap M_{i}  \text{ such that } ||x-y||_{2}\leq \sqrt{d}h/\mu.$$
Let $H_{y,h}$ the closed hypercube of $\mathfrak{C}_{h}$ containing $y$. Hence, $H_{y,h}\subset \overline{M_{i}}.$ and, for all $z\in H_{y,h}$, $||z-x||_{2}\leq \sqrt{d}h(1+1/\mu)$.
Then, again, Assumption \textbf{A1} and \textbf{A2} ensure that,
$$H_{y,h}\subset \mathcal{F}_{\lambda+L+2\sqrt{3\kappa}N_{h}h^{\alpha}}.$$
Then, Lemma \ref{lmm1b} gives $H_{y,h}\in \mathfrak{C}_{h,\lambda}$ and thus, as $x\in H_{y,h}^{\sqrt{d}h/\mu}$, $x\in \widehat{\mathcal{F}}_{\lambda}$, which proves the lower inclusion.\\\\
For the upper inclusion, let,
$$x\in\left(\mathcal{F}_{\lambda-2\sqrt{3\kappa}N_{h}h^{\alpha}}^{\sqrt{d}h}\right)^{c}$$
and $H_{x,h}$ the hypercube of $\mathfrak{C}_{h}$ containing $x$. We then have, $H_{x,h}\subset \smash{\mathcal{F}_{\lambda-2\sqrt{3\kappa}N_{h}h^{\alpha}}^{c}}$. Hence, Lemma \ref{lmm1b} gives that,
$$H_{x,h}\subset \left(\bigcup\limits_{H\in \mathfrak{C}_{h,\lambda}}H\right)^{c}$$
and thus,
$$\bigcup\limits_{H\in \mathfrak{C}_{h,\lambda}}H\subset \mathcal{F}_{\lambda-2\sqrt{3\kappa}N_{h}h^{\alpha}}^{\sqrt{d}h}.$$
Consequently,
$$\left(\bigcup\limits_{H\in \mathfrak{C}_{h,\lambda}}H\right)^{\lceil\sqrt{d}/\mu\rceil h}=\widehat{\mathcal{F}}_{\lambda}\subset \mathcal{F}_{\lambda-2\sqrt{3\kappa}N_{h}h^{\alpha}}^{(\sqrt{d}+\lceil\sqrt{d}/\mu\rceil)h}$$
and the proof is complete.
\end{proof}
\subsubsection{Ingredient 2 : geometry of the thickened level sets}
The lower inclusion of Lemma \ref{interleavish 3} gives directly a morphism $\overline{\psi}:\mathbb{V}_{f,0}\rightarrow\widehat{\mathbb{V}}_{f,0}$. But the upper inclusion is not sufficient to provide similarly $\overline{\phi}:\widehat{\mathbb{V}}_{f,0}\rightarrow\mathbb{V}_{f,0}$. To overcome this issue, in Lemma \ref{satisfyC2b}, we exploit assumption \textbf{A3} to construct morphisms from $(H_{0}(\mathcal{F}_{\lambda}^{r}))_{\lambda\in\mathbb{R}}$ into $V_{f,0}$, for all $r<R_{\mu}/\sqrt{d}$. For $r= (\sqrt{d}+\lceil\sqrt{d}/\mu\rceil)h$, composing this morphism with the one induced by the upper inclusion of Lemma \ref{interleavish 3} will give us our morphism $\overline{\phi}$.
\begin{lmm}
\label{satisfyC2b}
    Let $f\in S_{d}(L,\alpha,\mu,R_{\mu})$. For all $ h< R_{\mu}/\sqrt{d}$, there exists a morphism $\phi$ such that, for all $\lambda\in\mathbb{R}$,
    \begin{equation}
    \label{diagC2}
        \begin{tikzcd}
	{H_{0}\left(\mathcal{F}_{\lambda}\right)} && {H_{0}\left(\mathcal{F}_{\lambda-L\left(\sqrt{d}(2+2/\mu^{2})\right)^{\alpha}h^{\alpha}}\right)} \\
	& {H_{0}\left(\mathcal{F}_{\lambda}^{h}\right)}
	\arrow[from=1-1, to=1-3]
	\arrow[from=1-1, to=2-2]
	\arrow["{\phi_{\lambda}}"', from=2-2, to=1-3]
    \end{tikzcd}
    \end{equation}
    is a commutative diagram (unspecified map come from set inclusions).
\end{lmm}
To prove this proposition, we use the following lemma. This result involves the notion of deformation retract that we recall here.
\begin{dfntn}
A subspace $A$ of $X$ is called a \textbf{deformation retract} of $X$ if there is a continuous $F: X \times [0,1] \rightarrow X$ such that for all $x \in X$ and $a \in A$,
\begin{itemize}
    \item $F(x, 0)=x$ 
    \item $F(x, 1) \in A$
    \item $F(a, 1)=a$.
\end{itemize}
The function $x\mapsto F(x,1)$ is called a (deformation) \textbf{retraction} from $X$ to $A$.
\end{dfntn}
\begin{lmm}{\citep[][Theorem 12]{kim2020homotopy}}
\label{deformation retract}
Let $K\subset[0,1]^{d}$ a compact set, for all $0\leq r<\operatorname{reach}_{\mu}(K)$,
$\overline{B}_{2}(K,r)$ retracts by deformation onto $K$ and the associated retraction $R: \overline{B}_{2}(K,r)\rightarrow K $ satisfies for all $x\in \overline{B}_{2}(K,r)$, $R(x)\in \overline{B}_{2}(x, 2r/\mu^{2})\cap K$.
\end{lmm}
The first part of the claim is Theorem 12 from \cite{kim2020homotopy} and the second part follows easily from their construction. Elements of proof can be found in Appendix \ref{appendix : def retract}.
\begin{proof}[Proof of Lemma \ref{satisfyC2b}]
Any connected component $B$ of $\mathcal{F}_{\lambda}^{h}$ contains at least a connected component $A$ of $\mathcal{F}_{\lambda}$. Suppose that $B$ contains $A$ and $A^{'}$ two disjoint connected components of $\mathcal{F}_{\lambda}$, then there exist $x\in A$ and $y\in A^{'}$ such that $||x-y||_{\infty}\leq 2h$. Suppose that $x\in \overline{M}_{i}$, $y\in \overline{M}_{j}$, $\lim_{z\in M_{i}\rightarrow x}f(z)=f(x)$ and $\lim_{z\in M_{j}\rightarrow y}f(z)=f(y)$, $i,j\in\{1,...,l\}$.\\\\
Let,
$$R:\overline{B}_{2}(\partial M_{i}\cup \partial M_{j},\sqrt{d}h)\rightarrow \partial M_{i}\cup \partial M_{j}$$
the deformation retraction of $\overline{B}_{2}(\partial M_{i}\cup \partial M_{j},\sqrt{d}h)$ onto $\partial M_{i}\cup \partial M_{j}$ given by Lemma \ref{deformation retract} under Assumption \textbf{A3}. We use this deformation retraction to construct a continuous path from $x$ to $y$ included in $\overline{M}_{i}\cup \overline{M}_{j}$. To do so, consider, for all $z\in[x,y]$ the function :
$$\gamma(z)=\begin{cases} z&\textit{if } z\in \overline{M}_{i}\cup \overline{M}_{j}\\
R(z)&\text{else}
\end{cases}$$
Observe that $\gamma(x)=x$, $\gamma(y)=y$ and, as $R$ is continuous and $R(z)=z$ for all $z\in \partial M_i\cup \partial M_j$, $\gamma$ is a continuous path between $x$ and $y$. Furthermore, as $R(z)\subset \partial M_i\cup \partial M_j$ for all $z\in \overline{B}_{2}(\partial M_{i}\cup \partial M_{j},\sqrt{d}h)$, we have that $\gamma([x,y])$ is a continuous path between $x$ and $y$ included in $\overline{M}_{i}\cup \overline{M}_{j}$. Now, by Lemma \ref{deformation retract} and as $||x-y||_{\infty}\leq 2h$, we know that, for all $z\in [x,y]\setminus(\overline{M}_{i}\cup \overline{M}_{j})$,
$$R(z)\in \overline{B}_{2}(z,2\sqrt{d}h/\mu^{2})\subset  \overline{B}_{2}(x,\sqrt{d}h(2+2/\mu^{2})).$$
and, for all $z\in [x,y]\cap(\overline{M}_{i}\cup \overline{M}_{j})$,
$$z\in \overline{B}_{2}(x,2\sqrt{d}h).$$
Thus, 
$$\gamma([x,y])\subset \overline{B}_{2}(x,\sqrt{d}h(2+2/\mu^{2}))\cap (\overline{M}_{i}\cup \overline{M}_{j}).$$
Assumptions \textbf{A1} and \textbf{A2} then ensure that $\gamma([x,y])\subset \mathcal{F}_{\lambda-L((2+2/\mu^{2})\sqrt{d}h)^{\alpha}}$. Consequently, there exists a continuous path between $A$ and $A^{'}$ in $\mathcal{F}_{\lambda-L((2+2/\mu^{2})\sqrt{d}h)^{\alpha}}$, and hence, $A$ and $A^{'}$ are connected in $\mathcal{F}_{\lambda-L((2+2/\mu^{2})\sqrt{d}h)^{\alpha}}$.\\\\
We can then properly define,
$$\left\{\begin{array}{ll}
\phi_{\lambda}:H_{0}(\mathcal{F}_{\lambda}^{h})&\rightarrow H_{0}(\mathcal{F}_{\lambda-L((2+2/\mu^{2})\sqrt{d})^{\alpha}h^{\alpha}})\\
\quad\quad[x]&\longmapsto[y]
\end{array}\right.$$
with $y$ belonging to $A$ any connected component of $\mathcal{F}_{\lambda}$ included in $B$, the connected component of $\mathcal{F}_{\lambda}^{h}$ containing $x$. By the foregoing, $[y]$ is independent of the choice of such $A$ and such $y$ in $A$ and Diagram \ref{diagC2} commutes.
\end{proof}
\subsubsection{Ingredient 3 : concentration of $N_{h}$}
 The two previous ingredients will permit to establish an interleaving between $\widehat{\mathbb{V}}_{f,0}$ and $\mathbb{V}_{f,0}$. This interleaving will depend on the variable $N_{h}$. Consequently, to derive convergence rates from this interleaving, we need concentration inequalities over $N_{h}$.
\begin{lmm}
\label{noiseconcentration}
   For all $h>0$, 
    $$\mathbb{P}\left(\bigcup\limits_{H\in \mathfrak{C}_{h}}\left|\frac{|X\cap H|}{nh^{d}}-\mathbb{E}\left[\frac{|X\cap H|}{nh^{d}}\right]\right|\geq t\right)\leq \frac{2}{h^{d}}\exp\left(-\frac{t^{2}nh^{d}}{3\kappa}\right).$$
\end{lmm}
\begin{proof}[Proof of Lemma \ref{noiseconcentration}]
Let $h>0$ and $H\in \mathfrak{C}_{h}$. We can write, $$|X\cap H|=\sum\limits_{i=1}^{n}Y_{i}^{H}$$
with, for all $i\in\{1,...,n\}$,
$$Y_{i}^{H}=\mathbb{1}_{X_{i}\in H}$$
As $X_{1},...,X_{n}$ are i.i.d., $Y_{1}^{H},...,Y_{n}^{H} $ are i.i.d. Bernoulli variable, and $|X\cap H|$ is then a binomial variable of parameters $\left(n, \mathbb{P}(X_{1}\in H)\right)$. It follows from the Chernoff bound for binomial distribution that,
$$\mathbb{P}\left(\left|\frac{|X\cap H|}{nh^{d}}-\mathbb{E}\left[\frac{|X\cap H|}{nh^{d}}\right]\right|\geq t\right)\leq 2\exp\left(-\frac{t^{2}nh^{2d}}{3\mathbb{P}(X_{1}\in H)}\right).$$
Lemma \ref{lemmBound} then gives that $\mathbb{P}(X_{1}\in H)\leq h^{d}\kappa$, hence,
$$\mathbb{P}\left(\left|\frac{|X\cap H|}{nh^{d}}-\mathbb{E}\left[\frac{|X\cap H|}{nh^{d}}\right]\right|\geq t\right)\leq 2\exp\left(-\frac{t^{2}nh^{d}}{3\kappa}\right).$$
By union bounds, it follows,
\begin{align*}
 \mathbb{P}\left(\bigcup\limits_{H\in \mathfrak{C}_{h}}\left\{\left|\frac{|X\cap H|}{nh^{d}}-\mathbb{E}\left[\frac{|X\cap H|}{nh^{d}}\right]\right|\geq t\right\}\right)&\leq 2\left|\mathfrak{C}_{h}\right|\exp\left(-\frac{t^{2}nh^{d}}{3\kappa}\right)\\
 &\leq \frac{2}{h^{d}}\exp\left(-\frac{t^{2}nh^{d}}{3\kappa}\right).   
\end{align*}
\end{proof}
Lemma \ref{noiseconcentration} implies that $N_{h}$ is sub Gaussian. More precisely, there exist two constants $c_{0}$ and $c_{1}$, depending only on $d$, $L$, $\alpha$ and $R_{\mu}$, such that, for all $h\leq 1$,
$$\mathbb{P}\left(N_{h}>t\right)\leq c_{0}\exp\left(-c_{1}t^{2}\right).$$
\subsubsection{Assembling the ingredients}
Equipped with Lemmas \ref{interleavish 3}, \ref{satisfyC2b} and \ref{noiseconcentration} we can now proceed to the proof of Proposition \ref{MainProp}.
\begin{proof}[Proof of Proposition \ref{MainProp}]
It suffices to show the result for (arbitrarily) large $n$ (up to rescaling $\Tilde{c_{0}}$), thus suppose that $n$ is sufficiently large for the application of Lemmas \ref{interleavish 3} and \ref{satisfyC2b} used in this proof. We denote $k_{1}=\sqrt{d}+\lceil\sqrt{d}/\mu\rceil$ and $k_{2}=\sqrt{d}\left(2+2/\mu^{2}\right)(\sqrt{d}+\lceil\sqrt{d}/\mu\rceil)$.\\\\
First, we construct $\overline{\phi}$. Let $\lambda\in\mathbb{R}$, Lemma \ref{lmm1b} imply that any connected component $A$ of $\widehat{\mathcal{F}}_{\lambda}$ intersects a connected component $B$ of $\mathcal{F}_{\lambda-2\sqrt{3\kappa}N_{h}h^{\alpha}}$. Now suppose that $A$ intersects 2 such components $B_{1}$ and $B_{2}$, then, by Lemma \ref{interleavish 3}, $B_{1}$ and $B_{2}$ are connected in $\mathcal{F}_{\lambda-2\sqrt{3\kappa}N_{h}h^{\alpha}}^{k_{1}h}$. Thus, as a consequence of Lemma \ref{satisfyC2b} (applied for $h^{'}=k_1h$, which requires $h<R_\mu/k_1\sqrt{d}$), $B_{1}$ and $B_{2}$ are connected in $\mathcal{F}_{\lambda-\left(2\sqrt{3\kappa}N_{h}+Lk_{2}^{\alpha}\right)h^{\alpha}}$. We can then define properly the applications, 
$$
\left\{
    \begin{array}{ll}
\overline{\phi}_{\lambda}:H_{0}\left(\widehat{\mathcal{F}}_{\lambda}\right)&\longrightarrow H_{0}\left(\mathcal{F}_{\lambda-\left(2\sqrt{3\kappa}N_{h}+
Lk_{2}^{\alpha}\right)h^{\alpha}}\right)\\
\quad\quad[x]&\longmapsto[y]
    \end{array}
\right.
$$
with $y$ belonging to $B$ a connected component of $\mathcal{F}_{\lambda-2\sqrt{3\kappa}N_{h}h^{\alpha}}$ intersecting $A$, the connected component of $\widehat{\mathcal{F}}_{\lambda}$ containing $x$. The previous remark ensures that $[y]$ is independent of the choice of $B$ and $y$ in $B$.   \\\\
Now, for the construction of $\overline{\psi}$, by Lemma \ref{interleavish 3}, we have,
$$\mathcal{F}_{\lambda}\subset \widehat{\mathcal{F}}_{\lambda-\left(2\sqrt{3\kappa}N_{h}+Lk_{1}^{\alpha}\right)h^{\alpha}}$$
and we simply take $\overline{\phi}$ the morphism induced by this inclusion, i.e,
$$
\left\{
\begin{array}{ll}
\overline{\psi}_{\lambda}:H_{0}\left(\mathcal{F}_{\lambda}\right)&\longrightarrow H_{0}\left(\widehat{\mathcal{F}}_{\lambda-\left(2\sqrt{3\kappa}N_{h}+Lk_{1}^{\alpha}\right)h^{\alpha}}\right)\\
\quad\quad[x]&\longmapsto[x]
    \end{array}
\right.
$$
Denote $K_{1}=2\sqrt{3\kappa}N_{h}+Lk_{1}^{\alpha}$ and $K_{2}=2\sqrt{3\kappa}N_{h}+
Lk_{2}^{\alpha}$. The foregoing ensures the commutativity of the following diagrams :
\begin{equation*}
\begin{tikzcd}
	{H_{0}\left(\mathcal{F}_{\lambda}\right)} && {H_{0}\left(\mathcal{F}_{\lambda^{'}}\right)} \\
	\\
	{H_{0}\left(\widehat{\mathcal{F}}_{\lambda-K_{1}h^{\alpha}}\right)} && {H_{0}\left(\widehat{\mathcal{F}}_{\lambda^{'}-K_{1}h^{\alpha}}\right)}
	\arrow[from=1-1, to=1-3, "v_{\lambda}^{\lambda^{'}}"]
	\arrow[from=1-1, to=3-1,"\overline{\psi}_{\lambda}"]
	\arrow[from=1-3, to=3-3,"\overline{\psi}_{\lambda^{'}}"]
	\arrow[from=3-1, to=3-3, "\hat{v}_{\lambda-K_1h^{\alpha}}^{\lambda^{'}-K_1h^{\alpha}}"]
\end{tikzcd}
\end{equation*}
\begin{equation*}
\begin{tikzcd}
	{H_{0}\left(\widehat{\mathcal{F}}_{\lambda}\right)} && {H_{0}\left(\widehat{\mathcal{F}}_{\lambda^{'}}\right)} \\
	\\
	{H_{0}\left(\mathcal{F}_{\lambda-K_{2}h^{\alpha}}\right)} && {H_{0}\left(\mathcal{F}_{\lambda^{'}-K_{2}h^{\alpha}}\right)}
	\arrow[from=1-1, to=1-3, "\hat{v}_{\lambda}^{\lambda^{'}}"]
	\arrow[from=1-1, to=3-1,"\overline{\phi}_{\lambda}"]
	\arrow[from=3-1, to=3-3, "v_{\lambda-K_2h^{\alpha}}^{\lambda^{'}-K_2h^{\alpha}}"]
	\arrow[from=1-3, to=3-3,"\overline{\phi}_{\lambda^{'}}"]
\end{tikzcd}
\end{equation*}
\begin{equation*}
    \begin{tikzcd}
	{H_{0}\left(\mathcal{F}_{\lambda}\right)} && {H_{0}\left(\mathcal{F}_{\lambda-(K_{1}+K_{2})h^{\alpha}}\right)} \\
	& {H_{0}\left(\widehat{\mathcal{F}}_{\lambda-K_{1}h^{\alpha}}\right)}
	\arrow[from=1-1, to=1-3, "v_{\lambda}^{\lambda+(K_1+K_2)h^{\alpha}}"]
	\arrow[from=1-1, to=2-2,"\overline{\psi}_{\lambda}"]
	\arrow[from=2-2, to=1-3,"\overline{\phi}_{\lambda-K_{1}h^{\alpha}}"]
\end{tikzcd}
\end{equation*}
\begin{equation*}
\begin{tikzcd}
	{H_{0}\left(\widehat{\mathcal{F}}_{\lambda}\right)} && {H_{0}\left(\widehat{\mathcal{F}}_{\lambda-(K_{1}+K_{2})h^{\alpha}}\right)} \\
	& {H_{0}\left(\mathcal{F}_{\lambda-K_{2}h^{\alpha}}\right)}
	\arrow[from=1-1, to=1-3, "\hat{v}_{\lambda}^{\lambda+(K_1+K_2)h^{\alpha}}"]
	\arrow[from=1-1, to=2-2,"\overline{\phi}_{\lambda}"]
	\arrow[from=2-2, to=1-3,"\overline{\psi}_{\lambda-K_{2}h^{\alpha}}"]
\end{tikzcd}
\end{equation*}
Hence $\widehat{\mathbb{V}}_{f,0}$ and $\mathbb{V}_{f,0}$ are $(K_{1}+K_{2})h^{\alpha}$-interleaved, and thus we get from the algebraic stability theorem \citep{Chazal2009} that,
\begin{equation*}
d_{b}\left(\operatorname{dgm}\left(\widehat{\mathbb{V}}_{f,0}\right),\operatorname{dgm}\left(\mathbb{V}_{f,0}\right)\right)\leq (K_{1}+K_{2})h^{\alpha} 
\end{equation*}
We then conclude, using the concentration of $N_{h}$ given in Lemma \ref{noiseconcentration}.
\begin{align}
&\mathbb{P}\left(d_{b}\left(\widehat{\operatorname{dgm}(f)},\operatorname{dgm}(f)\right)\geq th^{\alpha}\right)\nonumber\\
&\leq \mathbb{P}\left(K_{1}+K_{2}\geq t\right)\nonumber\\
&=\mathbb{P}\left(N_{h}\geq \frac{t-L\left(k_{1}^{\alpha}+k_{2}^{\alpha}\right)}{4\sqrt{3\kappa}}\right)\nonumber\\
&\leq c_{0}\exp\left(-c_{1}\left(\frac{t-L\left(k_{1}^{\alpha}+k_{2}^{\alpha}\right)}{4\sqrt{3\kappa}}\right)^{2}\right)\nonumber\\
&\leq c_{0}\exp\left(-\frac{c_{1}}{48 \kappa}t^{2}\right)\exp\left(c_{1}\frac{L\left(k_{1}^{\alpha}+k_{2}^{\alpha}\right)}{24\kappa}t\right)\nonumber\exp\left(-c_{1}\left(\frac{L\left(k_{1}^{\alpha}+k_{2}^{\alpha}\right)}{4\sqrt{3\kappa}}\right)^{2}\right)\\
&= c_{0}\exp\left(-\frac{c_{1}}{96 \kappa}t^{2}\right)\exp\left(c_{1}\frac{L\left(k_{1}^{\alpha}+k_{2}^{\alpha}\right)}{24\kappa}t-\frac{c_{1}}{96 \kappa}t^{2}\right)\exp\left(-c_{1}\left(\frac{L\left(k_{1}^{\alpha}+k_{2}^{\alpha}\right)}{4\sqrt{3\kappa}}\right)^{2}\right)\nonumber
\end{align}
Thus, as for all $t\in\mathbb{R}$,
$$c_{1}\frac{L\left(k_{1}^{\alpha}+k_{2}^{\alpha}\right)}{24\kappa}t-\frac{c_{1}}{96 \kappa}t^{2}\leq \frac{c_1 L^{2}\left(k_{1}^{\alpha}+k_{2}^{\alpha}\right)^{2}}{24\kappa}$$
we have :
\begin{align}
\mathbb{P}\left(d_{b}\left(\widehat{\operatorname{dgm}(f)},\operatorname{dgm}(f)\right)\geq th^{\alpha}\right)\nonumber&\leq \mathbb{P}\left(K_{1}+K_{2}\geq t\right)\nonumber\\
&\leq c_{0}\exp\left(-c_{1}\left(\frac{L\left(k_{1}^{\alpha}+k_{2}^{\alpha}\right)}{4\sqrt{3\kappa}}\right)^{2}+\frac{c_1 L^{2}\left(k_{1}^{\alpha}+k_{2}^{\alpha}\right)^{2}}{24\kappa}\right)\exp\left(-\frac{c_{1}}{96 \kappa}t^{2}\right)\nonumber\\
&= c_{0}\exp\left(c_{1}\frac{c_1 L^{2}\left(k_{1}^{\alpha}+k_{2}^{\alpha}\right)^{2}}{48\kappa}\right)\exp\left(-\frac{c_{1}}{96 \kappa}t^{2}\right)\label{eq: concentration}
\end{align}
and the result follows.
\end{proof}
\subsection{Proof of Theorem \ref{coroModes1}}
\noindent We now proceed to prove Theorem \ref{coroModes1}. The high-level outline of the proof is is the following : 
\begin{itemize}
    \item First, we prove that, with high probability, for all $i\in\{1,...,k\}$, $\overline{C}_{i}$, the connected component of $\mathcal{F}_{m_{i}-(K_{1}+2\sqrt{3\kappa}N_{h})h^{\alpha}}^{k_{1}h}$ containing $C_{i}=\{x_{i}\}$, contains a $\hat{C}_{i}$, i.e $\overline{C}_{i}$ contains a connected component appearing in the filtration $\hat{\mathcal{F}}$ with associated lifetime exceeding $\delta/2$.
    \item We show, with high probability, that the birth time associated to $\hat{C}_{i}$ is close to $m_{i}$.
    \item We use Proposition \ref{MainProp}, to show that, with high probability, $k=\hat{k}$.
    \item Finally, we use assumption \textbf{A4} to bound, with high probability, the diameter of $\overline{C}_{i}$, and thus the distance between $x_{i}$ and $\hat{x}_{i}$.
\end{itemize}
The last two steps are pretty straightforward, but the two first ones are more technical. The proofs of these first claims are then decomposed into the three following lemmas. Let us define, for all $i\in\{1,...,k\}$, for all $\lambda<m_{i}$, $\overline{C}^{i}_{\lambda}$ the connected component of $\mathcal{F}_{\lambda}^{k_{1}h}$ containing $C_{i}=\{x_{i}\}$ and denote,
$$\tilde{b}_{i}=\sup\left\{\lambda\in\mathbb{R}\text{ s.t } \widehat{\mathcal{F}}_{\lambda}\cap\overline{C}^{i}_{m_{i}+(K_{1}+K_{2})h^{\alpha}-\delta}\ne \emptyset\right\}$$
with $k_{1}$, $K_{1}$ and $K_{2}$ as defined in the proof of Proposition \ref{MainProp}. We start by establishing upper and lower bounds on $\tilde{b}_{i}$ through the following lemma.
\begin{lmm}
\label{claim 1}
Consider the event $E_{1}$ :
\begin{equation*}
\left(K_{1}+K_{2}\right)\left(\frac{\log(n)}{n}\right)^{\frac{\alpha}{d+2\alpha}}<\delta/8.
\end{equation*}
Under $E_{1}$, we have, for all $i\in\{1,...,k\}$,
\begin{equation*}
    m_{i}-K_{1}h^{\alpha}\leq \tilde{b}_{i}\leq m_{i}+2\sqrt{3\kappa}N_{h}h^{\alpha}.
\end{equation*}
\end{lmm}
\begin{proof}
The lower inclusion of Lemma \ref{interleavish 3} ensures that $x_{i}\in \widehat{\mathcal{F}}_{m_{i}-K_{1}h^{\alpha}}$ and thus,
$$\tilde{b}_{i}\geq m_{i}-K_{1}h^{\alpha}.$$
The upper inclusion of Lemma \ref{interleavish 3} then implies that, under $E_{1}$, any connected component $\tilde{C}^{i}$ of $\widehat{\mathcal{F}}_{\tilde{b}_{i}}$ intersecting $\overline{C}^{i}_{m_{i}+(K_{1}+K_{2})h^{\alpha}-\delta}$ is included in the connected component $\overline{C}^{i}_{m_{i}+(K_{1}+K_{2})h^{\alpha}-\delta}$.
By Lemma \ref{lmm1b}, any connected component of $\widehat{\mathcal{F}}_{\tilde{b}_{i}}$ intersects $\mathcal{F}_{\tilde{b}_{i}-2\sqrt{3\kappa}N_{h}h^{\alpha}}$. By assumption \textbf{A4}, $\overline{C}^{i}_{m_{i}+(K_{1}+K_{2})h^{\alpha}-\delta}\subset B_{2}(x_{i},h_{0})$ and, if,
$$\tilde{b}_{i}>m_{i}+2\sqrt{3\kappa}N_{h}h^{\alpha}$$
then,
$$\mathcal{F}_{\tilde{b}_{i}-2\sqrt{3\kappa}N_{h}h^{\alpha}}\cap B_{2}(x_{i},h_{0}) =\emptyset$$
and thus,
$$\mathcal{F}_{\tilde{b}_{i}-2\sqrt{3\kappa}N_{h}h^{\alpha}}\cap \overline{C}^{i}_{m_{i}+(K_{1}+K_{2})h^{\alpha}-\delta} =\emptyset.$$
Hence, we have,
$$\tilde{b}_{i}\leq m_{i}+2\sqrt{3\kappa}N_{h}h^{\alpha}.$$
\end{proof}
Now, we prove that for all $i\in\{i,...,k\}$, $\overline{C}_{i}$
contains a connected component of $\widehat{\mathcal{F}}_{\tilde{b}_{i}}$ with associated lifetime in $\widehat{\operatorname{dgm}(f)}$ exceeding $\delta/2$. The proof is divided into the two following lemmas.
\begin{lmm}
\label{claim 2}
Under the event $E_{1}$, for all $i\in\{1,...,k\}$, any connected component $\tilde{C}^{i}$ of $\widehat{\mathcal{F}}_{\tilde{b}_{i}}$ intersecting $\smash{\overline{C}^{i}_{m_{i}+(K_{1}+K_{2})h^{\alpha}-\delta}}$ is included in $\smash{\overline{C}_{i}}$.
\end{lmm}
\begin{proof}
First note that, under $E_{1}$, we have,
$$\overline{C}_{i}\subset \overline{C}^{i}_{m_{i}+(K_{1}+K_{2})h^{\alpha}-\delta}$$
and by Assumption \textbf{A4} and Lemma \ref{claim 1},
$$\mathcal{F}_{\tilde{b}_{i}-2\sqrt{3\kappa}N_{h}h^{\alpha}}\cap \overline{C}^{i}_{m_{i}+(K_{1}+K_{2})h^{\alpha}-\delta}\subset \overline{C}_{i}.$$ 
By Lemma \ref{lmm1b}, any connected component $\tilde{C}^{i}$ of $\widehat{\mathcal{F}}_{\tilde{b}_{i}}$ intersects $\mathcal{F}_{\tilde{b}_{i}-2\sqrt{3\kappa}N_{h}h^{\alpha}}$. Thus, if furthermore $\tilde{C}^{i}$ intersects $\overline{C}^{i}_{m_{i}+(K_{1}+K_{2})h^{\alpha}-\delta}$, it intersects $\overline{C}_{i}$. As, by Lemma \ref{claim 1}, $\tilde{b}_{i}>m_{i}-K_{1}h^{\alpha}$, Lemma \ref{interleavish 3} then ensures that any such $\tilde{C}^{i}$ is then included in $\overline{C}_{i}$.    
\end{proof}
By definition of $\tilde{b}_{i}$, a consequence of Lemma \ref{claim 2} is that, under $E_{1}$, for all $i\in\{1,...,k\}$, any connected component $\tilde{C}^{i}$ of $\widehat{\mathcal{F}}_{\tilde{b}_{i}}$ intersecting $\overline{C}^{i}_{m_{i}+(K_{1}+K_{2})h^{\alpha}-\delta}$ has associated birth time $\tilde{b}_{i}$. It now suffices to prove that, for all $i\in\{1,...,k\}$, there exists such a $\tilde{C}^{i}$ with associated death time $\tilde{d}_{i}$ satisfying $\tilde{b}_{i}-\tilde{d}_{i}>\delta/2$.
\begin{lmm}
\label{claim 3}
Under $E_{1}$, for all $i\in\{1,...,k\}$, there exists a connected 
component $\tilde{C}^{i}$ of $\widehat{\mathcal{F}}_{\tilde{b}_{i}}$ intersecting $\overline{C}^{i}_{m_{i}+(K_{1}+K_{2})h^{\alpha}-\delta}$ with an associated death time $\tilde{d}_{i}$ in $\widehat{\operatorname{dgm}(f)}$, satisfying, 
$$\tilde{b}_{i}-\tilde{d}_{i}>\delta/2.$$
\end{lmm}
\begin{proof}
Let $\tilde{C}^{i}$ a connected component of $\hat{\mathcal{F}}_{\tilde{b_{i}}}$ intersecting $\overline{C}^{i}_{m_{i}+(K_{1}+K_{2})h^{\alpha}-\delta}$. For all $\lambda\leq\tilde{b}_{i}$, we denote $\Tilde{C}_{\lambda}^{i}$ the connected component of $\widehat{\mathcal{F}}_{\lambda}$ containing $\tilde{C}^{i}$. To prove Lemma \ref{claim 3}, it suffices to show that, for all $\tilde{b}_{i}-\delta/2\leq \lambda\leq\tilde{b}_{i}$,
$$\text{if }x\in \tilde{C}_{\lambda}^{i}\text{ then } x\notin \widehat{\mathcal{F}}_{\lambda^{'}}\text{, for all } \lambda^{'}>\tilde{b}_{i}.$$
Applying Lemma \ref{interleavish 3} and Lemma \ref{claim 1}, we have that, for all $\tilde{b}_{i}-3\delta/4<\lambda\leq\tilde{b}_{i}$, under $E_{1}$, any connected components of $\widehat{\mathcal{F}}_{\lambda}$ intersecting $\overline{C}^{i}_{m_{i}+(K_{1}+K_{2})h^{\alpha}-\delta}$ is included in $\overline{C}^{i}_{m_{i}+(K_{1}+K_{2})h^{\alpha}-\delta}$. Thus, under $E_{1}$, for all $\tilde{b}_{i}-3\delta/4<\lambda\leq\tilde{b}_{i}$, 
$$\Tilde{C}_{\lambda}^{i}\subset \overline{C}^{i}_{m_{i}+(K_{1}+K_{2})h^{\alpha}-\delta}$$
Then, by definition of $\tilde{b}_{i}$, under $E_{1}$, for all $\tilde{b}_{i}-3\delta/4<\lambda\leq\tilde{b}_{i}$,
$$\text{if } x\in \tilde{C}_{\lambda}^{i} \text{ then } x\notin \widehat{\mathcal{F}}_{\lambda^{'}}\text{, for all }\lambda^{'}>\tilde{b}_{i}$$
and Lemma \ref{claim 3} follows.
\end{proof}
Equipped with Lemma \ref{claim 1}, \ref{claim 2} and \ref{claim 3}, we can now prove Theorem \ref{coroModes1}.
\begin{proof}[Proof of Theorem \ref{coroModes1}]
Let $f\in S_{d}(L,\alpha,\mu,R_{\mu},C,h_{0})$. From the proof of Proposition \ref{MainProp}, we have that,
\begin{equation}
\label{eq: dgm 1}
d_{\infty}\left(\widehat{\operatorname{dgm}(f)},\operatorname{dgm}(f)\right)\leq\left(K_{1}+K_{2}\right)\left(\frac{\log(n)}{n}\right)^{\frac{\alpha}{d+2\alpha}}.
\end{equation}
Then, under $E_{1}$, it follows that $\hat{k}=k$. Recalling that $\delta=Ch_{0}^{\alpha}$, it follows from (\ref{eq: concentration}), that there exist $A_{1}$ and $B_{1}$ such that $E_{1}$ occurs with probability at least $1-A_{1}\exp(-B_{1}(h_0/h)^{2\alpha})$.\\\\
The combination of Lemma \ref{claim 2} and \ref{claim 3} ensures that, under $E_{1}$, for all $i\in\{1,...,k\}$ there exists a connected component $\tilde{C}^{i}$ appearing in $\widehat{\mathcal{F}}$ at time $\tilde{b}_{i}$, included in $\overline{C}^{i}$, such that its associated lifetime exceeds $\delta/2$. Thus, by definition of $\hat{C}$, for all $i\in\{1,...,k\}$,  $\Tilde{C}^{i}\in \hat{C}$. Lemma \ref{satisfyC2b} and Assumption \textbf{A4} imply that, under $E_{1}$, $\overline{C}_{1}$, ..., $\overline{C}_{k}$ are disjoint, thus $\Tilde{C}^{1},...,\Tilde{C}^{k}$ are disjoint. Hence, as $k=\hat{k}$, we have,
$$\hat{C}=\left\{\hat{C}_{1},...,\hat{C}_{k}\right\}=\left\{\Tilde{C}^{1},...,\Tilde{C}^{k}\right\}$$
and
$$\hat{m}=\left\{\hat{m}_{1},...,\hat{m}_{k}\right\}=\left\{\tilde{b}_{1},...,\tilde{b}_{k}\right\}$$
As, for all $i\in\{1,...,k\}$, $x_{i}\in C_{i}\subset\overline{C}_{i}$ and $\hat{x}_{i}\in\hat{C}_{i}=\Tilde{C}^{i}\subset \overline{C}_{i}$, to bound the distance between $\hat{x}_{i}$ and $x_{i}$, it suffices to bound the diameter of $\overline{C}_{i}$, defined by,
$$\operatorname{diam}\left(\overline{C}_{i}\right)=\max\limits_{x,y\in \overline{C}_{i}}||x-y||_{\infty}.$$
Consider the event $E_{2}$ :
\begin{equation*}
\frac{(K_{1}+2\sqrt{3\kappa}N_{h})^{1/\alpha}}{C}h<h_{0}/2. 
\end{equation*}
As $2\sqrt{3\kappa}N_{h}\leq K_2$, if $K_1+K_2< \left(C\frac{h_0}{2h}\right)^{\alpha}$ occurs then $E_2$ occurs. Thus
 by (\ref{eq: concentration}), there exist $A_{2}$ and $B_{2}$ such that $E_{2}$ occurs with probability at least $1-A_{2}\exp(-B_{2}(h_0/h)^{2\alpha})$.\\\\
Under $E_{2}$, as we assumed $h_0\geq 2(\sqrt{d}+\lceil\sqrt{d} / \mu\rceil)h$, we have : 
$$\left(\frac{\left(K_1+2 \sqrt{3 M_{d, L, \alpha, R_\mu}} N_h\right)^{1 / \alpha}}{C}+(\sqrt{d}+\lceil\sqrt{d} / \mu\rceil)\right) h<h_0$$
Thus, as $f$ satisfies Assumption \textbf{A4}, under $E_{2}$, for all $i\in \{1,...,k\}$,
$$ \operatorname{diam}(\overline{C}_{i})\leq \left(\frac{(K_{1}+2\sqrt{3\kappa}N_{h})^{1/\alpha}}{C}+(\sqrt{d}+\lceil\sqrt{d}/\mu\rceil)\right)h.$$
And thus,
\begin{equation}
\label{bound xi}
\|x_{i}-\hat{x}_{i}\|_{\infty}\leq \left(\frac{(K_{1}+2\sqrt{3\kappa}N_{h})^{1/\alpha}}{C}+(\sqrt{d}+\lceil\sqrt{d}/\mu\rceil)\right)h.
\end{equation}
Consider the event $E_{3}$ :
$$\left(\frac{(K_{1}+2\sqrt{3\kappa}N_{h})^{1/\alpha}}{C}+(\sqrt{d}+\lceil\sqrt{d}/\mu\rceil)\right)\leq A$$
and the event $E_{4}$ : $K_{1}\leq A$.\\\\
Again as $2\sqrt{3\kappa}N_{h}\leq K_2$, by (\ref{eq: concentration}), there exist $A_{3}$, $B_{3}$, such that $E_{3}$ occurs with probability $1-A_{3}\exp(-B_{3}A^{2})$. And as $K_1\leq K_1+K_2$, by (\ref{eq: concentration}) there also exist $A_{4}$, $B_{4}$, such that $E_{4}$ occurs with probability $1-A_{4}\exp(-B_{4}A^{2})$.  Hence, it follows from Lemma \ref{claim 1} and (\ref{bound xi}) that, with probability at least,
\begin{align*}
    &1-\mathbb{P}\left(E_{1}^{c}\right)-\mathbb{P}\left(E_{2}^{c}\right)-\mathbb{P}\left(E_{3}^{c}\right)-\mathbb{P}\left(E_{4}^{c}\right)\\
    \geq&1-A_{1}\exp(-B_{1}(h_{0}/h)^{2\alpha})-A_{2}\exp(-B_{2}(h_{0}/h)^{2\alpha})\\
    &\quad-A_{3}\exp(-B_{3}A^{2})-A_{4}\exp(-B_{4}A^{2})
\end{align*}
$\hat{k}=k$ and for all $i\in\{1,...,k\}$ there exists (distinct) $(\hat{x}_{i},\hat{m}_{i})\in \hat{x}\times \hat{m}$ such that,
$$\|x_{i}-\hat{x}_{i}\|_{\infty}\leq A h$$
and
$$\left|\hat{m}_{i}-m_{i}\right|\leq A h^{\alpha}.$$
\end{proof}
\subsection{Proof of Theorem \ref{lowerboundsModes} and Proposition \ref{prp: tight separation}}
\label{Section : inf modes }
This section is devoted to the proofs of Theorem~\ref{lowerboundsModes} and Proposition~\ref{prp: tight separation}, both of which follow standard arguments for constructing minimax lower bounds.
\begin{proof}[Proof of Theorem \ref{lowerboundsModes}]
The proof essentially follows the same idea as the proof of Theorem 2 of \cite{ariascastro2021estimation}. For simplicity we fix $L=3$ and $C=1$ and suppose that $h_0<1/(2\sqrt{d})^{\alpha}$ and such that $1/(4\sqrt{d}h_{0})$ is an even integer, other cases can be treated similarly. Let 
$$\beta(x)=\left((2\sqrt{d})^{\alpha}h_{0}^{\alpha}-||x||_{\infty}^{\alpha}\right)_{+}.$$
 Let $x_{i_1,...,i_{d}}= 4\sqrt{d}h_{0}(i_1-1/2,...,i_{d}-1/2)$, $i_{1},...,i_{d}\in\left[1,1/(4\sqrt{d}h_{0})\right]\cap \mathbb{N}$, the centers of the hypercubes of the regular grid $G_{4\sqrt{d}h_{0}}$. We distinguish the sets :
$$P_{1}=\left\{i_{1},...,i_{d}\in\left[1,1/(4\sqrt{d}h_{0})\right]\cap \mathbb{N}\text{ s.t. }\sum_{j=1}^{d}i_j\text{ is even}\right\}$$ and 
$$P_{2}=\left\{i_{1},...,i_{d}\in\left[1,1/(4\sqrt{d}h_{0})\right]\cap \mathbb{N}\text{ s.t. }\sum_{j=1}^{d}i_j\text{ is odd}\right\}.$$
We define :
$$f_{0}(x)=1+\sum_{I\in P_{1}}\beta(x-x_I)-\sum_{I\in P_{2}}\beta(x-x_I).$$
Note that, as $|P_1|=|P_2|$ and for all $I,J\in P_{1}\cup P_{2}$, $\int_{[0,1]^{d}}\beta(x-x_I)dx=\int_{[0,1]^{d}}\beta(x-x_J)dx$, we have :
\begin{align*}
\int_{[0,1]^{d}}f_0(x)dx=1+\sum_{I\in P_{1}}\int_{[0,1]^{d}}\beta(x-x_I)dx-\sum_{I\in P_{2}}\int_{[0,1]^{d}}\beta(x-x_I)dx=1
\end{align*}
Now, as $h_{0}<1/(2\sqrt{d})^{\alpha}$, we have $f_{0}>0$ and thus $f_0$ is a probability density. As $\beta$ is $(1,\alpha)-$Hölder continuous and the maps $x\mapsto \beta(x-x_{I})$, $I\in P_1\cup P_2$ have disjoint support, $f_{0}$ is $(1,\alpha)-$Hölder continuous. It also follows that $f_{0}$ admits $1/(2(4\sqrt{d}h_{0})^{d})$ modes : the $(x_I)_{I\in P_1}$ all associated with the global maximum : $1+(2\sqrt{d})^{\alpha}h_{0}^{\alpha}$. Furthermore, as, for any $I\in P_1$, and any $x\in\{||x-x_I||_2\leq h_{0}\}\subset \{||x-x_I||_{\infty}\leq h_{0}\}$, $f_{0}(x)=\beta(x-x_{I})$ and  $x\mapsto \beta(x-x_I)$ satisfies \textbf{A4} for $C=1$, it follows that $f_{0}$ satisfies \textbf{A4} for $C=1$. Thus, $f_{0}$ belongs to $S_{d}(L,\alpha,\mu,R_{\mu},C,h_{0})$ for $C=1$ and all $L>C$, $0<\mu\leq 1$, $R_{\mu}>0$.
Now, let $0<h<\min(((1-(2\sqrt{d})^{\alpha}h_{0}^{\alpha})/2)^{1/\alpha},h_0)$ and define :
$$g(x)=2(h^{\alpha}-||x||_{\infty}^{\alpha})_{+}$$
with $\{y\geq 0\}$ the set of vectors of $\mathbb{R}^{d}$ with only positive coordinates. we fix $J_1\in P_2$ and, for all $I\in P_{1}$, we then define :
$$f_{I}(x)=f_{0}(x)+g(x-x_I-(h,...,h))-g(x-x_{J_1}-(h,...,h))$$
Now, as $h<((1-(2\sqrt{d})^{\alpha}h_{0}^{\alpha})/2)^{1/\alpha}$, $f_{I}\geq 1-(2\sqrt{d})^{\alpha}h_{0}^{\alpha}-2h^{\alpha}\geq 0$ and as $\int_{[0,1]^{d}}g(x-x_I)dx=\int_{[0,1]^{d}}g(x-x_{J_1})dx$, we have $\int_{[0,1]^{d}}f_{I}(x)dx=\int_{[0,1]^{d}}f_{0}(x)dx=1$. Thus, $f_{I}$ is a probability density. As $f_0$ and $g$ are $(1,\alpha)$ and $(2,\alpha)-$Hölder continuous, $f_{I}$ is $(3,\alpha)-$Hölder continuous. Hence, it satisfies \textbf{A1}, \textbf{A2} and \textbf{A3} for  $L=3$, and all $\mu\in]0,1]$ and $R_{\mu}>0$. Observe that, as $h\leq h_0$ and $||x_I-\tilde{x}_I||_{\infty}= h$, then :
$$\{||x-\tilde{x_i}||_{\infty}\leq h\}\subset\{||x-\tilde{x_i}||_{\infty}\leq h_0\}\subset \{||x-x_i||_{\infty}\leq 2h_0\}$$
and thus $f_I$ admits $1/(2(4\sqrt{d}h_{0})^{d})$ modes : the $(x_{J})_{J\in P_1\backslash\{I\}}$ and $\tilde{x}_{I}=x_{I}+(h,...,h)$ respectively associated to the local maxima :
$$1+(2\sqrt{d})^{\alpha}h_{0}^{\alpha}\text{ and }1+(2\sqrt{d})^{\alpha}h_{0}^{\alpha}+h^{\alpha}.$$
For all $J\in P_1\backslash\{I\}$, on $\{||x-x_J||_{2}\leq h_{0}\}\subset \{||x-x_J||_{\infty}\leq 2\sqrt{d}h_{0}\}$, $f_I(x)=f_0(x)$ and thus $f_I$ satisfies \textbf{A4} around all $(x_{J})_{J\in P_1\backslash\{I\}}$ for $C=1$. Now if $x\in \{||x-\tilde{x}_I||_{2}\leq h_{0}\}\subset \{||x-\tilde{x}_I||_{\infty}\leq \sqrt{d}h_{0}\}\subset \{||x-x_i||_{\infty}\leq 2\sqrt{d}h_0\}$, as $\alpha\in]0,1]$ we have :
\begin{align*}
f_I(\tilde{x}_{I})-f_I(x)&=h^{\alpha}+||x-x_I||_{\infty}^{\alpha}-2(h^{\alpha}-||x-\tilde{x}_I||_{\infty}^{\alpha})_{+}\\
&\geq-h^{\alpha}+2||x-\tilde{x}_I||_{\infty}^{\alpha}+||x-x_I||_{\infty}^{\alpha}\\
&\geq -h^{\alpha}+||x-\tilde{x}_I||_{\infty}^{\alpha}+(||x-\tilde{x}_I||_{\infty}+||x-x_I||_{\infty})^{\alpha}\\
&\geq -h^{\alpha}+||x-\tilde{x}_I||_{\infty}^{\alpha}+||x_I-\tilde{x}_I||_{\infty}^{\alpha}\\
&=||x-\tilde{x}_I||_{\infty}^{\alpha}
\end{align*}
which proves that $f_I$ satisfies \textbf{A4} around $\tilde{x}_I$ for $C=1$.
Then, $\{f_{0}\}\cup\{f_{I}, I\in P_1\}$ is a collection of probability densities of $S_{d}(L,\alpha,\mu,R_{\mu},C,h_{0})$, for $C=1$, $L=3$ and all $\mu\in]0,1]$ and $R_{\mu>0}$, that have, pairwise, modes separated by $h$ and associated local maxima separated by $h^{\alpha}$. More precisely for any distinct $f,\tilde{f}\in\{f_0\}\cup\{f_I, I\in P_1\}$, let $\mathcal{M}_f$ (resp. $\mathcal{M}_{\tilde{f}}$) the set of modes of $f$ (resp. $\tilde{f}$), $\Gamma(f,\tilde{f})$ the set of one to one maps from $\mathcal{M}_f$ to $\mathcal{M}_{\tilde{f}}$ and
$$\gamma^{*}\in\operatorname{argmin}_{\gamma\in \Gamma(f,\tilde{f})}\max_{x\in \mathcal{M}_{f}}||x-\gamma(x)||_{\infty}$$
we have :
$$\max_{x\in \mathcal{M}_f}||x-\gamma^{*}(x)||_{\infty}\geq h\text{ and}\max_{x\in \mathcal{M}_f}|f(x)-\tilde{f}(\gamma^{*}(x))|\geq h^{\alpha}.$$
Now, we prove that if $h\lesssim\left(\log(1/h_0)/n\right)^{\frac{1}{2\alpha+d}}$ then $\{f_{I}, I\in P_1\}$ are not all distinguishable from $f_{0}$ with high probability. To do so we use Theorem 2.6 from \cite{TsybakovBook}, which reduces the problem to studying the $\chi^{2}$ divergence between $\mathbb{P}^{n}_{f_{0}}$ and the $\mathbb{P}^{n}_{f_I}$, $I\in P_1$. More precisely, it suffices to prove that :
\begin{equation}
\label{eq:lowerbound modes}
  \frac{1}{|P_1|^2}\sum_{I\in P_1}\chi^{2}(\mathbb{P}^{n}_{f_{0}},\mathbb{P}^{n}_{f_{I}})=\frac{1}{|P_1|^2}\sum_{I\in P_1}\left(\int_{[0,1]^{d}}\frac{f_I^2}{f_0^2}\right)^n-1\underset{n\rightarrow+\infty}{\longrightarrow}0
\end{equation}
whenever $h=o\left(\left(\log(1/h_0)/n\right)^{\frac{1}{2\alpha+d}}\right)$. Let $I\in P_1$, we have :
\begin{align*}
\int_{[0,1]^{d}}\frac{f^{2}_{I}}{f_{0}}&= \int_{[0,1]^{d}}\frac{\left(f_{0}+g(x-\tilde{x}_I)-g(x-\tilde{x}_{J_1})\right)^{2}}{f_{0}}\\
&\leq 1+\frac{1}{1-(2\sqrt{d})^{\alpha}h_{0}^{\alpha}}\int_{[0,1]^{d}}g(x-\tilde{x}_I)^{2}dx+\frac{1}{1-(2\sqrt{d})^{\alpha}h_{0}^{\alpha}}\int_{[0,1]^{d}}g(x-\tilde{x}_{J_1})^{2}dx\\
&\leq 1+\frac{2}{1-(2\sqrt{d})^{\alpha}h_{0}^{\alpha}}\int_{||x||_{\infty}\leq h}g^{2}(x)dx\\
&\leq 1+\frac{2}{1-(2\sqrt{d})^{\alpha}h_{0}^{\alpha}}\int_{||x||_{\infty}\leq h}h^{2\alpha}dx\\
&=1+ O\left(h^{2\alpha+d}\right)
\end{align*}
It follows that if $h=o\left((\log(|P_1|)/n)^{\frac{1}{2\alpha+d}}\right)$, then we have (\ref{eq:lowerbound modes}). Finally, observing that $|P_1|\asymp 1/h_0$, Theorem \ref{lowerboundsModes} is proved.
\end{proof}
\begin{proof}[Proof of Proposition \ref{prp: tight separation}]
The proof is a standard Le Cam's two point argument, using again $\chi^2$ divergence. Once again for simplicity we will consider specific parameters $L$ and $C$, but other cases can be treated similarly. Let $f_{1}=1+||x||_{\infty}^{\alpha}$ and $\tilde{f}_{1}=f_{1}/Z_1$ with $Z_{1}=\int_{[0,1]^d}f_1$. The function $\tilde{f}_{1}$ is probability density satisfying \textbf{A1}-\textbf{A4}, for $C=1/Z_1$ and all $L>C$, $\mu\in]0,1]$, $R_{\mu}>0$ and $h_{0}\leq 1$. It admits a unique mode : $(1,...,1)$. Now, let $h<1/4$ and consider :
$$\tilde{f}_{0}(x)=\tilde{f}_{1}(x)+\frac{1}{Z_1}\left[\left(h^{\alpha}-\left\|x-(1/2,...,1/2)\right\|^{\alpha}_{\infty}\right)_{+}-2^{d}\left(h^{\alpha}-\|x\|^{\alpha}_{\infty}\right)_{+}\right].$$
Observing that :
$$\int_{[0,1]^{d}} \left(h^{\alpha}-\left\|x-(1/2,...,1/2)\right\|^{\alpha}_{\infty}\right)_{+}dx=2^d\int_{[0,1]^{d}}\left(h^{\alpha}-\|x\|^{\alpha}_{\infty}\right)_{+}$$
we have that $\int_{[0,1]^{d}}\tilde{f}_0=1$. For sufficiently small $h$, we have $\tilde{f}_0\geq 0$. Thus, $\tilde{f}_0$ is a probability density. It admits two modes $(1,...,1)$ and $(1/2,...,1/2)$ and by construction satisfies \textbf{A1}-\textbf{A4}, for $C=1/Z_1$, $h_{0}\asymp h$ and all $L>2^{d}C$, $\mu\in]0,1]$ and $R_{\mu}>0$ . Observe that, if we denote $\mathcal{M}_{0}=\{(1,...,1),(1/2,...,1/2)\}$ and $\mathcal{M}_1=\{(1,...,1))\}$ the sets of modes of $f_0$ and $f_1$, we have :
$$d_{H}\left(\mathcal{M}_1,\mathcal{M}_0\right)\geq 1/2.$$
Now, it suffices to show that :
$$n\chi^{2}(\tilde{f}_0,\tilde{f}_1)\underset{n\rightarrow+\infty}{\longrightarrow}0$$
whenever $h=o(n^{-\frac{1}{2\alpha+d}})$. 
\begin{align*}
\chi^{2}\left(\Tilde{f}_{0},\Tilde{f}_{1}\right)&=\int_{[0,1]^{d}}\frac{\Tilde{f}^{2}_{0}}{\Tilde{f}_{1}}-1\\
&= \int_{[0,1]^{d}}\frac{\left(\Tilde{f}_{1}(x)+\frac{1}{Z_1}\left(\left(h^{\alpha}-\left\|x-(1/2,...,1/2)\right\|^{\alpha}_{\infty}\right)_{+}-2^{d}\left(h^{\alpha}-\|x\|^{\alpha}_{\infty}\right)_{+}\right)\right)^{2}}{\Tilde{f}_{1}(x)}dx-1\\
&\leq \frac{1}{Z_1}\int_{\left\|x-(1/2,...,1/2)\right\|_{\infty}\leq h}\left(h^{\alpha}-\left\|x-(1/2,...,1/2)\right\|_{\infty}^{\alpha}\right)^{2}dx\\
&\quad+ \frac{2^d}{Z_1}\int_{\{\|x\|_{\infty}\leq h\}\cap[0,1]^{d}}\left(h^{\alpha}-\|x\|^{\alpha}_{\infty}\right)^{2}dx\\
&= \frac{2^{d+1}}{Z_1}\int_{\{\|x\|_{\infty}\leq h\}\cap[0,1]^{d}}\left(h^{\alpha}-\|x\|^{\alpha}_{\infty}\right)^{2}dx\\
&\leq  \frac{2^{d+1}}{Z_1}\int_{\{\|x\|_{\infty}\leq h\}\cap[0,1]^{d}}h^{2\alpha}dx\\
&= O\left(h^{2\alpha+d}\right).
\end{align*}
Thus, if  $h=o(n^{-\frac{1}{2\alpha+d}})$ then $n\chi^{2}(\tilde{f}_0,\tilde{f}_1)$ converges to zero and the result is proved.
\end{proof}
\section{A numerical illustration}
\label{Section: Numerical illustration}
From a more practical perspective, we now present a 2D numerical illustration of our results. As previously emphasized, an interesting aspect of our work is that we do not assume any regularity in the neighborhoods of the modes, even allowing for wild discontinuities at the modes themselves, for instance, modes may lie on the boundary of multiple distinct regular pieces. We consider a simple toy example of a density in the class \( S_{d}(L,\alpha,\mu,R_{\mu},C,h_{0}) \), which exhibits such typical discontinuities around some of its modes. On this example, we compare the performance of our persistence-based procedure with that of one of the most popular algorithms for mode inference: mean-shift. Our simulations reveal that recovering the exact number and locations of modes is challenging for the mean-shift algorithm. In contrast, and as predicted by our theoretical results, the persistence-based strategy we propose achieves accurate recovery.\\\\
More precisely, we consider the 2D density given by the normalization of $f$ define by,
\[
f(x, y) =
\begin{cases}
\frac{1}{2} \max\left(0, \tfrac{1}{8} - (x - \tfrac{1}{4})^2 - (y - \tfrac{1}{3})^2 \right) & \text{if } |x - \tfrac{1}{4}| + |y - \tfrac{1}{3}| \leq \tfrac{1}{8}, \\
 \frac{1}{2}\max\left(0,\; \tfrac{1}{8} - (x - \tfrac{1}{4})^2 - (y - \tfrac{2}{3})^2 \right) & \text{if } |x - \tfrac{1}{4}| + |y - \tfrac{2}{3}| \leq \tfrac{1}{8}, \\
\max\left(0,\; \tfrac{1}{8} - (x - \tfrac{2}{3})^2 - (y - \tfrac{1}{2})^2 \right) & \text{if } |x - \tfrac{2}{3}| + |y - \tfrac{3}{4}| \leq \tfrac{1}{4}\\
\max\left(0,\; \tfrac{1}{8} - (x - \tfrac{2}{3})^2 - (y - \tfrac{1}{2})^2 \right)&\text{if } |x - \tfrac{2}{3}| + |y - \tfrac{1}{4}| \leq \tfrac{1}{4}, \\
0 & \text{otherwise}.
\end{cases}
\]

This function admits three modes: two located within regular pieces, \( x_1 = (1/4,\,1/3) \) and \( x_2 = (1/4,\,2/3) \), and one at the boundary of three regular pieces, \( x_3 = (2/3,\,1/2) \). Around \( x_3 \), the function is not continuous in any full neighborhood, but only within two half-cones whose boundaries intersect at \( x_3 \). One can check that this function belongs to the class \( S_{d}(L,1,\mu,R_{\mu},C,h_{0}) \) for any \( \mu \leq 1/\sqrt{2} \), and for some constants \( L > 0 \), \( R_{\mu} > 0 \), \( C > 0 \), and \( h_0 > 0 \).

\begin{figure}[H]
    \centering
    \includegraphics[scale=0.7]{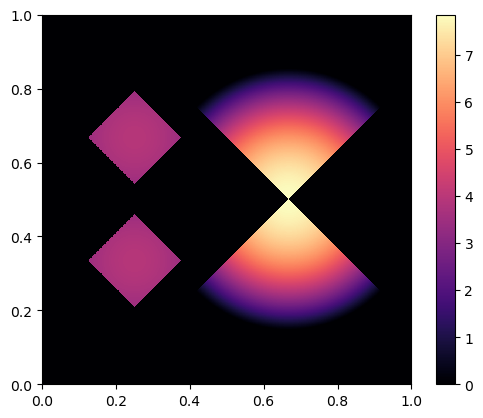}
    \caption{Graph of $f$ (normalized).}
\end{figure}

To evaluate the performance of our methods, we implement our estimator in \textsc{Python} using the package \texttt{GUDHI} \citep{gudhi:urm}, and in particular its \texttt{CubicalComplex} class \citep{gudhi:CubicalComplex} to compute persistence diagrams from histograms. For the mean-shift procedure, we use the \texttt{MeanShift} implementation from the \texttt{sklearn.cluster} module, which follows the algorithm proposed in \cite{Comanicu02}. We sample \( n = 40{,}000 \) points according to the normalized version of the density \( f \), using the rejection sampling method. Following our theoretical results, we calibrate our estimator by choosing the parameters \( \mu = 1/\sqrt{2} \) and \( h = (\log(n)/n)^{1/4}/8 \). We choose $\delta$ using the selection procedure of Appendix \ref{apdx:complement modes}. Since selecting a bandwidth for the mean-shift estimator is nontrivial in this setting, we run the algorithm with various bandwidths :
$$ b \in \{0.05,\, 0.15,\, 0.2,\, 0.23,\, 0.2425,\, 0.25,\, 0.3,\, 0.35,\, 0.4\},$$
which, based on our empirical observations, capture the range of behaviors of the mean-shift algorithm on this example.\\\\
In Figure~\ref{fig:meanshift}, we plot, for a fixed sample, the modes estimated by the mean-shift algorithm. Despite the relatively large sample size, the results are unsatisfactory. For most bandwidth values, the algorithm fails to recover the correct number of modes. For small bandwidths, it tends to split \( x_3 \) into two (or more) distinct modes, while for large bandwidths, it often merges \( x_1 \) and \( x_2 \) (and in extreme cases, even \( x_3 \)) into a single mode. At \( b = 0.2425 \), the correct number of modes is recovered, but the result remains unsatisfactory: \( x_1 \) and \( x_2 \) appear to be merged into a single mode, while \( x_3 \) is split into two. Overall, there is no clear ``sweet spot'' where \( x_1 \) and \( x_2 \) are properly separated, and \( x_3 \) is correctly identified as a single mode. When we omit the number of modes and instead focus on the Hausdorff distance between the sets of estimated modes and the true modes, we observe that lower bandwidths lead to more precise localization of the modes, with the estimated modes concentrating around the true ones. In contrast, in Figure~\ref{fig:ours}, for the same sample, our persistence-based strategy appears to perform much better, recovering the correct number of modes with estimated locations that are relatively close to the true ones.
\begin{figure}[H]
    \centering
    \includegraphics[scale=0.25]{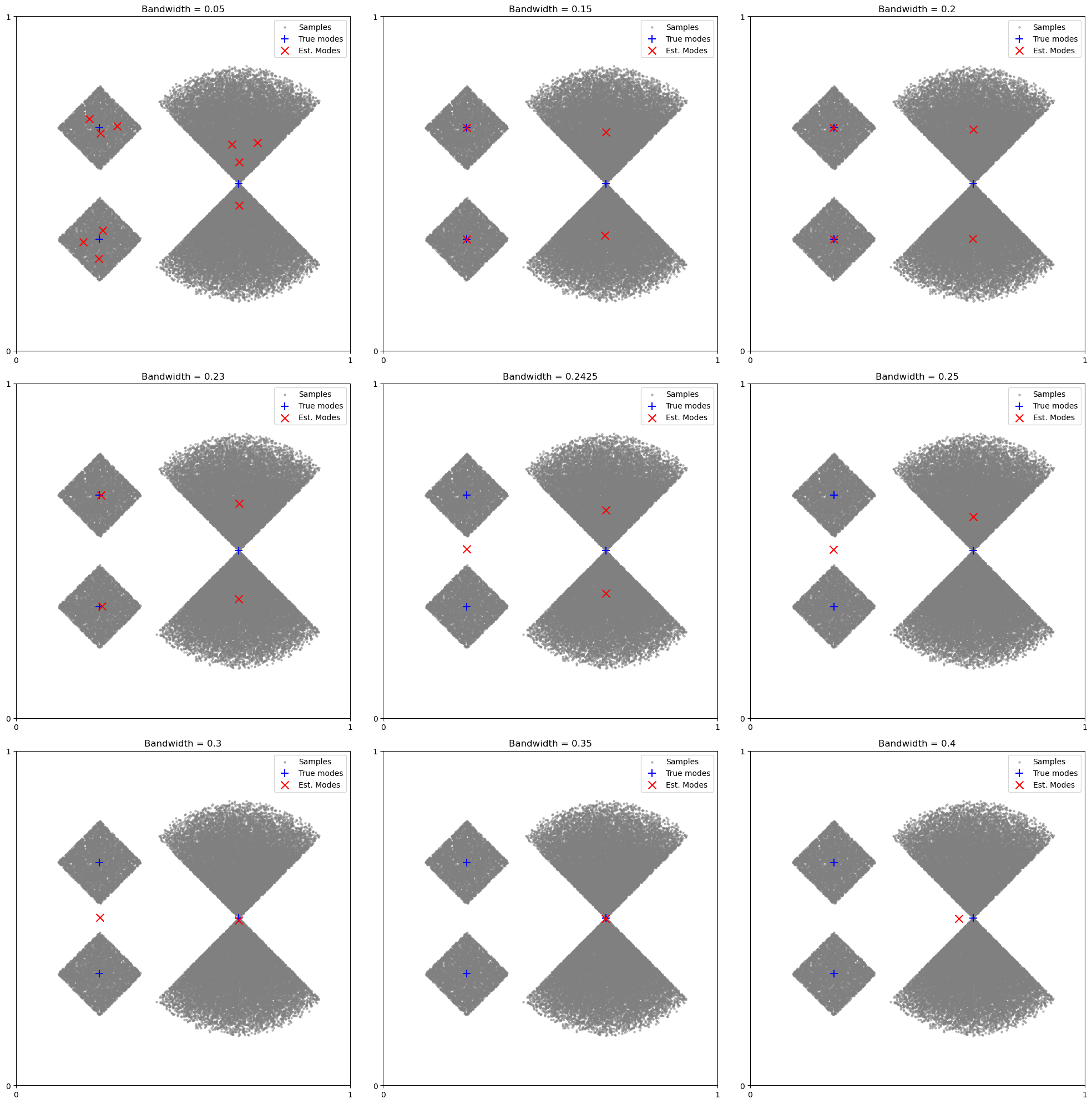}
    \caption{Behaviors of the mean-shift algorithm for various bandwidths \( b \in \{0.05,\, 0.15,\, 0.2,\, 0.23,\, 0.2425,\, 0.25,\, 0.3,\, 0.35,\, 0.4\} \). The algorithm fails to properly recover the modes for all bandwidth values.}
    \label{fig:meanshift}
\end{figure}
\begin{figure}
    \centering
    \includegraphics[scale=0.4]{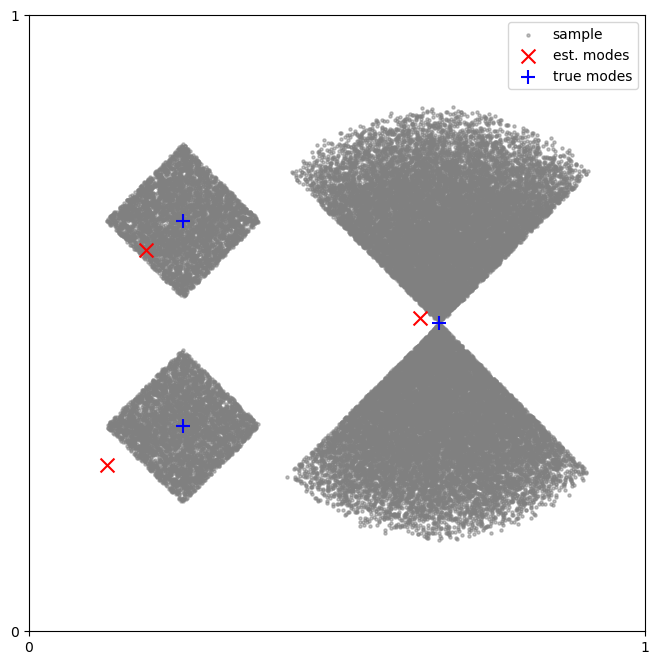}
    \caption{Behavior of the proposed persistence-based approach. The number of modes is exactly recovered and the modes relatively well located.}
    \label{fig:ours}
\end{figure}
To quantify this observation and ensure that it is not merely an anomaly specific to the sampled data, we perform a simulation over \( r = 100 \) samples. For each sample, we evaluate whether the correct number of modes has been recovered, calculate the average number of modes estimated, and compute the error in Hausdorff distance on the mode locations. The results are presented in Table~\ref{tab:results} and confirm our initial observations, with our method outperforming the mean-shift algorithm. Our method identifies the exact number of modes with an accuracy of 89\%, whereas the mean-shift fails to do so for all the \( r = 100 \) simulations, except in the case where \( b = 0.2425 \), which still results in a low accuracy of 37\%.  Regarding the error in Hausdorff distance, the error is consistently smaller for our approach compared to the mean-shift algorithm, except for the case where \( b = 0.05 \). However, in this case, the number of modes is really poorly estimated.

\begin{table}[H]
\begin{tabular}{cl|c|c|c|lllll}
\cline{3-5}
\multicolumn{1}{l}{}                              &          & \begin{tabular}[c]{@{}c@{}}Exact modes number\\  identification\end{tabular} & \begin{tabular}[c]{@{}c@{}}Avg. numbers of\\ modes\end{tabular} & \begin{tabular}[c]{@{}c@{}}Avg. Hausdorff\\ distance\end{tabular} &  &  &  &  &  \\ \cline{1-5}
\multicolumn{2}{|c|}{Persistence-based}                       & 89\%                                                                         & 3.11                                                            & 0.1294                                                            &  &  &  &  &  \\ \cline{1-5}
\multicolumn{1}{|c|}{\multirow{9}{*}{Mean-Shift}} & b=0.05   & 0\%                                                                          & 7.95                                                            & 0.1012                                                            &  &  &  &  &  \\ \cline{2-5}
\multicolumn{1}{|c|}{}                            & b=0.15   & 0\%                                                                          & 4                                                               & 0.1535                                                            &  &  &  &  &  \\ \cline{2-5}
\multicolumn{1}{|c|}{}                            & b=0.2    & 0\%                                                                          & 4                                                               & 0.1627                                                            &  &  &  &  &  \\ \cline{2-5}
\multicolumn{1}{|c|}{}                            & b=0.22   & 0\%                                                                          & 4                                                               & 0.1528                                                            &  &  &  &  &  \\ \cline{2-5}
\multicolumn{1}{|c|}{}                            & b=0.2425 & 37\%                                                                         & 2.51                                                            & 0.1673                                                            &  &  &  &  &  \\ \cline{2-5}
\multicolumn{1}{|c|}{}                            & b=0.25   & 0\%                                                                          & 2                                                               & 0.1719                                                            &  &  &  &  &  \\ \cline{2-5}
\multicolumn{1}{|c|}{}                            & b=0.3    & 0\%                                                                          & 1.65                                                            & 0.2673                                                            &  &  &  &  &  \\ \cline{2-5}
\multicolumn{1}{|c|}{}                            & b=0.35   & 0\%                                                                          & 1                                                               & 0.4481                                                            &  &  &  &  &  \\ \cline{2-5}
\multicolumn{1}{|c|}{}                            & b=0.4    & 0\%                                                                          & 1                                                               & 0.4102                                                            &  &  &  &  &  \\ \cline{1-5}
\end{tabular}
\caption{Simulation results over $r=100$ samples.}
\label{tab:results}
\end{table}
\section{Discussion}
Exploiting the link between modes and persistence, we propose an estimator that consistently infers the number of modes, their locations, and the associated local maxima under mild assumptions. The central contribution of our work is the identification of a threshold for mode separation (and equivalently prominence): below this threshold, mode detection is impossible, while above it, our procedure achieves optimal (or near-optimal) performance in the minimax sense.

Along the way, we extend the approach and results of~\cite{Henneuse24a} on \( H_0 \)-persistence diagram estimation in the density model. These findings are of independent interest, as they broaden the applicability of persistent homology inference for densities, which has so far been largely restricted to regular settings. Our results demonstrate strong robustness to discontinuities in \( H_0 \)-persistence diagram estimation. We are currently working on generalizing these methods to higher-order homology.

Despite its theoretical appeal and the encouraging simulations we ran, the proposed method has some practical limitations. In particular, it lacks adaptivity, as it relies on prior knowledge of the parameters \( \alpha \), \( \mu \), and \( \delta \), which are typically unknown in real-world scenarios. The dependence on \( \alpha \) can likely be mitigated using standard adaptive estimation techniques, such as Lepski’s method~\cite{Lepski91}. To reduce the dependence on \( \delta \), we provide a selection procedure in Appendix~\ref{apdx:complement modes}. We show that, with this additional selection step and under a slightly stronger threshold condition \( h_0 \), it is still possible to estimate the modes and associated local maxima at minimax rates. This procedure is also used in the simulations presented in Section~\ref{Section: Numerical illustration}, where it yields good empirical results. The dependence on \( \mu \), however, appears more problematic, and we currently have no clear approach to address it. Nevertheless, it is worth noting that this lack of adaptivity is not specific to our method but is a common limitation of existing multiple mode estimators. Developing adaptive strategies for multimodal inference, possibly through adaptive variants of our procedure, thus remains a promising direction for future research.

\section*{Acknowledgement}
The author would like to thank Frédéric Chazal and Pascal Massart for our (many) helpful discussions. The author acknowledges the support of the ANR TopAI chair (ANR–19–CHIA–0001)

\bibliographystyle{plainnat}
\bibliography{bibliographie}

\begin{thebibliography}{59}
\providecommand{\natexlab}[1]{#1}
\providecommand{\url}[1]{\texttt{#1}}
\expandafter\ifx\csname urlstyle\endcsname\relax
  \providecommand{\doi}[1]{doi: #1}\else
  \providecommand{\doi}{doi: \begingroup \urlstyle{rm}\Url}\fi

\bibitem[Abraham et~al.(2003)Abraham, Biau, and Cadre]{Abraham03}
Christophe Abraham, Gérard Biau, and Benoît Cadre.
\newblock Simple estimation of the mode of a multivariate density.
\newblock \emph{Canadian Journal of Statistics}, 31\penalty0 (1):\penalty0 23--34, 2003.

\bibitem[Abraham et~al.(2004)Abraham, Biau, and Cadre]{Abrham04}
Christophe Abraham, Gérard Biau, and Benoît Cadre.
\newblock On the asymptotic properties of a simple estimate of the mode.
\newblock \emph{ESAIM: PS}, 8:\penalty0 1--11, 2004.

\bibitem[Ameijeiras-Alonso et~al.(2019)Ameijeiras-Alonso, Crujeiras, and Rodr{\'i}guez-Casal]{AmeijeirasAlonso2016}
Jose Ameijeiras-Alonso, Rosa~M. Crujeiras, and Alberto Rodr{\'i}guez-Casal.
\newblock Mode testing, critical bandwidth and excess mass.
\newblock \emph{TEST}, 28:\penalty0 900--919, 2019.

\bibitem[Arias-Castro et~al.(2016)Arias-Castro, Mason, and Pelletier]{ariascastro16a}
Ery Arias-Castro, David Mason, and Bruno Pelletier.
\newblock On the estimation of the gradient lines of a density and the consistency of the mean-shift algorithm.
\newblock \emph{Journal of Machine Learning Research}, 17\penalty0 (43):\penalty0 1--28, 2016.

\bibitem[Arias-Castro et~al.(2022)Arias-Castro, Qiao, and Zheng]{ariascastro2021estimation}
Ery Arias-Castro, Wanli Qiao, and Lin Zheng.
\newblock {Estimation of the global mode of a density: Minimaxity, adaptation, and computational complexity}.
\newblock \emph{Electronic Journal of Statistics}, 16\penalty0 (1):\penalty0 2774 -- 2795, 2022.

\bibitem[Barannikov(1994)]{Baranikov94}
Serguei Barannikov.
\newblock The framed morse complex and its invariants.
\newblock \emph{Adv. Soviet Math.}, 21:\penalty0 93–115, 1994.

\bibitem[Bauer et~al.(2014)Bauer, Munk, Sieling, and Wardetzky]{Bauer14}
Ulrich Bauer, Axel Munk, Hannes Sieling, and Max Wardetzky.
\newblock Persistent homology meets statistical inference - a case study: Detecting modes of one-dimensional signals.
\newblock 2014.

\bibitem[Bobrowski et~al.(2017)Bobrowski, Mukherjee, and Taylor]{Bobrowski}
Omer Bobrowski, Sayan Mukherjee, and Jonathan~E. Taylor.
\newblock {Topological consistency via kernel estimation}.
\newblock \emph{Bernoulli}, 23\penalty0 (1):\penalty0 288 -- 328, 2017.

\bibitem[Bubenik et~al.(2009)Bubenik, Carlsson, Kim, and Luo]{BCL2009}
Peter Bubenik, Gunnar Carlsson, Peter Kim, and Zhiming Luo.
\newblock Statistical topology via morse theory persistence and nonparametric estimation.
\newblock \emph{Contemporary Mathematics}, 516:\penalty0 75--92, 2009.

\bibitem[Burman and Polonik(2009)]{Burman2009}
Prabir Burman and Wolfgang Polonik.
\newblock Multivariate mode hunting: Data analytic tools with measures of significance.
\newblock \emph{Journal of Multivariate Analysis}, 100\penalty0 (6):\penalty0 1198--1218, 2009.

\bibitem[Carreira-Perpinan(2007)]{PerpinanA}
Miguel Carreira-Perpinan.
\newblock Gaussian mean-shift is an em algorithm.
\newblock \emph{IEEE Transactions on Pattern Analysis and Machine Intelligence}, 29\penalty0 (5):\penalty0 767--776, 2007.

\bibitem[Carreira-Perpinan and Williams(2003)]{PerpinanB}
Miguel Carreira-Perpinan and Christopher Williams.
\newblock On the number of modes of a gaussian mixture.
\newblock In \emph{Scale Space Methods in Computer Vision}, pages 625--640. Springer Berlin Heidelberg, 2003.

\bibitem[Chacón(2012)]{Chacon12}
José Chacón.
\newblock Clusters and water flows: a novel approach to modal clustering through morse theory.
\newblock 2012.

\bibitem[Chacón(2020)]{Chacon20}
José Chacón.
\newblock The modal age of statistics.
\newblock \emph{International Statistical Review}, 88, 2020.

\bibitem[Chazal et~al.(2009)Chazal, Cohen-Steiner, Glisse, Guibas, and Oudot]{Chazal2009}
Fr\'{e}d\'{e}ric Chazal, David Cohen-Steiner, Marc Glisse, Leonidas~J. Guibas, and Steve~Y. Oudot.
\newblock Proximity of persistence modules and their diagrams.
\newblock In \emph{Proceedings of the Twenty-Fifth Annual Symposium on Computational Geometry}, SCG '09, page 237–246, 2009.

\bibitem[Chazal et~al.(2016)Chazal, Oudot, Glisse, and de~Silva]{chazal2013}
Fr{\'e}d{\'e}ric Chazal, Steve Oudot, Marc Glisse, and Vin de~Silva.
\newblock \emph{{The Structure and Stability of Persistence Modules}}.
\newblock SpringerBriefs in Mathematics. {Springer Verlag}, 2016.

\bibitem[Chazal and Michel(2021)]{chazal2021introduction}
Frédéric Chazal and Bertrand Michel.
\newblock An introduction to topological data analysis: fundamental and practical aspects for data scientists, 2021.

\bibitem[Chazal et~al.(2006)Chazal, Cohen-Steiner, and Lieutier]{mureach}
Frédéric Chazal, David Cohen-Steiner, and André Lieutier.
\newblock A sampling theory for compact sets in euclidean space.
\newblock \emph{Discrete \& Computational Geometry}, 41:\penalty0 461--479, 2006.

\bibitem[Chazal et~al.(2007)Chazal, Cohen-Steiner, Lieutier, and Thibert]{SSDO07}
Frédéric Chazal, David Cohen-Steiner, André Lieutier, and Boris Thibert.
\newblock Shape smoothing using double offset.
\newblock \emph{Proceedings - SPM 2007: ACM Symposium on Solid and Physical Modeling}, 2007.

\bibitem[Chazal et~al.(2011)Chazal, Guibas, Oudot, and Skraba]{ChazalGuibasOudotPrimoz11}
Frédéric Chazal, Leonidas Guibas, Steve Oudot, and Primoz Skraba.
\newblock Persistence-based clustering in riemannian manifolds.
\newblock \emph{Journal of the ACM}, 60, 2011.

\bibitem[Cheng and Hall(1998)]{ChengHall98}
Ming-Yen Cheng and Peter Hall.
\newblock Calibrating the excess mass and dip tests of modality.
\newblock \emph{Journal of the Royal Statistical Society. Series B (Statistical Methodology)}, 60\penalty0 (3):\penalty0 579--589, 1998.

\bibitem[Cheng(1995)]{Cheng95}
Yizong Cheng.
\newblock Mean shift, mode seeking, and clustering.
\newblock \emph{IEEE Transactions on Pattern Analysis and Machine Intelligence}, 17\penalty0 (8):\penalty0 790--799, 1995.

\bibitem[Chernoff(1964)]{Chernoff64}
Herbert Chernoff.
\newblock Estimation of the mode.
\newblock \emph{Annals of the Institute of Statistical Mathematics}, 16, 1964.

\bibitem[Cohen-Steiner et~al.(2005)Cohen-Steiner, Edelsbrunner, and Harer]{CSEH2005}
David Cohen-Steiner, Herbert Edelsbrunner, and John Harer.
\newblock Stability of persistence diagrams.
\newblock \emph{Discrete and Computational Geometry - DCG}, 37:\penalty0 263--271, 2005.

\bibitem[Comaniciu and Meer(2002)]{Comanicu02}
Dorin Comaniciu and Peter Meer.
\newblock Mean shift: a robust approach toward feature space analysis.
\newblock \emph{IEEE Transactions on Pattern Analysis and Machine Intelligence}, 24\penalty0 (5):\penalty0 603--619, 2002.

\bibitem[Crawley-Boevey(2012)]{Crawley2012}
William Crawley-Boevey.
\newblock Decomposition of pointwise finite-dimensional persistence modules, 2012.

\bibitem[Dasgupta and Kpotufe(2014)]{DasguptaKoptufe14}
Sanjoy Dasgupta and Samory Kpotufe.
\newblock Optimal rates for k-nn density and mode estimation.
\newblock In \emph{Advances in Neural Information Processing Systems}, volume~27. Curran Associates, Inc., 2014.

\bibitem[Devroye(1979)]{Devroye79}
Luc Devroye.
\newblock Recursive estimation of the mode of a multivariate density.
\newblock \emph{Canadian Journal of Statistics}, 7\penalty0 (2):\penalty0 159--167, 1979.

\bibitem[Dlotko(2025)]{gudhi:CubicalComplex}
Pawel Dlotko.
\newblock Cubical complex.
\newblock In \emph{GUDHI User and Reference Manual}. GUDHI Editorial Board, 3.11.0 edition, 2025.
\newblock URL \url{https://gudhi.inria.fr/doc/3.11.0/group__cubical__complex.html}.

\bibitem[Donoho and Liu(1991)]{DonohoLiu91}
David Donoho and Richard Liu.
\newblock {Geometrizing Rates of Convergence, III}.
\newblock \emph{The Annals of Statistics}, 19\penalty0 (2):\penalty0 668 -- 701, 1991.

\bibitem[Duong et~al.(2008)Duong, Cowling, Koch, and Wand]{Duong08}
Tarn Duong, Arianna Cowling, Inge Koch, and Matt Wand.
\newblock Feature significance for multivariate kernel density estimation.
\newblock \emph{Computational Statistics \& Data Analysis}, 52:\penalty0 4225--4242, 2008.

\bibitem[Eddy(1980)]{Eddy80}
William Eddy.
\newblock {Optimum Kernel Estimators of the Mode}.
\newblock \emph{The Annals of Statistics}, 8\penalty0 (4):\penalty0 870 -- 882, 1980.

\bibitem[Federer(1959)]{Fed59}
Herbert Federer.
\newblock Curvature measures.
\newblock \emph{Transaction of the American Mathematical Society}, 1959.

\bibitem[Fisher and Marron(2001)]{Fisher01}
Nick Fisher and Steve Marron.
\newblock Mode testing via the excess mass estimate.
\newblock \emph{Biometrika}, 88\penalty0 (2):\penalty0 499--517, 2001.

\bibitem[Fukunaga and Hostetler(1975)]{Fukunaga}
Keinosuke Fukunaga and Larry Hostetler.
\newblock The estimation of the gradient of a density function, with applications in pattern recognition.
\newblock \emph{IEEE Transactions on Information Theory}, 21\penalty0 (1):\penalty0 32--40, 1975.

\bibitem[Genovese et~al.(2015)Genovese, Perone-Pacifico, Verdinelli, and Wasserman]{Genovese16}
Christopher Genovese, Marco Perone-Pacifico, Isabella Verdinelli, and Larry Wasserman.
\newblock {Non-Parametric Inference for Density Modes}.
\newblock \emph{Journal of the Royal Statistical Society Series B: Statistical Methodology}, 78\penalty0 (1):\penalty0 99--126, 2015.

\bibitem[GUDHI(2025)]{gudhi:urm}
GUDHI.
\newblock \emph{GUDHI User and Reference Manual}.
\newblock GUDHI Editorial Board, 3.11.0 edition, 2025.
\newblock URL \url{https://gudhi.inria.fr/doc/3.11.0/}.

\bibitem[Hall and York(2001)]{HallYork01}
Peter Hall and Matthew York.
\newblock On the calibration of silverman's test for multimodality.
\newblock \emph{Statistica Sinica}, 11\penalty0 (2):\penalty0 515--536, 2001.

\bibitem[Hartigan and Hartigan(1985)]{Hartigan85}
John Hartigan and P.~M. Hartigan.
\newblock {The Dip Test of Unimodality}.
\newblock \emph{The Annals of Statistics}, 13\penalty0 (1):\penalty0 70 -- 84, 1985.

\bibitem[Hatcher(2000)]{Hatcher}
Allen Hatcher.
\newblock \emph{{Algebraic topology}}.
\newblock Cambridge Univ. Press, Cambridge, 2000.

\bibitem[Henneuse(2024)]{Henneuse24a}
Hugo Henneuse.
\newblock Persistence diagram estimation of multivariate piecewise hölder-continuous signals.
\newblock 2024.

\bibitem[Jiang and Kpotufe(2017)]{jiang17}
Heinrich Jiang and Samory Kpotufe.
\newblock {Modal-set estimation with an application to clustering}.
\newblock In \emph{Proceedings of the 20th International Conference on Artificial Intelligence and Statistics}, volume~54 of \emph{Proceedings of Machine Learning Research}, pages 1197--1206. PMLR, 2017.

\bibitem[Kim(2018)]{KimThesis}
Jisu Kim.
\newblock \emph{Statistical Inference for Geometric Data}.
\newblock PhD thesis, Carnegie Mellon University, 2018.

\bibitem[Kim et~al.(2020)Kim, Shin, Chazal, Rinaldo, and Wasserman]{kim2020homotopy}
Jisu Kim, Jaehyeok Shin, Frédéric Chazal, Alessandro Rinaldo, and Larry Wasserman.
\newblock Homotopy reconstruction via the cech complex and the vietoris-rips complex, 2020.

\bibitem[Klemelä(2005)]{KlemelaMode}
Jussi Klemelä.
\newblock Adaptive estimation of the mode of a multivariate density.
\newblock \emph{Journal of Nonparametric Statistics}, 17\penalty0 (1):\penalty0 83--105, 2005.

\bibitem[Konakov(1974)]{Konakov}
Valentin Konakov.
\newblock On the asymptotic normality of the mode of multidimensional distributions.
\newblock \emph{Theory of Probability \& Its Applications}, 18\penalty0 (4):\penalty0 794--799, 1974.

\bibitem[Lepski(1991)]{Lepski91}
Oleg Lepski.
\newblock On a problem of adaptive estimation in gaussian white noise.
\newblock \emph{Theory of Probability \& Its Applications}, 35\penalty0 (3):\penalty0 454--466, 1991.

\bibitem[Lepski(1992)]{Lepski92}
Oleg Lepski.
\newblock Asymptotically minimax adaptive estimation. i: Upper bounds. optimally adaptive estimates.
\newblock \emph{Theory of Probability \& Its Applications}, 36\penalty0 (4):\penalty0 682--697, 1992.

\bibitem[Li et~al.(2007)Li, Ray, and Lindsay]{Li07}
Jia Li, Surajit Ray, and Bruce Lindsay.
\newblock A nonparametric statistical approach to clustering via mode identification.
\newblock \emph{Journal of Machine Learning Research}, 8:\penalty0 1687--1723, 2007.

\bibitem[Minnotte(1997)]{Minotte97}
Michael Minnotte.
\newblock {Nonparametric testing of the existence of modes}.
\newblock \emph{The Annals of Statistics}, 25\penalty0 (4):\penalty0 1646 -- 1660, 1997.

\bibitem[Mokkadem and Pelletier(2003)]{Mokkadem03}
Abdelkader Mokkadem and Mariane Pelletier.
\newblock The law of the iterated logarithm for the multivariate kernel mode estimator.
\newblock \emph{Esaim: Probability and Statistics}, 7:\penalty0 1--21, 2003.

\bibitem[Parzen(1962)]{Parzen62}
Emanuel Parzen.
\newblock {On Estimation of a Probability Density Function and Mode}.
\newblock \emph{The Annals of Mathematical Statistics}, 33\penalty0 (3):\penalty0 1065 -- 1076, 1962.

\bibitem[Romano(1988)]{Romano88}
Joseph Romano.
\newblock {On Weak Convergence and Optimality of Kernel Density Estimates of the Mode}.
\newblock \emph{The Annals of Statistics}, 16\penalty0 (2):\penalty0 629 -- 647, 1988.

\bibitem[Samanta(1973)]{Samantha73}
Mrityunjay Samanta.
\newblock Nonparametric estimation of the mode of a multivariate density.
\newblock \emph{South African Statistical Journal}, 7:\penalty0 109–117, 1973.

\bibitem[Shin et~al.(2017)Shin, Kim, Rinaldo, and Wasserman]{Shin2017}
Jaehyeok Shin, Jisu Kim, Alessandro Rinaldo, and Larry Wasserman.
\newblock Confidence sets for persistent homology of the kde filtration.
\newblock 2017.

\bibitem[Silverman(1981)]{Silverman81}
Bernard Silverman.
\newblock Using kernel density estimates to investigate multimodality.
\newblock \emph{Journal of the Royal Statistical Society: Series B (Methodological)}, 43\penalty0 (1):\penalty0 97--99, 1981.

\bibitem[Tsybakov(1990)]{Tsybakov90}
Alexandre Tsybakov.
\newblock Recursive estimation of the mode of a multivariate distribution.
\newblock \emph{Problemy Peredachi Informatsii}, 26:\penalty0 38--45, 1990.

\bibitem[Tsybakov(2008)]{TsybakovBook}
Alexandre Tsybakov.
\newblock \emph{Introduction to Nonparametric Estimation}.
\newblock Springer Publishing Company, Incorporated, 2008.

\bibitem[Vieu(1996)]{Vieu96}
Phillipe Vieu.
\newblock A note on density mode estimation.
\newblock \emph{Statistics and Probability Letters}, 26:\penalty0 297--307, 1996.

\end{thebibliography}
\appendix
\section{Auxiliary results on persistence diagram estimation}
\label{sec: aux result}
In this section, we provide complementary results to Proposition~\ref{MainProp}. Namely, we establish Corollary~\ref{estimation borne sup}, which provides consistency guarantees in expectation, and Proposition~\ref{Lower Bound Diag}, in which we derive minimax lower bounds matching the rates obtained in Proposition~\ref{MainProp}, thereby proving the optimality of our procedures.
\begin{crllr}
\label{estimation borne sup}
Let $h\asymp\left(\frac{\log(n)}{n}\right)^{\frac{1}{d+2\alpha}}$, we have : 
$$\underset{f\in S_{d}(L,\alpha,\mu,R_{\mu})}{\sup}\quad\mathbb{E}\left(d_{\infty}\left(\widehat{\operatorname{dgm}(f)},\operatorname{dgm}(f)\right)\right)\lesssim \left(\frac{\log(n)}{n}\right)^{\frac{\alpha}{d+2\alpha}}$$
\end{crllr}
\begin{proof}
  The sub-Gaussian concentration provided by Proposition \ref{MainProp}, gives that, for all $A\geq 0$,
$$\mathbb{P}\left(\frac{d_{b}\left(\widehat{\operatorname{dgm}(f)},\operatorname{dgm}(f)\right)}{h^{\alpha}}\geq A\right)\leq \tilde{c}_{0}\exp\left(-\tilde{c}_{1}A^{2}\right).$$
 Now, we have,
 \begin{align*}
     \quad\mathbb{E}\left(\frac{d_{b}\left(\widehat{\operatorname{dgm}(f)},\operatorname{dgm}(f)\right)}{h^{p\alpha}}\right)&=\int_{0}^{+\infty}\mathbb{P}\left(\frac{d_{b}\left(\widehat{\operatorname{dgm}(f)},\operatorname{dgm}(f)\right)}{h^{\alpha}}\geq A\right)dA\\
     &\leq \int_{0}^{+\infty}\tilde{c}_{0}\exp\left(-\tilde{c}_{1}A^{2}\right)dA<+\infty.
\end{align*}
\end{proof}
\begin{prpstn}
\label{Lower Bound Diag}
    $$\underset{\widehat{\operatorname{dgm}(f)}}{\inf}\quad\underset{f\in S_{d}(L,\alpha,\mu,R_{\mu})}{\sup}\quad\mathbb{E}\left(d_{\infty}\left(\widehat{\operatorname{dgm}(f)},\operatorname{dgm}(f)\right)\right)\gtrsim \left(\frac{\log(n)}{n}\right)^{\frac{\alpha}{d+2\alpha}}$$
where the infimum is taken over all the estimators of $\operatorname{dgm}(f)$.
\end{prpstn}
The proof of Proposition \ref{Lower Bound Diag} follows essentially as the proof of Theorem 3 from \cite{Henneuse24a}. This follows from the minimax lower bound technique presented in Theorem 2.6 of \cite{TsybakovBook}. The idea is, for all, $$r_{n}=o\left(\left(\frac{\log(n)}{n}\right)^{\frac{\alpha}{d+2\alpha}}\right)$$
to exhibit a finite collection of functions in $S_{d}(L,\alpha,\mu,R_{\mu})$ such that their persistence diagrams are pairwise at distance $2r_{n}$ but indistinguishable, with high certainty.
\begin{proof}
 For $m$ integer in $\left[0, \lfloor1/h\rfloor\right]$, we define,
$$f_{h,m}(x_{1},...,x_{d})=\frac{L}{4}x_{1}^{\alpha}+\frac{L}{2}\left(h^{\alpha}-\|(x_{1},...,x_{d})-m/\lfloor1/h\rfloor(1,...,1)\|^{\alpha}_{\infty}\right)_{+}$$
and 
$$\Tilde{f}_{m,h}=\frac{1+f_{h,m}}{\|1+f_{h,m}\|_{1}}$$
The $f_{h,m}$ are $(L,\alpha)-$Hölder. As for all $m$, $f_{h,m}$ are positive functions, 
$$\|1+f_{h,m}\|_{1}\geq 1$$
consequently, $\Tilde{f}_{h,m}$ is $(L,\alpha)-$Hölder. As by construction  $\Tilde{f}_{h,m}$ are probability densities, they belong to $S_{d}(L,\alpha,\mu,R_{\mu})$ for all $\mu\in]0,1]$ and $R_{\mu}>0$.\\\\
We have, for sufficiently small $h$ and all $\frac{1}{4}\lfloor1/h\rfloor<m<\frac{3}{4}\lfloor1/h\rfloor$, 
$$\operatorname{dgm}(f_{h,m})=\left\{(L/4,-\infty),\left(\frac{L}{4}\left(\frac{m}{\lfloor1/h\rfloor}\right)^{\alpha}-\frac{L}{4}h^{\alpha},\frac{L}{4}\left(\frac{m}{\lfloor1/h\rfloor}\right)^{\alpha}-\frac{L}{2}h^{\alpha}\right)\right\}$$
and consequently,
\begin{align*}
&\operatorname{dgm}(\Tilde{f}_{h,m})=\left\{\left(\frac{L}{4\|1+f_{h,m}\|_{1}},-\infty\right),\left(\frac{1+\frac{L}{4}\left(\frac{m}{\lfloor1/h\rfloor}\right)^{\alpha}-\frac{L}{4}h^{\alpha}}{\|1+f_{h,m}\|_{1}},\frac{1+\frac{L}{4}\left(\frac{m}{\lfloor1/h\rfloor}\right)^{\alpha}-\frac{L}{2}h^{\alpha}}{\|1+f_{h,m}\|_{1}}\right)\right\}.
\end{align*}
As for all $\frac{1}{4}\lfloor1/h\rfloor<m,m^{'}<\frac{3}{4}\lfloor1/h\rfloor$, $\|1+f_{h,m}\|_{1}=\|1+f_{h,m^{'}}\|_{1}$, we then have, for all $\frac{1}{4}\lfloor1/h\rfloor<m\ne m^{'}<\frac{3}{4}\lfloor1/h\rfloor$,
\begin{equation}
\label{hypothesis separate}
d_{\infty}\left(\operatorname{dgm}(f_{h,m}),\operatorname{dgm}(f_{h,m^{'}})\right)\geq \frac{Lh^{\alpha}}{4\|1+f_{h,m}\|_{1}}\geq \frac{Lh^{\alpha}}{4(1+L)}
\end{equation}
We set $r_{n}=\frac{Lh^{\alpha}}{8(1+L)}$, then, by (\ref{hypothesis separate}), for all $\frac{1}{4}\lfloor1/h\rfloor<m\ne m^{'}<\frac{3}{4}\lfloor1/h\rfloor$,
$$d_{\infty}\left(\operatorname{dgm}(f_{h,m}),\operatorname{dgm}(f_{h,m^{'}})\right)\geq 2r_{n}$$
For a fixed signal $f$, denote $\mathbb{P}^{n}_{f}$ the joint probability distribution of $X_{1},...,X_{n}$. From Theorem 2.6 of \cite{TsybakovBook}, it now suffices to show that if $r_{n}=o\left(\left(\sfrac{\log(n)}{n}\right)^{\frac{p\alpha}{d+2\alpha}}\right)$, then, for a fixed $m$,
\begin{align}
\label{chi2}
  &\left(\frac{2}{\left\lfloor \frac{1}{h}\right\rfloor-2}\right)^{2}\sum\limits_{\frac{1}{4}\lfloor1/h\rfloor<m^{'} <\frac{3}{4}\lfloor1/h\rfloor,m^{'}\ne m}\chi^{2}\left(\mathbb{P}^{n}_{\Tilde{f}_{h,m^{'}}}\mathbb{P}^{n}_{\Tilde{f}_{h,m}}\right)\\
  =&\left(\frac{2}{\left\lfloor \frac{1}{h}\right\rfloor-2}\right)^{2}\sum\limits_{\frac{1}{4}\lfloor1/h\rfloor<m^{'} <\frac{3}{4}\lfloor1/h\rfloor,m^{'}\ne m}\mathbb{E}_{\mathbb{P}^{n}_{\Tilde{f}_{h,m}}}\left[\left(\frac{d\mathbb{P}^{n}_{\Tilde{f}_{h,m^{'}}}}{d\mathbb{P}^{n}_{\Tilde{f}_{h,m}}}\right)^{2}\right]-1  
\end{align}
converges to zero when $n$ converges to infinity. Note that,
\begin{align*}
\mathbb{E}_{\mathbb{P}^{n}_{\Tilde{f}_{h,m}}}\left(\left(\frac{d\mathbb{P}^{n}_{\Tilde{f}_{h,m^{'}}}}{d\mathbb{P}^{n}_{\Tilde{f}_{h,m}}}\right)^{2}\right)&=\mathbb{E}_{\mathbb{P}^{n}_{\Tilde{f}_{h,m}}}\left(\frac{\prod\limits_{i=1}^{n}\Tilde{f}^{2}_{h,m^{'}}(X_{i})}{\prod\limits_{i=1}^{n}\Tilde{f}^{2}_{h,m}(X_{i})}\right)\nonumber\\
&=\prod\limits_{i=1}^{n}\mathbb{E}_{\mathbb{P}_{\Tilde{f}_{h,m}}}\left(\frac{\Tilde{f}^{2}_{h,m^{'}}(X_{i})}{f^{2}_{d,h,m^{1},L,\alpha}(X_{i})}\right)\nonumber\\
&=\left(\mathbb{E}_{\mathbb{P}_{\Tilde{f}_{h,m}}}\left(\frac{\Tilde{f}^{2}_{h,m^{'}}(X_{1})}{\Tilde{f}^{2}_{h,m}(X_{1})}\right)\right)^{n}\nonumber.\\
\end{align*}
We denote $H_{m}$ the hypercube defined by $\|(x_{1},...,x_{d})-m/\lfloor1/h\rfloor(1,...,1)\|\leq h$. Remarking that,
\begin{align*}
\Tilde{f}_{h,m^{'}}=\Tilde{f}_{h,m}&+\frac{L}{2\|1+f_{h,m}\|_{1}}\left(h^{\alpha}-\|(x_{1},...,x_{d})-m/\lfloor1/h\rfloor(1,...,1)\|^{\alpha}_{\infty}\right)_{+}\\
&-\frac{L}{2\|1+f_{h,m}\|_{1}}\left(h^{\alpha}-\|(x_{1},...,x_{d})-m^{'}/\lfloor1/h\rfloor(1,...,1)\|^{\alpha}_{\infty}\right)_{+}
\end{align*}
and 
$$\Tilde{f}_{h,m}(x)\geq \frac{1}{1+L}$$
We then have,
\begin{align}
&\mathbb{E}_{\mathbb{P}_{\Tilde{f}_{h,m}}}\left(\frac{\Tilde{f}^{2}_{h,m^{'}}(X_{1})}{\Tilde{f}^{2}_{h,m}(X_{1})}\right)\nonumber\\
&=1+\frac{L}{\|1+f_{h,m}\|_{1}}\int_{H_{m}}h^{\alpha}-\|(x_{1},...,x_{d})-m/\lfloor1/h\rfloor(1,...,1)\|^{\alpha}_{\infty}dt_{1}...dt_{d}\nonumber\\
&\quad-\frac{L}{\|1+f_{h,m}\|_{1}}\int_{H_{m^{'}}}h^{\alpha}-\|(x_{1},...,x_{d})
m^{'}/\lfloor1/h\rfloor(1,...,1)\|^{\alpha}_{\infty}dt_{1}...dt_{d}\nonumber\\
&\quad + \frac{L^{2}}{4\|1+f_{h,m}\|^{2}_{1}}\int_{H_{m}}\frac{\left(h^{\alpha}-\|(x_{1},...,x_{d})-m/\lfloor1/h\rfloor(1,...,1)\|^{\alpha}_{\infty}\right)^{2}}{\Tilde{f}_{h,m}}dt_{1}...dt_{d}\nonumber\\
&\quad + \frac{L^{2}}{4\|1+f_{h,m}\|^{2}_{1}}\int_{H_{m^{'}}}\frac{\left(h^{\alpha}-\|(x_{1},...,x_{d})-m^{'}/\lfloor1/h\rfloor(1,...,1)\|^{\alpha}_{\infty}\right)^{2}}{\Tilde{f}_{h,m}}dt_{1}...dt_{d}\nonumber\\
&= 1+ \frac{L^{2}}{2\|1+f_{h,m}\|^{2}_{1}}\int_{H_{m}}\frac{\left(h^{\alpha}-\|(x_{1},...,x_{d})-m/\lfloor1/h\rfloor(1,...,1)\|^{\alpha}_{\infty}\right)^{2}}{\Tilde{f}_{h,m}}dt_{1}...dt_{d}\nonumber\\
&\leq 1+ L^{2}(1+L)\int_{H_{m}}\left(h^{\alpha}-\|(x_{1},...,x_{d})-m/\lfloor1/h\rfloor(1,...,1)\|^{\alpha}_{\infty}\right)^{2}dt_{1}...dt_{d}\nonumber\\
&\leq 1+ L^{2}(1+L)\int_{H_{m}}h^{2\alpha}dt_{1}...dt_{d}\nonumber\\
&\quad+L^{2}(1+L)\int_{H_{m}}\|(x_{1},...,x_{d})-m/\lfloor1/h\rfloor(1,...,1)\|^{2\alpha}_{\infty}dt_{1}...dt_{d}\nonumber\\
&\leq 1+ 2L^{2}(1+L)h^{2\alpha+d}=1+O(h^{2\alpha+d})\nonumber
\end{align}
Hence, if $h=o\left(\left(\frac{\log(n)}{n}\right)^{\frac{1}{d+2\alpha}}\right)$, we have that (\ref{chi2}) converges to zero when $n$ converges to infinity. Consequently, if $r_{n}=o\left(\left(\frac{\log(n)}{n}\right)^{\frac{\alpha}{d+2\alpha}}\right)$, then $h=o\left(\left(\frac{\log(n)}{n}\right)^{\frac{1}{d+2\alpha}}\right)$ and we get the conclusion.
\end{proof}
\section{Complement on modes inference : handling unknown $\delta$}
\label{apdx:complement modes}
In this section we adapt the procedure for modes inference presented in \ref{sec: Proc and result} to the more realistic case where $\delta$ is unknown. The idea is selecting $\delta$ through penalization. The procedure remains the same, except for the construction of $\hat{k}$. Let,
\begin{equation}
\label{eq: minimization}
    R(\delta)=d_{\infty}\left(\widehat{\operatorname{dgm}(f)},\overline{\operatorname{dgm}(f)}^{\delta}\right)+\frac{h^{\alpha}}{\delta}
\end{equation}
with
$$\overline{\operatorname{dgm}(f)}^{\delta}=\left\{(b,d)\in\widehat{\operatorname{dgm}(f)}\text{ s.t }b-d\geq \delta\right\}.$$
We consider,
$$\hat{\delta}=\operatorname{argmin}_{\delta\in]0,+\infty[} R(\delta)=\operatorname{argmin}_{\delta=|b-d|, (b,d)\in \widehat{\operatorname{dgm}(f)}} R(\delta)$$
the existence of this minimum is proved in Lemma \ref{adaptive length lemma} and $\hat{k}$ is then defined as,
$$\hat{k}=\left|\left\{(b,d)\in\widehat{\operatorname{dgm}\left(f\right)}, b-d>\hat{\delta}\right\}\right|.$$
With this choice of $\hat{\delta}$, we now show that, under the assumption :$$\delta=Ch_{0}^{\alpha}>\sqrt{\frac{12}{7}h^{\alpha}}\text{, }h\asymp \left(\sfrac{\log(n)}{n}\right)^{\frac{1}{d+2\alpha}}$$ 
our procedure allows us to infer, with high probability, the exact number of modes, and to estimate their locations and associated maxima at the minimax rates (up to a logarithmic factor), as in Theorem~\ref{coroModes1}. This additional assumption requires  that $h_{0}$ is of order at least $\asymp \left(\sfrac{\log(n)}{n}\right)^{\frac{1}{2(d+2\alpha)}}$ or equivalently that $\delta$ is of order at least $\left(\sfrac{\log(n)}{n}\right)^{\frac{\alpha}{2(d+2\alpha)}}$. In particular, this covers the case where the parameters $L$, $\alpha$, $\mu$, $R_{\mu}$, $C$ and $h_0$ are fixed and $n$ is sufficiently large. 
\begin{thrm}
\label{ThmModes2}
Let $f\in S_{d}(L,\alpha,\mu,R_{\mu},C,h_{0})$, $h\simeq\left(\sfrac{\log(n)}{n}\right)^{\frac{1}{d+2\alpha}}$ and suppose $\delta=Ch_{0}^{\alpha}>\sqrt{\frac{12}{7}h^{\alpha}}$. There exist $C^{'}_{0}$,..., $C^{'}_{3}$ depending only on $L$ ,$\alpha$, $\mu$, $R_{\mu}$, $C$ such that, for all $A\geq 0$ with probability at least $1-C^{'}_{0}\exp(-C^{'}_{1}A^{2})-C^{'}_{2}\exp(-C^{'}_{3}(h_{0}/h)^{-\alpha})$,
$$k=\hat{k}$$
and, for all $i\in\{1,...,k\}$, there exists distinct $(\hat{x}_{i},\hat{m}_{i})$ such that,
$$\|\hat{x}_{i}-x_{i}\|_{\infty}\leq A h$$
and 
$$|\hat{m}_{i}-m_{i}|\leq A h^{\alpha}.$$
\end{thrm}
The key idea to prove this result is contained in the following lemma. Let $\delta^{*}$, the distance between $\operatorname{dgm(f)}$ and the diagonal and $\Tilde{\delta}$ the distance between the set,
$$\left\{(b,d)\in\widehat{\operatorname{dgm}(f)}\text{ s.t. }b-d>(K_{1}+K_{2})h^{\alpha}\right\}$$
and the diagonal. Note that $\delta^{*}$ and $\Tilde{\delta}$ are possibly infinite. 
\begin{lmm}
\label{adaptive length lemma}
Let $h\simeq\left(\sfrac{\log(n)}{n}\right)^{\frac{1}{d+2\alpha}}$, and $E_{5}$ the event :
$$\max\left((K_{1}+K_{2})^{2},K_{1}+K_{2}\right)h_{\alpha}<1/2.$$
If $\delta^{*}\geq \delta>\sqrt{\frac{12}{7}h^{\alpha}}$, under $E_{5}\cap E_{1}$, $R$ admits a minimum over $]0,1]$ attained only at $\delta=\Tilde{\delta}$. 
\end{lmm}
\begin{proof}
   Under $E_{1}$, by Theorem \ref{MainProp}, 
   \begin{equation*}
   \label{eq: l tilde}
    7\delta^{*}/8\leq \delta^{*}-(K_{1}+K_{2})h^{\alpha}\leq \Tilde{\delta}\leq \delta^{*}+(K_{1}+K_{2})h^{\alpha}\leq 9\delta^{*}/8
   \end{equation*}
Then, we have, under $E_{1}$,
   $$R\left(\Tilde{\delta}\right)\leq (K_{1}+K_{2})h^{\alpha}+\frac{h^{\alpha}}{\Tilde{\delta}}\leq (K_{1}+K_{2})h^{\alpha}+\frac{8h^{\alpha}}{7\delta^{*}}<\frac{\delta^{*}}{8}+\frac{8h^{\alpha}}{7\delta^{*}}$$
   \begin{itemize}
   \item If $0<\delta\leq(K_{1}+K_{2})h^{\alpha}$, under $E_{1}$ and $E_{5}$,
   $$R\left(\delta\right)\geq \frac{h^{\alpha}}{(K_{1}+K_{2})h^{\alpha}}=\frac{1}{(K_{1}+K_{2})}> (K_{1}+K_{2})h^{\alpha}+\frac{8h^{\alpha}}{7\delta^{*}}\geq R\left(\Tilde{\delta}\right).$$
   \item If $(K_{1}+K_{2})h^{\alpha}<\delta<\Tilde{\delta}$, then by definition of $\Tilde{\delta}$,
   $$d_{\infty}\left(\widehat{\operatorname{dgm}(f)},\overline{\operatorname{dgm}(f)}^{\delta}\right)=d_{\infty}\left(\widehat{\operatorname{dgm}(f)},\overline{\operatorname{dgm}(f)}^{\Tilde{\delta}}\right)$$
   and $h^{\alpha}/\delta>h^{\alpha}/\Tilde{\delta}$. Thus, $R(\delta)>R\left(\Tilde{\delta}\right)$.
   \item If $\delta>\Tilde{\delta}$, under $E_{1}$ and as $\delta^{*}\geq \delta>\sqrt{\frac{12}{7}h^{\alpha}}$,
   $$R(\delta)\geq \Tilde{\delta}\geq 7\delta^{*}/8> \frac{\delta^{*}}{8}+\frac{8h^{\alpha}}{7\delta^{*}}\geq R(\Tilde{\delta}).$$
   \end{itemize}
   Hence, in all cases, if $\delta\ne\Tilde{\delta}$, $R(\delta)>R(\Tilde{\delta})$.
\end{proof}
\begin{proof}[Proof of Theorem \ref{ThmModes2}]
    Lemma \ref{adaptive length lemma}, imply that, under $E_{1}\cap E_{5}$, for sufficiently large $n$,
    $$\hat{\delta}=\min(\Tilde{\delta},1)\leq \Tilde{\delta}$$
    Thus, by definition of $\Tilde{\delta}$, 
    \begin{align*}
    d_{\infty}\left(\operatorname{dgm}(f),\overline{\operatorname{dgm}(f)}^{\hat{\delta}}\right)&\leq d_{\infty}\left(\widehat{\operatorname{dgm}(f)},\operatorname{dgm}(f)\right) + d_{\infty}\left(\widehat{\operatorname{dgm}(f)},\overline{\operatorname{dgm}(f)}^{\Tilde{\delta}}\right)\\
    &\leq 2(K_{1}+K_{2})h^{\alpha}.
    \end{align*}
    By \eqref{eq: concentration} there exist $A_{5}$ and $B_{5}$ such that $E_{5}$ occurs with probability $1-B_{5}\exp(-A_{5}h^{-2\alpha})\geq 1-B_{5}\exp(-A_{5}(h_{0}/h)^{2\alpha})$. The proof then follows as in the proof of Theorem \ref{coroModes1}.
    \end{proof}
\section{Proof of Lemma \ref{deformation retract}}
\label{appendix : def retract}
This section is dedicated to the proof of Lemma \ref{deformation retract}.
\begin{proof}[Proof of Lemma \ref{deformation retract}]
The first part of the claim is Theorem 12 of \cite{kim2020homotopy}. A standard technique to construct deformation retraction in differential topology is to exploit the flow coming from a smooth underlying vector field. In \cite{kim2020homotopy}, their idea is to use the vector field defined on  $B_{2}(K,r)\setminus K$, by $W(x)=-\nabla_{d_K}(x)$. But this vector field is not continuous. To overcome this issue, they construct a locally finite covering $(U_{x_{i}})_{i\in \mathbb{N}}$ of $B_{2}(K,r)\setminus K$ and an associated partition of the unity $(\rho_{i})_{{i\in \mathbb{N}}}$, such that $\overline{W}(x)=\sum_{i\in \mathbb{N}}\rho_{i}(x)w(x_{i})$ shares the same dynamic as $W$. More precisely, they show that $\overline{W}$ induces a smooth flow $C$ that can be extended on $B_{2}(K,r)\times [0,+\infty[$ such that for all $x\in B_{2}(K,r)$, for all $t\geq 2r/\mu^{2}$, $C(x,t)=C(x,2r/\mu^{2})\in K$. We make an additional remark, denote $d_{C}$ the arc length distance along $C$, as $||W||\leq 1$, we have,
\begin{align}
d_{C}(x,C(x,2r/\mu^{2}))&=\int_{0}^{2r/\mu^{2}}\left|\frac{\partial }{\partial t}C(x,t)\right|dt\nonumber\\
&\leq\int_{0}^{2r/\mu^{2}}||W(C(x,t)||_{2}dt\nonumber\\
&\leq 2r/\mu^{2}\label{arc length bound}
\end{align} Thus for all $t\in[0,+\infty[$, $||x-C(x,t)||_{2}\leq 2r/\mu^{2}$. Now taking, $F(x,s)=C(x,2rs/\mu^{2})$, provide a deformation retract of $B_{2}(K,r)$ onto $K$ and the associated retraction $R: x\mapsto F(x,1)$ satisfies, $R(x)\in \overline{B}_{2}(x,2r/\mu^{2})$ by (\ref{arc length bound}).
\end{proof}
\section{Proof of Lemma \ref{lmm1b}}
\label{Appendix : lmm1b}
This Appendix is dedicated to the proof of Lemma \ref{lmm1b}.
\begin{proof}[Proof of Lemma \ref{lmm1b}]
Let $f\in S_{d}(L,\alpha,\mu,R_{\mu})$. Let $H\subset \mathcal{F}_{\lambda-2\sqrt{3\kappa}N_{h}h^{\alpha}}^{c}\cap \mathfrak{C}_{h}$ and $H^{'}\subset\mathcal{F}_{\lambda+2\sqrt{3\kappa}N_{h}+h^{\alpha}}\cap \mathfrak{C}_{h}$. note that $f|_{H^{'}}<\lambda-2\sqrt{3\kappa}N_{h}h^{\alpha}$ and $f|_{H}\geq\lambda+2\sqrt{3\kappa}N_{h}h^{\alpha}$. Thus,
$$\mathbb{P}(X_{1}\in H|N_{h})<(\lambda-2\sqrt{3\kappa}N_{h}h^{\alpha})h^{d}$$
and 
$$\mathbb{P}(X_{1}\in H^{'}|N_{h})\geq(\lambda+2\sqrt{3\kappa}N_{h}h^{\alpha})h^{d}.$$
As, supposing $N_{h}$ known, $|X\cap H|$ is $(n,\mathbb{P}(X_{1}\in H|N_{h})-$binomial and $|X\cap H^{'}|$ is $(n,\mathbb{P}(X_{1}\in H^{'}|N_{h})-$binomial we have that,
$$\mathbb{E}\left[\frac{|X\cap H|}{nh^{d}}|N_{h}\right]=\frac{\mathbb{P}(X_{1}\in H|N_{h})}{h^{d}}<\lambda-2\sqrt{3\kappa}N_{h}h^{\alpha}$$
and 
$$\mathbb{E}\left[\frac{|X\cap H^{'}|}{nh^{d}}|N_{h}\right]=\frac{\mathbb{P}(X_{1}\in H^{'}|N_{h})}{h^{d}}\geq\lambda+2\sqrt{3\kappa}N_{h}h^{\alpha}.$$
Now, observe that :
\begin{align*}
\frac{|X\cap H|}{nh^{d}}&= \mathbb{E}\left[\frac{|X\cap H|}{nh^{d}}|N_{h}\right]+\frac{|X\cap H|}{nh^{d}}-\mathbb{E}\left[\frac{|X\cap H|}{nh^{d}}|N_{h}\right]\\
&= \mathbb{E}\left[\frac{|X\cap H|}{nh^{d}}|N_{h}\right]+\frac{|X\cap H|}{nh^{d}}-\mathbb{E}\left[\frac{|X\cap H|}{nh^{d}}\right]-\mathbb{E}\left[\frac{|X\cap H|}{nh^{d}}-\mathbb{E}\left[\frac{|X\cap H|}{nh^{d}}\right]|N_{h}\right].\\
\end{align*}
By definition of $N_{h}$, we have:
$$\left|\frac{|X\cap H|}{nh^{d}}-\mathbb{E}\left[\frac{|X\cap H|}{nh^{d}}\right]\right|\leq N_{h}\sqrt{3\kappa\frac{\log\left(1/h^{d}\right)}{nh^{d}}}$$
and 
$$\mathbb{E}\left[\left|\frac{|X\cap H|}{nh^{d}}-\mathbb{E}\left[\frac{|X\cap H|}{nh^{d}}\right]\right||N_h\right]\leq N_{h}\sqrt{3\kappa\frac{\log\left(1/h^{d}\right)}{nh^{d}}}.$$
Thus, by the choice made for $h$, it follows that :
\begin{align*}
\frac{|X\cap H|}{nh^{d}}&\leq \mathbb{E}\left[\frac{|X\cap H|}{nh^{d}}|N_{h}\right]+2\sqrt{3\kappa}N_{h}\sqrt{\frac{\log\left(1/h^{d}\right)}{nh^{d}}}\\
&\leq \lambda+2\sqrt{3\kappa}N_{h}\left(\sqrt{\frac{\log\left(1/h^{d}\right)}{nh^{d}}}-h^{\alpha}\right)<\lambda
\end{align*}
and similarly,
\begin{align*}
\frac{|X\cap H^{'}|}{nh^{d}}&\geq \mathbb{E}\left[\frac{|X\cap H^{'}|}{nh^{d}}|N_{h}\right]-2\sqrt{3\kappa}N_{h}\sqrt{\frac{\log\left(1/h^{d}\right)}{nh^{d}}}\\
&\geq \lambda+2\sqrt{3\kappa}N_{h}\left(h^{\alpha}-\sqrt{\frac{\log\left(1/h^{d}\right)}{nh^{d}}}\right)>\lambda
\end{align*}
and the proof is complete.
\end{proof}
\section{Proof of Lemma \ref{lemmBound}}
\label{Appendix: Lemmbound}
This Appendix is dedicated to the proof of Lemma \ref{lemmBound}.
\begin{proof}[Proof of Lemma \ref{lemmBound}]
Let $i\in\{1,...,l\}$. First suppose that $d=1$, in this case, $\partial M_{i}$ is the union of two singletons $\{x_{1}\}\cup\{x_{2}\}$. The $\mu$-medial axis is equal here to the set of critical points, which is simply given by $\{(x_{1}+x_{2})/2\}$. Thus, by Assumption \textbf{A3}, we have $1\geq |x_{1}-x_{2}|\geq 2R_{\mu}$. Let $x\in M_{i}$, by Assumption \textbf{A1},
$$2R_{\mu}\left(f(x)-L\right)\leq \int_{x_{1}}^{x_{2}}f(u)du.$$
Hence, as $f$ is a density,
$$\int_{x_{1}}^{x_{2}}f(u)du\leq 1$$
and 
$$f(x)\leq \frac{1}{2R_{\mu}}+L:=\kappa(1,L,\alpha,R_{\mu}).$$
Now, for $d>1$, if the $\mu-$medial axis of  $\partial M_{i}$ intersected with $M_{i}$ is empty, then $M_{i}^{c}$ has infinite $\mu-$reach. By Theorem 12 of \cite{kim2020homotopy}, this implies that for any $r>0$, $\left(M_{i}^{c}\right)^{r}$ is homotopy equivalent to $M_{i}^{c}$. For $r>\sqrt{d}$, $\left(M_{i}^{c}\right)^{r}=[0,1]^{d}$, and thus has trivial $s-$homotopy groups for $s>0$ but $M_{i}^{c}$ contains a non-trivial $d-1$ cycles represented by $\partial M_{i}$, which gives a contradiction. Consequently, $M_{i}$ contains a point $x_{1}$ belonging to the $\mu-$medial axis of  $\partial M_{i}$. By Assumption \textbf{A3}, $B_{2}(x_{1},R_{\mu}/2)$ is contained into $M_{i}$. Thus, Assumption \textbf{A1} ensures that,
$$V_{d}(R_{\mu}/2))(f(x_{1})-L(2R_{\mu})^{\alpha})\leq \int_{B_{2}(x_{1},R_{\mu}/2)}f(u)du$$
with $V_{d}(R_{\mu}/2))$ the volume of the $d-$dimensional Euclidean ball of radius $R_{\mu}/2$. As $f$ is a density,
$$\int_{B_{2}(x_{1},R_{\mu}/2)}f(u)du\leq 1$$
and hence,
$$f(x_{1})\leq \frac{1}{V_{d}(R_{\mu}/2))}+LR_{\mu}^{\alpha}.$$
Now, let $x\in M_{i}$, we have $||x-x_{1}||\leq \sqrt{d}$, using Assumption \textbf{A1} again, it follows that,
$$f(x)\leq f(x_{1})+Ld^{\alpha/2}\leq  \frac{1}{V_{d}(R_{\mu}))}+L(R_{\mu}^{\alpha}+d^{\alpha/2}):=\kappa(d,L,\alpha,R_{\mu}).$$
To conclude, it suffices to note that Assumption \textbf{A2} ensures that,
$$\sup\limits_{x\in \bigcup\limits_{i=1}^{l}M_{i}}f(x)=\sup\limits_{x\in [0,1]^{d}}f(x)$$
and thus for all $x\in[0,1]^{d}$,
$$f(x)\leq \kappa(1,L,\alpha,R_{\mu}).$$
\end{proof}
\section{Proof for the $q-$tameness}
\label{Appendix:q-tame}
Here, we prove the claim that the persistence modules we introduced are $q-$tame and consequently their persistence diagrams are well-defined.
\begin{prpstn}
\label{LmmQtame1}
Let $f\in S_{d}(L,\alpha,\mu,R_{\mu})$ then $f$ is $q$-tame.
\end{prpstn}
\begin{proof}
 Let $s\in\mathbb{N}$ and $\mathbb{V}_{0,f}$ the persistence module (for the $0-$th homology) associated to the superlevel filtration $\mathcal{F}$, and for fixed levels $\lambda>\lambda^{'}$ let denote $v_{\lambda}^{\lambda^{'}}$ the associated map. Let $\lambda\in\mathbb{R}$ and $h< \frac{R_{\mu}}{2}$. By Lemma \ref{satisfyC2b},  $$v_{\lambda}^{\lambda-L\left(\sqrt{d}(2+2/\mu^{2})\right)^{\alpha}h^{\alpha}}=\phi_{\lambda}\circ \Tilde{i}_{\lambda}$$
 with,
 $$\Tilde{i}_{\lambda}:H_{0}\left(\mathcal{F}_{\lambda}\right)\rightarrow H_{0}\left(\mathcal{F}_{\lambda}^{h}\right)\text{ induced by inclusion.}$$  
 By assumption \textbf{A1} and \textbf{A2}, $\mathcal{F}_{\lambda}$ is compact. As $[0,1]^{d}$ is triangulable,  $\mathcal{F}_{\lambda}$ is covered by finitely many cells of the triangulation, and so there is a finite simplicial complex $K$ such that $\overline{\mathcal{F}_{\lambda}}\subset K \subset \mathcal{F}_{\lambda}^{h}$. Consequently, $\Tilde{i}_{\lambda}$ factors through the finite dimensional space $H_{0}(K)$ and is then of finite rank.\\\\
 Thus, $v_{\lambda}^{\lambda-L\left(\sqrt{d}(2+2/\mu^{2})\right)^{\alpha}h^{\alpha}}$ is of finite rank for all $0<h<\frac{R_{\mu}}{2}$. As for any $\lambda>\lambda^{'}>\lambda^{''}$, $v_{\lambda}^{\lambda^{''}}=v_{\lambda^{'}}^{\lambda^{''}}\circ v_{\lambda}^{\lambda^{'}}$ we then have that $v_{\lambda}^{\lambda^{'}}$ is of finite rank for all $\lambda>\lambda^{'}$. Hence, $f$ is $q$-tame.
\end{proof}
\begin{prpstn}
\label{qtamenessEstim}
Let $f:[0,1]^{d}\rightarrow\mathbb{R}$ then, for all $s\in\mathbb{N}$ $\widehat{\mathbb{V}}_{0,f}$ is $q$-tame. 
\end{prpstn}
\begin{proof}
Let $h>0$ and $\lambda\in \mathbb{R}$. By construction $\widehat{\mathcal{F}}_{\lambda}$ is a union of cubes of the regular grid $G(h)$, thus $H_{0}(\widehat{\mathcal{F}}_{\lambda})$ is finite dimensional. Thus $\widehat{\mathbb{V}}_{0,f}$ is $q$-tame by Theorem 1.1 of \cite{Crawley2012}.
\end{proof}
\end{document}